\documentclass[12pt]{amsart}

\usepackage[margin = 1in]{geometry}
\usepackage{amsfonts, amsmath,amscd, amssymb, cite, color, latexsym, mathrsfs, mathtools,
slashed, stmaryrd, verbatim, wasysym }
\usepackage[all]{xy}
\usepackage[pdftex]{graphicx}
\usepackage{hyperref}
\usepackage[applemac]{inputenc}

\newtheorem{theorem}{Theorem}
\newtheorem{corollary}[theorem]{Corollary}
\newtheorem{definition}{Definition}

\newtheorem{lemma}[theorem]{Lemma}
\newtheorem{proposition}[theorem]{Proposition}
\newtheorem{remark}{Remark}
\newcommand{\De}{D_\Gamma}

\numberwithin{equation}{section}
\numberwithin{theorem}{section}

\let\oldsqrt\sqrt
\def\sqrt{\mathpalette\DHLhksqrt}
\def\DHLhksqrt#1#2{%
\setbox0=\hbox{$#1\oldsqrt{#2\,}$}\dimen0=\ht0
\advance\dimen0-0.2\ht0
\setbox2=\hbox{\vrule height\ht0 depth -\dimen0}%
{\box0\lower0.4pt\box2}}

\newcommand{\bhs}[1]{\mathfrak B_{#1}}
\renewcommand{\tilde}{\widetilde}
\renewcommand{\bar}{\overline}

\renewcommand{\hat}[1]{\widehat{#1}}

\newcommand{\wt}[1]{\widetilde{#1}}
\newcommand{\rest}[1]{\big\rvert_{#1}} % restriction e.g. to boundary

\newcommand\lra{\longrightarrow}
\newcommand\xlra[1]{\xrightarrow{\phantom{x} #1 \phantom{x}}}

\newcommand\pa{\partial}

\newcommand\eps\varepsilon
\renewcommand\epsilon\varepsilon

\newcommand\iie{\operatorname{iie}}

\newcommand\CI{{\mathcal{C}}^{\infty}}
\newcommand\CIc{{\mathcal{C}}^{\infty}_c}

\newcommand{\lrpar}[1]{\left( #1 \right)}
\newcommand{\lrspar}[1]{\left[ #1 \right]}
\newcommand\ang[1]{\langle #1 \rangle}

\newcommand{\norm}[1]{\lVert #1 \rVert}
\newcommand{\abs}[1]{\left\lvert #1 \right\rvert}

\newcommand\BQ{\operatorname{BQ}}

\newcommand\Che{\operatorname{Che}}

\newcommand\Id{\operatorname{Id}}

\newcommand\Ind{\operatorname{Ind}}

\newcommand\pt{\operatorname{pt}}

\newcommand\sign{\operatorname{sign}}

\newcommand\supp{\operatorname{supp}}

\newcommand\Mand{\text{ and }}

\newcommand\Min{\text{ in }}

\newcommand\Mon{\text{ on }}

\newcommand\Mwhere{\text{ where }}

\newcommand\paperintro%
        {%
         }
\newcommand\paperbody%
        {%
         }

\newcommand\bbB{\mathbb{B}}
\newcommand\bbC{\mathbb{C}}

\newcommand\bbK{\mathbb{K}}

\newcommand\bbN{\mathbb{N}}

\newcommand\bbQ{\mathbb{Q}}
\newcommand\bbR{\mathbb{R}}
\newcommand\bbS{\mathbb{S}}

\newcommand\bbZ{\mathbb{Z}}

\newcommand\cC{\mathcal{C}}
\newcommand\cD{\mathcal{D}}
\newcommand\cE{\mathcal{E}}
\newcommand\cF{\mathcal{F}}
\newcommand\cG{\mathcal{G}}

\newcommand\cL{\mathcal{L}}

\newcommand\cN{\mathcal{N}}

\newcommand\cS{\mathcal{S}}
\newcommand\cT{\mathcal{T}}
\newcommand\cU{\mathcal{U}}

\newcommand\cW{\mathcal{W}}

\newcommand\sB{\mathscr{B}}
\newcommand\sC{\mathscr{C}}
\newcommand\sD{\mathscr{D}}

\newcommand\sF{\mathscr{F}}
\newcommand\sG{\mathscr{G}}

\newcommand\sM{\mathscr{M}}

\newcommand\sW{\mathscr{W}}
\newcommand\sX{\mathscr{X}}
\newcommand\sY{\mathscr{Y}}
\newcommand\sZ{\mathscr{Z}}

\DeclareMathAlphabet{\mathpzc}{OT1}{pzc}{m}{it}

\newcommand{\sing}[1]{\widehat{  #1 }} %singular part
\newcommand{\sm}[1]{#1} %regular part
\newcommand{\res}[1]{\widetilde{ #1 }} %resolution
\newcommand{\cov}[1]{{ #1 }_{\Gamma}} %resolution

\newcommand{\transv}{\mathrel{\text{\tpitchfork}}}
\makeatletter
\newcommand{\tpitchfork}{%
  \vbox{
    \baselineskip\z@skip
    \lineskip-.52ex
    \lineskiplimit\maxdimen
    \m@th
    \ialign{##\crcr\hidewidth\smash{$-$}\hidewidth\crcr$\pitchfork$\crcr}
  }%
}
\makeatother

\usepackage{color}
\definecolor{darkgreen}{cmyk}{1,0,1,.2}
\definecolor{m}{rgb}{1,0.1,1}
\definecolor{b}{rgb}{0,0.1,1}
\newcommand{\striad}[1]{( {#1}; \pa_0 {#1}, \pa_1 {#1})} %smooth triad
\newcommand{\triad}[1]{(\sing{#1}; \pa_0 \sing{#1}, \pa_1 \sing{#1})} %singular triad
\newcommand{\tK}{\operatorname{K}}
\newcommand{\tL}{\operatorname{L}}
\newcommand{\tN}{\operatorname{N}}
\newcommand{\tS}{\operatorname{S}}
\newcommand{\id}{\operatorname{id}}
\newcommand{\op}{\operatorname{op}}

\begin{document}
\pagestyle{myheadings}
\markboth{Pierre Albin and Paolo Piazza}{Stratified surgery}

%\date{Last compiled \today; last edited  \heuteIst or later}

\title{Stratified surgery and  K-theory invariants of the signature operator}

\author{Pierre Albin and Paolo Piazza}
\maketitle

\begin{abstract}

In work of Higson-Roe the fundamental role of the signature as a homotopy and bordism invariant for oriented manifolds is made manifest in how it and related secondary invariants define a natural transformation between the (Browder-Novikov-Sullivan-Wall) surgery exact sequence and a long exact sequence of $C^*$-algebra K-theory groups.

In recent years the (higher) signature invariants have been extended from closed oriented manifolds to a class of stratified spaces known as L-spaces or Cheeger spaces. In this paper we show that secondary invariants, such as the $\rho$-class, also extend from closed manifolds to Cheeger spaces. We revisit a surgery exact sequence for stratified spaces originally introduced by Browder-Quinn and obtain a natural transformation analogous to that of Higson-Roe. We also discuss geometric applications.

%Let $X$ be  a smoothly stratified pseudomanifold. We make the assumption that $X$ is a Witt space or,
%more generally, a Cheeger space.
%We map the Browder-Quinn surgery sequence associated to $X$ to a K-theory sequence of $C^*$-algebras. In particular,
%we define the rho-class of an element in the Browder-Quinn structure set $\tS_{{\rm BQ}} (X)$.
\end{abstract}

\section{Introduction}

The discovery by Milnor of smooth manifolds that are homemorphic to $\bbS^7$ but not diffeomorphic to it,  a milestone of modern mathematics, gave rise to the development of methods for classifying smooth manifolds within a given homotopy class. (The undecidability of the word problem makes an unrestricted classification impossible.) A convenient formulation of these methods is the surgery exact sequence of Browder, Novikov, Sullivan, and Wall which, roughly speaking, relates the `structure set' $\tS(X)$ of homotopy equivalences to a given smooth oriented manifold $X$ (with a bordism equivalence relation) with the set $\tN(X)$ of degree one maps preserving normal bundle information, known as  `normal invariants', (also with a bordism equivalence relation) and an algebraically defined $L$-group depending only on $\Gamma=\pi_1X,$ the fundamental group of $X,$
\begin{equation}\label{surgery-seq}
	\xymatrix{
	\ldots \ar[r]&
	\tL_{m+1}(\bbZ\Gamma) \ar@{..>}[r] &
	\tS(X) \ar[r]&
	\tN(X) \ar[r]&
	\tL_m(\bbZ\Gamma)}.
\end{equation}
(See below and, e.g., 
\cite{Wall, ranicki-ags, lueck-survey,Crowley-Luck-Macko} for more on the surgery exact sequence.)

In a series of papers Higson and Roe \cite{higson-roeI, higson-roeII,higson-roeIII}
established the remarkable result that there are natural maps out of the surgery  sequence \eqref{surgery-seq}, into a long exact sequence of K-theory groups of certain $C^*$-algebras  and that these maps make the resulting
diagram commute.
 The $C^*$-algebras in question are $C^*(X_{\Gamma})^{\Gamma}$ and $D^*(X_{\Gamma})^{\Gamma}$, obtained as the closures of the $\Gamma$-equivariant operators on the universal cover $X_{\Gamma}$ of $X$ 
 that satisfy a  finite propagation property and, in addition, are respectively  `locally compact' or  `pseudolocal'.
 The former $C^*$-algebra is an ideal in the latter so we have a short exact sequence
\begin{equation*}
	0\rightarrow  C^*(X_\Gamma)^\Gamma\rightarrow 
	D^*(X_\Gamma)^\Gamma\rightarrow 
	D^*(X_\Gamma)^\Gamma/C^*(X_\Gamma)^\Gamma\rightarrow 0
\end{equation*}
which gives rise to a long exact sequence in K-theory known as the {\em analytic surgery sequence} of Higson and Roe. Making use of the canonical isomorphisms
\begin{equation*}
	K_{*+1}  (D^*(X_\Gamma)^\Gamma/C^*(X_\Gamma)^\Gamma)= K_{*} (X)
	\quad\text{and}\quad  
	K_{*}  (C^*(X_\Gamma)^\Gamma)=K_* (C^*_r \Gamma)
\end{equation*}
with the $K$-homology of $X$ and the K-theory of the reduced $C^*$-algebra of $\Gamma,$ the long exact sequence reads
\begin{equation}\label{eq:HR}
	\cdots\rightarrow K_{m+1}  (C^*_r\Gamma)
	\rightarrow K_{m+1}  (D^*(X_\Gamma)^\Gamma)
	\rightarrow  K_{m}(X)
	\rightarrow K_m (C^*_r\Gamma)
	\rightarrow \cdots
\end{equation}
The result of Higson and Roe is thus a commutative diagram of long exact sequences
\begin{equation}\label{HRdiagram-universal}
	\xymatrix{
	\tL_{m+1}(\bbZ\Gamma) \ar@{..>}[r]\ar[d]^-{\gamma} &
	\tS(X) \ar[r]\ar[d]^-{\rho} &
	\tN(X) \ar[r]\ar[d]^-{\beta} &
	\tL_m(\bbZ\Gamma) \ar[d]^-{\gamma} \\
	K_{m+1}  (C^*_r\Gamma)[\tfrac12] \ar[r] &
	K_{m+1}  (D^*(X_\Gamma)^\Gamma)[\tfrac12] \ar[r] &
	K_{m}(X)[\tfrac12] \ar[r] &
	K_m (C^*_r\Gamma)[\tfrac12] }
\end{equation}
where we use the short-hand $A[\tfrac12]$ to indicate $A \otimes_{\bbZ} \bbZ[\tfrac12]$ whenever $A$ is an Abelian group. These maps were  recast by the second author and Schick \cite{PS-akt} in a more index-theoretic light, using
in a crucial way properties of the signature operator on Galois $\Gamma$-coverings. 
Geometric applications of the interplay between the two sequences have been given, for example, in 
\cite{Chang-Weinberger}, \cite{Wahl_higher_rho}, \cite{Zenobi:Mapping}, \cite{WY-finite}, \cite{Weinberger-Xie-Yu}.

In this paper we generalize all of this to the setting of stratified spaces.

\begin{theorem}
Every $m$-dimensional, oriented, smoothly stratified Cheeger space, $\sing X,$ with fundamental group $\Gamma$, gives rise to a commutative diagram
\begin{equation}\label{eq:IntThm}
	\xymatrix{
	\tL_{\BQ}(\sing X \times I) \ar@{..>}[r] \ar[d]^-{\Ind_{{\rm APS}}} & 
	\tS_{\BQ}(\sing X) \ar[r]  \ar[d]^-{\rho} &
	\tN_{\BQ}(\sing X) \ar[r] \ar[d]^-{\beta} &
	\tL_{\BQ}(\sing X) \ar[d]^-{\Ind_{{\rm APS}}}\\
	K_{m+1}(C^*_r\Gamma)[\tfrac12] \ar[r] &
	K_{m+1}(D^* (\cov{\sing X})^\Gamma)[\tfrac12] \ar[r] &
	K_{m}(\sing X)[\tfrac12] \ar[r] &
	K_{m}(C^*_r\Gamma)[\tfrac12] }
\end{equation}
between the Browder-Quinn surgery exact sequence and the Higson-Roe analytic surgery sequence.
\end{theorem}

We prove a better version of this diagram in \S\ref{sect:all-strata} involving the signature operator on all of the strata of $\sing X,$ but refer the reader to the text so as not to introduce more notation.\\

The top row of \eqref{eq:IntThm} is the surgery sequence for stratified spaces of Browder-Quinn \cite{BQ}. One of our contributions in this paper is a detailed treatment in \S\ref{sect:BQ} of the Browder-Quinn surgery exact sequence in the setting of smoothly stratified spaces (a.k.a. Thom-Mather stratified spaces). The original treatment in \cite{BQ} is quite sparse and its generalization in \cite{W-book} uses algebraic tools applicable in its setting of homotopically stratified spaces while for our purposes it is necessary to have geometric proofs that stay within the category of smooth stratifications. Our treatment naturally draws heavily from these two sources.

A feature of the Browder-Quinn surgery sequence is that if $\sing X$ is a Witt space \cite{Siegel:Witt} or a Cheeger space \cite{banagl-annals, ABLMP, ALMP:Hodge} then all of the spaces that arise in the surgery sequence are also Witt spaces, respectively Cheeger spaces. This allows us to bring to bear the analysis that we have developed in joint work with Eric Leichtnam and Rafe Mazzeo \cite{ALMP:Witt, ALMP:Hodge, ALMP:Novikov} to define the vertical maps in \eqref{eq:IntThm} and to show that the diagram commutes. Notice that while the vertical maps are defined in analogy with \cite{PS-akt}, there are
substantial technical differences, especially in the Cheeger case, where ideal boundary conditions must be chosen.

In detail, the vertical maps out of the Browder-Quinn L-groups are Atiyah-Patodi-Singer index classes;
the map out of the normal invariants $\tN_{\BQ}(\sing X)$ is given in terms of the fundamental class, in K-homology, associated
to the signature operator on a Cheeger space; finally the rho-map is a true secondary invariant associated to 
a suitable perturbation of the signature operator. In the smooth case 
this rho-map is directly connected with well-known numeric rho invariants; we comment on the validity 
of this principle in the singular case at the end of the paper.
As already remarked, all of these constructions 
 depend upon the definition of ideal boundary conditions; these depend, in turn, on the choice of a {\it mezzoperversity}
 and a major theme in this article is the detailed analysis of the dependence of these classes on the choice of a mezzoperversity
 and the proof of the remarkable fact that our maps are in fact all independent of the choice of a mezzoperversity.
 
 We end this introduction with a remark on the spin-Dirac operator. One of the advantages of the 
approach to the Higson-Roe theorem given in \cite{PS-akt} is that it applies also to the spin-Dirac operator,
giving in particular the mapping of the Stolz surgery sequence for positive scalar metrics to the 
analytic surgery sequence of Higson-Roe, see \cite{PS-Stolz}. In a future publication, building also on \cite{a-gr-spin, Piazza-Vertman, a-gr-dirac},
we plan to discuss the extension of the results in \cite{PS-Stolz} to the stratified setting. \\

 The paper is organized as follows. In Section \ref{sect:BQ} we give a rigorous and detailed treatment
 of the relevant results stated in the paper by Browder and Quinn. In Section \ref{sect:classes}
 we specialize to Cheeger spaces and give a coarse theoretic treatment of some of the
 results in  \cite{ALMP:Witt, ALMP:Hodge, ALMP:Novikov}; in particular we define the fundamental K-homology class 
 of a Cheeger space without boundary and  the associated index class. Finally, in the invertible case, 
 we introduce the rho class of an invertible perturbation of the signature operator.
In Section \ref{sect:bordism} we pass to manifolds with boundary, with a particular emphasis on the notion
of Cheeger space bordism. It is in this section that we explain the statement of the delocalized Atiyah-Patodi-Singer index
theorem on Cheeger spaces, a key tool in our analysis, and we illustrate its proof, building on \cite{PS-Stolz, PS-akt}.
In Section \ref{sect:HS}
we recall and expand results around the Hilsum-Skandalis perturbation for the signature operator
on the disjoint union of two Cheeger spaces that are stratified homotopy equivalent. In Section \ref{sect:mapping}
 we finally define the vertical maps in the diagram that maps the Browder-Quinn
surgery sequence to the Higson-Roe surgery sequence; we prove the well-definedness of these maps and that
they are independent of the choice of a mezzoperversity. We then prove the commutativity
of the squares of the diagram. We end this section by observing that it is in fact possible to consider
different diagrams, each one  associated to an individual closed stratum. 
Section \ref{sect:applications}, the last section of the paper, presents some geometric applications
of our main result, in the spirit of \cite{Chang-Weinberger}.

\bigskip
\noindent
{\bf Acknowledgements.}
The first author was partially supported by  Simons Foundation grant \#317883 and NSF grant DMS-1711325, and is grateful to Sapienza Universit\`a di Roma, to Universit\'e  Paris 7 (\'Equipe Alg\`ebres d'Op\'erateurs),  and to M.I.T. for their hospitality during the completion of this work.
Financial support for the visits to Sapienza was provided by  {\it Istituto Nazionale di Alta Matematica (INDAM)}.\\
The second author thanks Universit\'e Paris 7  (\'Equipe Alg\`ebres d'Op\'erateurs)  
for hospitality and financial support during the completion of this
work; additional funding for these visits was provided by {\it Istituto Nazionale di Alta Matematica} and {\it Ministero dell'Istruzione,
dell' Universit\`a e della Ricerca}, {\it MIUR}, under the project
{\it Spazi di Moduli e Teoria di Lie}.\\
We are grateful to Markus Banagl, Jim Davis, Wolfgang L\"uck, Thomas Schick, Georges Skandalis, Shmuel  Weinberger and Vito Felice Zenobi
for their interest in this work and for interesting and useful discussions.\\

%%%%%%%%%%%%%%
\tableofcontents
%%%%%%%%%%%%%%

\subsection*{Notation}

Below we will occasionally use diagrams such as
\begin{equation*}
	(\sing M; \pa_1\sing M, \pa_2\sing M) \xlra f 
	(\sing Y; \pa_1\sing Y, \pa_2 \sing Y).
\end{equation*}
This should be understood to imply that 
$\pa \sing M = \pa_1 \sing M \cup \pa_2\sing M,$ 
$\pa \sing Y = \pa_1 \sing Y \cup \pa_2\sing Y,$ also 
that $\pa_1 \sing M \cap \pa_2 \sing M = \pa(\pa_i \sing M)$ (and similarly for $\sing Y$),
and that $f$ restricts to maps
\begin{equation*}
	f|: \pa_1 \sing M \lra \pa_1 \sing Y, \quad
	f|: \pa_2 \sing M \lra \pa_2 \sing Y,
\end{equation*}
which we sometimes denote $\pa_1f,$ $\pa_2f,$ respectively.
Occasionally it will be useful to decompose the boundary of a space into more than two pieces, in which case similar conventions are in effect.

Our main object of study will be smoothly stratified spaces. As reviewed below, this will mean Thom-Mather straified pseudomanifolds. A bordism between two stratified spaces will be a stratified space with boundary and a bordism between two stratified spaces with boundary will be a stratified space with corners (as is well known, e.g.,  \cite[\S8.3]{Wall:DiffTop}, this is only useful if restrictions are placed on part of the boundary). For a careful discussion of these concepts we refer the reader to \cite{Verona} (see also \cite[\S6]{Albin:Hodge}). While we do not use the language of `n-ads' as in \cite{Wall}, it is clear that the constructions below extend to `n-ads of stratifed spaces'.

Note that the boundary of a manifold with corners is not itself a manifold with corners, but rather a union of manifolds with corners with various identifications of boundary faces. Following Melrose (see, e.g., \cite{MelroseSlides}), an `articulated manifold (without boundary)' is a finite union of connected components of the boundary of a compact manifold with corners (thus guaranteeing that the identifications of boundary faces are consistent).
More generally, an `articulated manifold with corners' is a finite union of boundary hypersurfaces of the boundary of a compact manifold with corners (not necessarily making up full connected components of the boundary).

Similarly an `articulated stratified space (without boundary)' refers to a finite union of connected components of the boundary of a compact stratified space with corners, and an {`articulated stratified space with corners' is a finite union of boundary hypersurfaces of a stratified space with corners. 
 
Working with `articulated stratified spaces' is analogous to working with n-ads in the category of stratified spaces; for instance, if $(\sing M; \pa_1\sing M, \pa_2\sing M)$ is as above and $\sing M$ is a stratified space with corners then $\pa_i\sing M$ are articulated stratified spaces with corners.
An alternate approach, see, e.g., \cite[\S2.6]{Wall:DiffTop}, is to `round the corners' and work with spaces with boundary.

%%%%%%%%%%%%%%
\section{Browder-Quinn Surgery}\label{sect:BQ}
%%%%%%%%%%%%%%

We will make use of the surgery theory for stratified spaces of Browder and Quinn \cite{BQ}.
Some of the results we need for the purpose of defining maps into K-theory are implicit in their exposition so we have decided to include a more explicit description of this surgery theory.
In carrying this out we have benefitted from Weinberger's exposition \cite[\S 7.1]{W-book} where some of the proofs below are sketched (e.g., the $\Pi$-$\Pi$ theorem, Theorem \ref{thm:pipi}), as well as from
\cite{Wall, Ranicki, Davis-Petrosyan, Crowley-Luck-Macko, Anderson, Quinn} and \cite{Dovermann-Rothenberg} (from which we adapted the proof of the Wall realization theorem, Theorem \ref{thm:WallReal}). We are also happy to acknowledge useful conversations and email exchanges with Shmuel Weinberger, Markus Banagl, Wolfgang L\"uck, and Jim Davis.

\subsection{Browder-Quinn stratified spaces and transverse maps}\label{sec:BQStrat}$\,$\\
Although there are many notions of stratified spaces, perhaps the most common is that of a Whitney stratified space. If $L$ is a smooth manifold then a Whitney stratification of a subset $\sing X \subseteq L$ is a locally finite collection of pairwise disjoint smooth submanifolds covering $\sing X,$ known as strata, satisfying the `frontier condition'
\begin{equation*}
	Y \cap \bar{Y'} \neq \emptyset \implies Y \subseteq \bar{Y'}
\end{equation*}
and `Whitney's condition (B)' concerning the relations of the tangent spaces of the strata. For example, Whitney showed \cite{Whitney:Tangents} that algebraic varieties admit Whitney stratifications.
It was subsequently shown by Thom and Mather \cite{Mather} that in a Whitney stratified space neighborhoods of the strata have geometric structure and this was abstracted in the notion of Thom-Mather stratified space.

A further abstraction was given by Browder and Quinn \cite{BQ} (cf. \cite{Hughes-Weinberger, Dovermann-Schultz, W-book}).
They fix a category $\sF$ of `manifolds with fibrations' such as smooth manifolds and locally trivial smooth fiber bundles, PL manifolds and block bundles with manifold fibers, topological manifolds and locally trivial topological fiber bundles, or Poincar\'e spaces and maps whose homotopy fiber satisfies Poincar\'e duality. Although Browder-Quinn do not specify what properties are necessary in the category $\sF,$ an important property is that there be a notion of pull-back in the category $\sF.$

An $\sF$-stratified space is a topological space $\sing X$ filtered by subsets $X_a$ indexed by a partially ordered set $A$ satisfying the following.
If for each $a\in A$ we let 
\begin{equation*}
	X_{\pa a} = \bigcup \{ X_b : b \in A, b <a \}
\end{equation*}
then each $X_a$ is equipped with a closed neighborhood $N_a =N(X_a)$ of $X_{\pa a}$ in $X_a$ and a projection $\nu_a: \pa N_a \lra X_{\pa a}$ such that\\
i) $X_a \setminus X_{\pa a}$ and $\pa N_a$ are manifolds in $\sF,$\\
ii) $N_a$ is the mapping cylinder of $\nu_a$ (with $\pa N_a$ and $X_{\pa a}$ corresponding to the top and bottom of the cylinder),\\
iii) If $a, b \in A,$ $b<a,$ $W_b = X_b \setminus \mathrm{int}(N_b),$ then 
\begin{equation*}
	\nu_a\rest{} : \nu_a^{-1}(W_b) \lra W_b
\end{equation*}
is a fibration in $\sF.$\\

If $\sing X$ and $\sing M$ are two $\sF$-stratified spaces whose filtrations are indexed by the same partially ordered set $A,$ then a filtration-preserving map $f: \sing X \lra \sing M$ of $\sF$-stratified sets is said to be {\bf transverse} if each fibration in $\sing X$ is the pull-back along $f$ of the corresponding fibration in $\sing M.$

\begin{remark}
When $\cF$ is equal to the category of smooth manifolds and locally trivial smooth fiber bundles, Browder-Quinn $\cF$-stratified spaces are the same as Thom-Mather stratified spaces. One could show this by, for example, proceeding as in \cite{ALMP:Witt} and proving that any Browder-Quinn $\sF$-stratified space can be `resolved' to an $\sF$-space `with corners' and iterated fibration structures, and any such can be collapsed to a Browder-Quinn $\sF$-stratified space. Instead of developing this, we will work directly with Thom-Mather stratified spaces in establishing the Browder-Quinn surgery sequence below.
\end{remark}

\subsection{Smoothly stratified spaces and transverse maps} \label{sec:TransMaps}$\,$\\
From now on we will only work with $\sF$ equal to smooth manifolds and locally trivial smooth fiber bundles, i.e., the setting of Thom-Mather stratified spaces. For this class of spaces there is a construction going back to Thom \cite{Thom:Ensembles} and carried out in \cite{ALMP:Witt} that replaces a stratified space, $\sing X,$ with its `resolution', $\res X,$ a manifold with corners and an iterated fibration structure. We now recall this construction.\\

Let $\sing X$ be a stratified space with singular strata 
\begin{equation*}
	\cS(\sing X) = \{ \sm Y_1, \sm Y_2, \ldots, \sm Y_\ell \}.
\end{equation*}
Each $\sm Y_i$ is subset of $\sing X$ that inherits the structure of an open manifold (indeed, of the interior of a manifold with corners). We write $\sm Y_i < \sm Y_j$ if the closure of $Y_j$ in $\sing X$ contains $Y_i.$ The closure of $Y_j$ in $\sing X$ is given by
\begin{equation*}
	\sing Y_j = \bigcup \{ \sm Y_i : \sm Y_i \leq \sm Y_j \}
\end{equation*}
and is itself a stratified space. Every point in $\sm Y_i$ has a neighborhood in $\sing X$ homeomorphic to a ball in $\bbR^{\dim \sm Y_i}$ times the cone over a stratified space $\sing Z_i,$ known as the link of $\sm Y_i$ in $\sing X.$

The resolution of $\sing X,$ denoted $\res X,$ is a smooth manifold with corners. Each stratum $\sm Y_i$ of $\sing X$ corresponds to a collective boundary hypersurface $\bhs{Y_i}$ of $\sing X,$ by which we mean a collection of boundary hypersurfaces no two of which intersect. Each collective boundary hypersurface participates in a fiber bundle,
\begin{equation*}
	\res Z_i - \bhs{Y_i} \xlra{\phi_{Y_i}} \res Y_i,
\end{equation*}
where the base is the resolution of $\sing Y_i$ and the typical fiber is the resolution of $\sing Z_i.$ If $\sm Y_i$ and $\sm Y_j$ are strata of $\sing X$ with $\sm Y_i < \sm Y_j$ then $\bhs{Y_i} \cap \bhs{Y_j} \neq \emptyset$ and we have a commutative diagram of fiber bundle maps
\begin{equation*}
	\xymatrix 
	{ \bhs{Y_i} \cap \bhs{Y_j} \ar[rr]^-{\phi_{Y_j}} \ar[rd]_-{\phi_{Y_i}} 
		& & \bhs{Y_iY_j} \subseteq \bhs{Y_j} \ar[ld]^-{\phi_{Y_iY_j}} \\
	& \res Y_i & }
\end{equation*}
where $\bhs{Y_iY_j}$ is a collective boundary hypersurface of $\bhs{Y_j}.$
We refer to a manifold with corners together with these collective boundary hypersurface fiber bundle maps as a manifold with an iterated fibration structure.

There is a canonical `blow-down map' between a manifold with corners and an iterated fibration structure $\res X$ and a stratified space $\sing X,$
\begin{equation*}
	\beta: \res X \lra \sing X
\end{equation*}
which collapses the fibers of the boundary fiber bundles to their base. Note that $\beta$ is a diffeomorphism between the interior of $\res X$ and the regular part of $\sing X.$

A continuous map between stratified spaces is {\bf stratum preserving} if the inverse image of a stratum is a union of strata.
A stratum preserving map $\sing M \xlra{\sing F} \sing X$ is {\bf smooth} if it lifts to a smooth map $\res M \xlra{\res{F}} \res X.$ 
We denote the space of such maps by
\begin{equation*}
	\CI_{\Phi}(\res M, \res X) \subseteq \CI(\res M, \res X)
\end{equation*}
and the corresponding maps between $\sing M$ and $\sing X$ by $\CI(\sing M, \sing X).$ Note that we have a natural identification
\begin{equation*}
	\beta_*:\CI_{\Phi}(\res M, \res X) \lra \CI(\sing M, \sing X)
\end{equation*}

The smooth map $\res F$ necessarily induces a fiber bundle map between the collective boundary hypersurfaces of $\res M$ and those of $\res X,$
\begin{equation}\label{def:FibBdleMap}
	\xymatrix{
	\res M \ar@{}[r]|-*[@]{\supseteq}  & \bhs{N} \ar[r]^-{\res F} \ar[d]_-{\phi_N} 
		& \bhs{Y}  \ar[d]^-{\phi_Y}  \ar@{}[r]|-*[@]{\subseteq} &  \res X \\
	& \res N \ar[r]^-{\res{F|_{\sing N}}} & \res Y }
\end{equation}
We will say that $\sing M \xlra{\sing F} \sing X$ and $\res M \xlra{\res F} \res X$ are {\bf transverse} if the commutative diagram of fiber bundles \eqref{def:FibBdleMap} is a pull-back diagram. We denote the class of such maps by
$\CI_{\transv}(\res M, \res X)$ and $\CI_{\transv}(\sing M, \sing X);$ the identification $\beta_*$ restricts to an identification
\begin{equation*}
	\beta_*:\CI_{\transv}(\res M, \res X) \lra \CI_{\transv}(\sing M, \sing X).
\end{equation*}
Note that $\CI_{\transv}(\sing M, \sing X)$ are the transverse maps of Browder-Quinn. (Indeed, the fibrations $\nu_a$ of a Browder-Quinn stratified space correspond in the smooth category to the fiber bundle maps $\phi_Y.$)
These maps are also used in \cite[Part I, \S4]{Fulton-Mac} and \cite[\S5.4]{IH2} where they are called `normal non-singular'. A weaker notion called `homotopy transverse' was used by Weinberger \cite[\S 5.2]{W-book}.

An example from \cite{IH2} is the inclusion $H\cap \sing X \hookrightarrow \sing X$ when $\sing X$ is a stratified subset of a smooth manifold and $H$ is a smooth submanifold transverse to the strata of $\sing X.$
Another example from the same source is the fiber bundle projection map for a fiber bundle over a stratified space with fiber a smooth manifold.\\

\begin{remark}
Browder-Quinn considered these maps first in an equivariant situation. If $G$ is a compact Lie group and $L,$ $L'$ are spaces with $G$-actions then a map $f: L \lra L'$ is {\em isovariant} if, for any $x \in L,$ $g \in G,$ 
\begin{equation*}
	f(gx) = gf(x), \; \Mand \; gf(x)=f(x) \iff gx=x.
\end{equation*}
An isovariant map is {\em transverse linear} if whenever $H \subseteq G$ is a subgroup, and $L^{\{H\}}$ denotes the subset of $L$ consisting of points whose isotropy group is conjugate to $H,$ there are $G$-vector bundle tubular neighborhoods 
\begin{equation*}
	L^{\{H\}} \subseteq U, \quad 
	(L')^{\{H\}} \subseteq U'
\end{equation*}
such that $f$ restricts to a $G$-linear vector bundle map $U \lra U'.$
Transverse linear isovariant maps are examples of transverse maps for the stratification of a space into the orbit types of a group action.
\end{remark}

There are some properties of a stratified space $\sing X$ such that the existence of a transverse map $\sing f: \sing M \lra \sing X$ implies that $\sing M$ also has this property. For example, if the dimensions of all of the links of $\sing X$ are odd, then the dimensions of all of the links of $\sing M$ must be odd as well. A class of stratified spaces determined by such a property will be said to be {\bf preserved by transverse maps}. As examples of such classes let us mention, in order of increasing generality: IP spaces, Witt spaces, and Cheeger spaces or L-spaces. Recall that a stratified space is a Witt space if, whenever the link of a stratum is even dimensional, its middle degree middle perversity intersection homology vanishes. A Witt space is an IP space (intersection Poincar\'e space) if, whenever the link $\hat Z$ of a stratum is odd dimensional, its middle perversity homology in degree $\tfrac12(\dim \hat Z-1)$ is torsion-free. A stratified space is an L-space if there is a self-dual sheaf compatible with the intersection homology sheaves of upper and lower middle perversity \cite{banagl-annals}. Smoothly stratified L-spaces are known as Cheeger spaces. 

\subsection{Surgery definitions}$\,$\\
When carrying out surgery constructions we will need to use stratified spaces with corners. 
By a stratified space with corners we mean a stratified space with collared corners \cite[Definition 2]{Albin:Hodge} or an `abstract stratification with faces' in the sense of Verona \cite[\S 5]{Verona}.
An `articulated stratified space with corners' is a union of boundary hypersurfaces of a stratified space with corners; thus it is  a union of stratified spaces with corners together with identifications of certain of their boundary hypersurfaces.
% To ensure consistency and `smoothness' of these identifications we require that there is a stratified space with corners
Smoothness of maps to and from articulated spaces is defined in the natural way, i.e., continuity on the whole and smoothness on each stratified space with corners.

\begin{definition}
Let $\sing M$ and $\sing X$ be oriented stratified spaces.\\
i) A {\bf BQ-transverse map}  $f: \sing M \lra \sing X$ is a transverse map that is orientation preserving and restricts to a diffeomorphism between strata of dimension less than five.\\
ii) A {\bf BQ-normal map} $f: \sing M \lra \sing X$ is a BQ-transverse map such that, in the notation of \S\ref{sec:BQStrat}, for each $a \in A,$ $f$ restricts to a degree one normal map
\begin{equation*}
	f_a: M_a \setminus \mathrm{int}(N(M_a)) \lra X_a \setminus \mathrm{int}(N(X_a))
\end{equation*}
meaning that there is an smooth vector bundle $\tau \lra X_a\setminus\mathrm{int}(N(X_a))$ and a bundle isomorphism $b: \nu_{M_a \setminus \mathrm{int}(N(M_a))} \lra f_a^*\tau$ covering $f_a.$ (We will not explicitly keep track of the bundle data as it will not affect our analytic maps.)\\
iii) A {\bf BQ-equivalence} $f: \sing M \lra \sing X$ is a BQ-transverse map whose restriction to each stratum is a homotopy equivalence. 
(By Miller's criterion \cite{Miller} (see \cite[Corollary 1.11]{ALMP:Novikov}), $f$ is a BQ-equivalence if and only if there is a BQ-transverse map $g: \sing X \lra \sing M$ and homotopies of $f\circ g$ and $g \circ f$ to the respective identities through BQ-transverse maps.)
\end{definition}

One could also work with simple homotopy equivalences  and there is an s-cobordism theorem in this context \cite[pg. 34]{BQ}.

Given a stratified space, its dimension function will refer to the function on the poset of strata that assigns to each stratum its dimension.

\begin{definition}\label{def:SurgeryGps}
Let $\sing X$ be a smooth oriented stratified space, possibly with boundary.
(Our convention is that the spaces below are allowed to be empty and a map between empty sets counts as an equivalence.)\\
a) 
Let $\cL_{\BQ,d}(\sing X)$ denote the set of diagrams
\begin{equation}\label{eq:BQL}
	(\sing M; \pa \sing M) \xlra{\phi}
	(\sing Y; \pa \sing Y) \xlra{\omega}
	\sing X
\end{equation}
where $\sing M$ and $\sing Y$ are oriented stratified spaces with corners, $d$ is the dimension function of $\sing M$ (if omitted, $\sing M$ and $\sing X$ are assumed to have the same dimension function), $\phi$ is BQ-normal, $\pa\phi$ is a BQ-equivalence (between articulated stratified spaces with corners), and $\omega$ is BQ-transverse. We refer to these diagrams as (Browder-Quinn) {\bf $\cL$-cycles over $\sing X.$}\\
A {\bf null bordism of an $\cL$-cycle over $\sing X$} as above will mean a diagram
\begin{equation*}
	(\sing N; \pa_1 \sing N, \pa_2 \sing N) \xlra{\Phi}
	(\sing Z; \pa_1 \sing Z, \pa_2 \sing Z) \xlra{\Omega}
	\sing X \times I
\end{equation*}
between stratified spaces with corners, with the same dimension function as $\sing M \times I,$ where $\Phi$ is BQ-normal, $\Phi|:\pa_2\sing N \lra \pa_2\sing Z$ is a BQ-equivalence, $\Omega$ is BQ-transverse, and
\begin{equation*}
	\lrpar{
	(\pa_1\sing N, \pa_{12}\sing N) \xlra{\Phi|} 
	(\pa_1 \sing Z, \pa_{12}\sing Z) \xlra{\pi\circ\Omega|} 
	\sing X}
	=
	\lrpar{
	(\sing M; \pa \sing M) \xlra{\phi}
	(\sing Y; \pa \sing Y) \xlra{\omega}
	\sing X}
\end{equation*}
(with $\pi:\sing X \times I \lra \sing X$ the projection). In this case we say that the $\cL$-cycle is null bordant.\\
The $\cL$-cycles over $\sing X$ naturally form an Abelian monoid with addition induced by disjoint union and zero given by the diagram with $\sing M = \sing Y = \emptyset.$ 
We say that two $\cL$-cycles over $\sing X,$ $\alpha,$ $\beta$ are equivalent if $\alpha + \beta^{\op}$ is null bordant, where $\beta^{\op}$ denotes $\beta$ with orientations reversed.
This is an equivalence relation, known as $\cL$-bordism, and the set of equivalence classes, denoted $\tL_{\BQ,d}(\sing X),$ is known as the {\bf Browder-Quinn L-group of $\sing X.$}\\
b) The (Browder-Quinn) {\bf normal invariants}, denoted $\cN_{\BQ}(\sing X)$ is the subset of $\cL_{\BQ}(\sing X)$ in which, with notation as above, $\sing Y = \sing X$ and $\omega =\id.$ A {\bf normal bordism} is an $\cL$-bordism between normal invariants which, in the notation above, has $\sing Z = \sing X \times I$ and $\Omega = \id.$ The set of normal invariants modulo normal bordisms is denoted $\tN_{\BQ}(\sing X).$\\

If $\sing X$ has boundary
we denote by $\cN_{\BQ}(\sing X, \pa \sing X)$ the subset of $\cN_{\BQ}(\sing X)$ in which, with notation as above,  $\phi|_{\pa \sing M}$ is a (stratum preserving) diffeomorphism. A normal bordism, relative to $\pa \sing X,$ will be a normal bordism as above in which $\pa_2 \sing N = \sing K \times I$ and $\Phi|_{\pa_2\sing N} = \phi|_{\pa \sing M} \times \id.$ The corresponding set of normal invariants is denoted $\tN_{\BQ}(\sing X, \pa\sing X).$\\

c) The (Browder-Quinn) {\bf Thom-Mather structures}, denoted $\cS_{\BQ}(\sing X)$ is the subset of $\cN_{\BQ}(\sing X)$ in which, with notation as above, $\phi$ is a BQ-equivalence. Two such objects are equivalent if there is a normal bordism between them which, in the notation above, has $\Phi$ a BQ-equivalence. The set of equivalence classes is denoted $\tS_{\BQ}(\sing X)$ and is known as the (Browder-Quinn) {\bf  structure set.}\\

Similarly, if $\sing X$ has boundary then $\cS_{\BQ}(\sing X, \pa \sing X)$ is the subset of $\cN(\sing X, \pa\sing X)$ in which $\phi$ is a BQ-equivalence. Equivalence classes modulo normal bordism relative to $\pa \sing X$ in which $\Phi$ is a BQ-equivalence form the set $\tS_{\BQ}(\sing X, \pa\sing X).$

\end{definition}

\begin{remark} \label{rem:PiPiSuff}
A BQ-transverse map $f:\sing X \lra \sing W$ induces a homomorphism $f_*:\cL_{\BQ,d}(\sing X) \lra \cL_{\BQ,d}(\sing W).$
It is easy to see that if $f$ is a BQ-equivalence then it induces an isomorphism $\tL_{\BQ,d}(\sing X) \cong \tL_{\BQ,d}(\sing W).$ Browder and Quinn point out \cite[Proposition 4.7]{BQ} that this is true for any BQ-transverse $f$ that satisfies the $\Pi$-$\Pi$ condition below. For example, $\tL_{\BQ, d}(\sing X) \cong \tL_{\BQ,d}(\sing X \times I).$
\end{remark}

\begin{remark}
The Pontrjagin-Thom constructions can be fit together using transversality to identify $\cN_{\BQ}(\sing X)$ with homotopy classes of maps from $\sing X$ into $G/O$ that are constant on strata of dimension less than five. Since we do not use this we do not elaborate, see \cite[pg. 140]{W-book}.
\end{remark}

\begin{remark} \label{rem:Classical}
If $\sing X$ is a smooth manifold, then $\tL_{BQ,d}(X) = \tL_d(\pi_1(X)).$
Indeed, our definition of $\tL_{BQ,d}(X)$ coincides with Wall's $L_d^1(X)$ from \cite[Chapter 9]{Wall} save that we required $Y$ in Definition \ref{def:SurgeryGps} to be smooth. However $X$ smooth and Wall's realization theorem (\cite[Theorem 10.4]{Wall} and Theorem \ref{thm:WallReal} below) shows that this does not change the group we obtain.
\end{remark}

The inclusion maps between the sets above descend to maps between the equivalence classes
\begin{equation*}
	\tS_{\BQ}(\sing X) \xlra{\eta} \tN_{\BQ}(\sing X) \xlra{\theta} \tL_{\BQ}(\sing X),
\end{equation*}
We will show that this sequence is exact %(Theorem \ref{{thm:exact1}}) 
and extends to the left
\begin{equation*}
	\xymatrix{
	\tL_{\BQ}(\sing X \times I) \ar@{..>}[r] & \tS_{\BQ}(\sing X)}
\end{equation*}
in that there is an action of the L-group of $\sing X \times I$ on the structure set of $\sing X.$
Moreover the extended sequence is exact in that two elements of $\tS_{\BQ}(\sing X)$ have the same image under $\eta$ if and only if they are in the same orbit of the L-group. For stratified spaces with boundary there are analogous sequences relative to the boundary. In fact the extension to the left of the sequence above, for a stratified space without boundary, is
\begin{equation*}
	\xymatrix{
	&\tS_{\BQ}(\sing X \times I, \sing X \times \pa I) \ar[r]^-{\eta} &
	\tN_{\BQ}(\sing X\times I, \sing X \times \pa I) \ar[r]^-{\theta} &
	\tL_{\BQ,d_{\sing X \times I}}(\sing X \times I) 
	\ar@{..>} `r[d] `[l] `[llld] `[dll] [dll]\\ 
	&\tS_{\BQ}(\sing X) \ar[r]^-{\eta} &
	\tN_{\BQ}(\sing X) \ar[r]^-{\theta} &
	\tL_{\BQ,d_{\sing X}}(\sing X)
	}
\end{equation*}
where, e.g.,  $d_{\sing X}$ is the dimension function of $\sing X.$
One can iterate and extend the sequence further to the left.

%\blue{It is unclear to me that this is the classic Browder-Novikov-Sullivan-Wall sequence if $\sing X$ is
%an orientable smooth compact manifold without boundary. I believe this compatibility is needed, especially
%in view of possible geometric applications.
%}
%\red{ Is this clear with Remark \ref{rem:Classical}? }
 
\subsection{The $\Pi$-$\Pi$ condition}

\begin{definition}
We say that a map $h: \sing M \lra \sing N$ between stratified spaces with corners satisfies {\bf the $\Pi$-$\Pi$ condition} if:\\
For every connected component of a stratum of $\sing N,$ $S_N,$ there is exactly one connected component of a stratum of $\sing M,$ $S_M,$ such that $h(S_M) \cap S_N \neq\emptyset.$ Moreover, $h(S_M) \subseteq S_N$ and $h_*: \pi_1(S_M) \lra \pi_1(S_N)$ is an isomorphism.
\end{definition}

The $\pi$-$\pi$ theorem (or surgery lemma) in our context is implicit in \cite{BQ} and presented by Weinberger in  \cite[pg. 140]{W-book}, where a proof is also sketched. We formulate it as in Quinn's thesis \cite[Theorem 2.4.4]{Quinn:Thesis} and prove it following \cite{W-book}.

\begin{theorem}[BQ $\Pi$-$\Pi$ theorem] \label{thm:pipi}
Let $\sing M,$ $\sing Y,$ $\sing X$ be stratified spaces with boundary, with the same dimension function, together with decompositions of their boundaries, e.g.,
$\pa \sing X = \pa_0 \sing X \cup \pa_1 \sing X,$ into two codimension zero stratified spaces with common boundary.
Consider a diagram
\begin{equation}\label{eq:pipi1}
	(\sing M; \pa_0\sing M, \pa_1\sing M) \xlra f 
	(\sing Y; \pa_0\sing Y, \pa_1 \sing Y) \xlra \omega 
	(\sing X; \pa_0\sing X, \pa_1\sing X)
\end{equation}
in which $f$ is BQ-normal, $f|: \pa_0 \sing M \lra \pa_0 \sing Y$ is a BQ-equivalence, $\omega$ is BQ-transverse and orientation preserving.

If the inclusion of $\pa_1 \sing Y$ into $\sing Y$ satisfies the $\Pi$-$\Pi$ condition then there is a bordism 
between \eqref{eq:pipi1} and
\begin{equation}\label{eq:pipi1'}
	(\sing M'; \pa_0\sing M, \pa_1\sing M') \xlra {f'} 
	(\sing Y; \pa_0\sing Y, \pa_1 \sing Y) \xlra \omega 
	(\sing X; \pa_0\sing X, \pa_1\sing X)
\end{equation}
where $f'$ satisfies the same properties as $f$ but is moreover a BQ-equivalence.
Explicitly this bordism is a diagram of oriented stratified spaces with corners
\begin{equation*}
	\sing N \xlra F \sing Y \times I \xlra {\omega \times \id} \sing X \times I
\end{equation*}
in which $\pa\sing N = \sing M \cup \sing M' \cup \pa_0\sing M \times I \cup \sing P,$ with $F$ is a BQ-normal map satisfying
\begin{equation}\label{eq:pipi1B}
	(\sing N; \sing M, \sing M', \pa_0\sing M \times I, \sing P)
	\xlra{ (F; f, f', f|\times \id, F|)}
	(\sing Y \times I; \sing Y \times \{0\}, \sing Y \times \{1\}, \pa_0\sing Y \times I, \pa_1\sing Y \times I).
\end{equation}

\end{theorem}

\begin{proof}
We proceed by induction on the depth of the stratified space.
Our base case is when $\sing Y,$ $\sing M,$ and $\sing X$ are smooth manifolds with boundary and this is Theorem 3.3 in Wall's book \cite{Wall}, since in dimension less than five our maps are diffeomorphisms by definition.

Suppose the theorem is established for all stratified spaces with boundary whose stratification has depth less than $k$ and consider \eqref{eq:pipi1} where $\sing Y$ (and hence $\sing M,$ $\sing X$) has a stratification of depth $k.$ Denote the subsets of depth $k$ by a ${}^\dagger$ decoration and note that these are smooth manifolds and that transversality of the maps in \eqref{eq:pipi1} implies that these subsets are preserved by these maps. 
Thus we obtain
\begin{equation}\label{eq:pipi2}
	\striad{M^{\dagger}} \xlra{f^{\dagger}} \striad {Y^{\dagger}} \xlra {\omega^{\dagger}} \striad{X^{\dagger}}
\end{equation}
satisfying the same conditions as the diagram \eqref{eq:pipi1}.
Since the $\Pi$-$\Pi$ condition holds, there is a bordism satisfying the same conditions as \eqref{eq:pipi1B},
\begin{equation}\label{eq:pipi3}
	{N^\dagger} \xlra {F^{\dagger}}  Y^{\dagger} \times I \xlra{ \omega^{\dagger} \times \id } X^{\dagger} \times I
\end{equation}
between \eqref{eq:pipi2} and
\begin{equation*}
	\striad{{M^{\dagger}}'} \xlra{{f^{\dagger}}'} \striad {{Y^{\dagger}}} \xlra {{\omega^{\dagger}}} \striad {{X^{\dagger}}}
\end{equation*}
with ${f^{\dagger}}'$ a homotopy equivalence.

Transversality of $f$ and $\omega$ guarantee that we can find neighborhoods $\cT_{M^{\dagger}} \subseteq \sing M,$ $\cT_{Y^{\dagger}} \subseteq \sing Y,$ and $\cT_{X^{\dagger}}\subseteq \sing X$ that fiber over $M^{\dagger},$ $Y^{\dagger},$ and $X^{\dagger}$ respectively, such that each square in 
\begin{equation}\label{eq:pipi4}
	\xymatrix{
	\cT_{M^{\dagger}} \ar[rr]^-{f|} \ar[d]^-{\phi_M^\dagger} 
		& & \cT_{Y^{\dagger}} \ar[rr]^-{\omega|} \ar[d]^-{\phi_Y^\dagger} 
		& & \cT_{X^{\dagger}} \ar[d]^-{\phi_X^\dagger}\\
	M^{\dagger} \ar[rr]^-{f^{\dagger}} & & Y^{\dagger} \ar[rr]^-{\omega^{\dagger}} & & X_{\dagger} }
\end{equation}
where $f|$ and $\omega|$ denote the restrictions of $f$ and $\omega$ respectively, is a pull-back square.
Note that for $i=0,1$ we have 
\begin{equation*}
	\pa_i \sing M \cap \cT_{M^{\dagger}} = (\phi_M^\dagger)^{-1}(\pa_i M^{\dagger})
\end{equation*}
and similarly for $\sing Y$ and $\sing X,$ so the top row of \eqref{eq:pipi4} is a diagram satisfying conditions analogous to \eqref{eq:pipi1}.

We have an extension of this diagram to a similar diagram over the bordism \eqref{eq:pipi3}
\begin{equation}\label{eq:pipi5}
	\xymatrix{
	{F^{\dagger}}^*(\cT_{Y^{\dagger}} \times I) \ar[rr]^-{\bar F^{\dagger}} \ar[d] 
		& & \cT_{Y^{\dagger}} \times I \ar[rr]^-{\omega| \times \id} \ar[d] & & \cT_{X^{\dagger}}\times I \ar[d]\\
	N^{\dagger} \ar[rr]^-{F^{\dagger}} & & Y^{\dagger}\times I \ar[rr]^-{\omega^{\dagger}\times\id} & &  X_{\dagger} \times I }
\end{equation}
with ${F^{\dagger}}^*(\cT_{Y^{\dagger}})|_{M^{\dagger}} = \cT_{M^{\dagger}}$ and $\bar F^{\dagger}$ the induced map on the pull-back.
Restricting this diagram to the `fixed' $\pa_0$ part of the boundary we get
\begin{equation*}
	\xymatrix{
	(\pa_0\sing M \cap \cT_{M^{\dagger}}) \times I \ar[rr]^-{f|\times \id} \ar[d] & &
		(\pa_0\sing Y \cap \cT_{Y^{\dagger}}) \times I \ar[rr]^-{\omega| \times \id} \ar[d] & &
		(\pa_0\sing X \cap \cT_{X^{\dagger}}) \times I \ar[d] \\
	\pa_0 M^{\dagger} \times I \ar[rr]^-{f^{\dagger}|\times \id} & &
		\pa_0 Y^{\dagger} \times I \ar[rr]^-{\omega^{\dagger}| \times \id} & &
		\pa_0 X^{\dagger} \times I }
\end{equation*}
and so the top row of \eqref{eq:pipi5} is an bordism, satisfying conditions analogous to \eqref{eq:pipi1B}, from the top row of \eqref{eq:pipi4} to
\begin{equation}\label{eq:pipi6}
	( {f^{\dagger}}')^*\cT_{Y^{\dagger}} \xlra { {\bar f^{\dagger}}' } 
		\cT_{Y^{\dagger}} \xlra {\omega | } 
		\cT_{X^{\dagger}}
\end{equation}
in which ${\bar f^{\dagger}}'$ is a BQ-equivalence.

Now, as in \cite[Theorem 2.14]{Browder:TransG} \cite[\S 4.3]{W-book}, we multiply each space in \eqref{eq:pipi1} by the unit interval and attach the top row of \eqref{eq:pipi5} to get a bordism
\begin{multline*}
	\sing M \times I 
	\bigcup_{\cT_{M^{\dagger}}\times \{1\} \sim {F^{\dagger}}^*(\cT_{Y^{\dagger}})|_{M^{\dagger}}}
	{F^{\dagger}}^*(\cT_{Y^{\dagger}}) 
	\xlra{  (f\times\id) \cup \bar F^{\dagger} }
	\sing Y \times I
	\bigcup_{\cT_{Y^{\dagger}}\times \{1\} \sim \cT_{Y^{\dagger}} \times \{0\}}
	\cT_{Y^{\dagger}} \times I \\
	\xlra{ \omega \times \id \cup \omega| \times \id }
	\sing X \times I
	\bigcup_{\cT_{X^{\dagger}}\times \{1\} \sim \cT_{X^{\dagger}} \times \{0\}}
	\cT_{X^{\dagger}} \times I 
\end{multline*}
from \eqref{eq:pipi1} to a similar diagram over $\sing X$ which we denote
\begin{equation}\label{eq:pipi7}
	\triad{P} \xlra{g} \triad Y \xlra \omega \triad X
\end{equation}
and which restricts to \eqref{eq:pipi6} in a neighborhood of the subsets of depth $k.$

Now we remove these neighborhoods of the subsets of depth $k$ to form
\begin{equation*}
	\sing P^+ = \sing P \setminus 
	( {f^{\dagger}}')^*\cT_{Y^{\dagger}}, \quad
	\sing Y^+ = \sing Y \setminus \cT_{Y^{\dagger}}, \quad
	\sing X^+ = \sing X \setminus \cT_{X^{\dagger}}.
\end{equation*}
These are stratified spaces with corners (see, e.g., \cite[\S6]{Albin:Hodge}, \cite{Verona}) 
and we define
\begin{equation*}
	\pa_0\sing Y^+ = (\pa_0\sing Y \cap \sing Y^+ ) \cup \pa\cT_{Y^{\dagger}}, \quad
	\pa_1\sing Y^+ = \pa_1\sing Y \cap \sing Y^+
\end{equation*}
and similarly for $\sing P^+$ and $\sing X^+.$
Note that the restrictions of $g$ and $\omega$ to $\pa_0\sing P^+$ and $\pa_0 \sing Y^+$ are BQ-equivalences so we have a diagram
\begin{equation*}
	(\sing P^+; \pa_0 \sing P^+, \pa_1\sing P^+) \xlra{g|}
	(\sing Y^+; \pa_0 \sing Y^+, \pa_1 \sing Y^+) \xlra{\omega|}
	(\sing X^+; \pa_0\sing X^+, \pa_1 \sing X^+)
\end{equation*}
satisfying conditions analogous to \eqref{eq:pipi1}. (Note that though the stratified space now has corners of codimension two, one can `smooth out the corners' as in \cite[\S2.6]{Wall:DiffTop}, \cite[\S3]{HMM2}.)\\

Moreover, the compatibility between the stratifications and the boundary faces implies that each stratum of  $\sing Y^+$ is homotopy equivalent to the corresponding stratum of $\sing Y,$ and the same is true for the strata of $\pa_1 \sing Y^+.$ (Indeed iterating this process of removing tubular neighborhoods of deepest strata produces the resolution of $\sing Y$ which does not change the homotopy type of the strata.) Significantly, the inclusion of $\pa_1\sing Y^+$ into $\sing Y^+$  satisfies the $\Pi$-$\Pi$ condition and since $\sing Y^+$ has depth less than $k$ we can apply our inductive hypothesis to find a bordism, satisfying conditions analogous to \eqref{eq:pipi1B},
\begin{equation*}
	\sing N^+ \xlra G \sing Y^+ \times I \xlra \Omega \sing X^+ \times I
\end{equation*}
between 
\begin{equation*}
\begin{gathered}
	\triad{P^+} \xlra {g|} \triad {Y^+} \xlra {\omega|} \triad {X^+} \\
	\Mand
	\triad{{P^+}'} \xlra {g'} \triad {Y^+} \xlra {\omega|} \triad {X^+}
\end{gathered}\end{equation*}
with $g'$ a  BQ-equivalence.
Since the bordism does not change the spaces $\pa(( {f^{\dagger}}')^*\cT_{Y^{\dagger}}),$ $\pa \cT_{Y^{\dagger}},$ $\pa \cT_{X^{\dagger}}$ or the maps between them, we can glue in the bordism \eqref{eq:pipi6} to finally obtain an bordism between \eqref{eq:pipi1} and \eqref{eq:pipi1'} with 
\begin{equation*}
	\sing M' = {{P^+}'} \cup ( {f^{\dagger}}')^*\cT_{Y^{\dagger}},
\end{equation*}
and $f' = g' \cup {\bar f^{\dagger}}'$ a BQ-equivalence as required.
\end{proof}

The key to applying the $\Pi$-$\Pi$ theorem is a result of Wall that allows us to represent every class in $\tL_{\BQ}(\sing X)$ by a `restricted' representative. This is sometimes referred to as ``$L^1=L^2$" evoking the notation of \cite[Chapter 9]{Wall}.

\begin{definition}
Let $\sing X$ be a stratified space (possibly with boundary).
An $\cL$-cycle over $\sing X,$
\begin{equation*}
	(\sing M; \pa\sing M) \xlra{\phi} (\sing Y;\pa\sing Y) \xlra{\omega} \sing X,
\end{equation*}
is a {\bf restricted $\cL$-cycle} if $\omega: \sing Y \lra \sing X$ satisfies the $\Pi$-$\Pi$ condition.

A null bordism of a restricted $\cL$-cycle over $\sing X,$
\begin{equation*}
	(\sing N; \pa_1 \sing N, \pa_2 \sing N) \xlra{\Phi} (\sing Z; \pa_1\sing Z, \pa_2\sing Z) \xlra{\Omega} \sing X \times I,
\end{equation*}
 is a {\bf restricted null bordism} if $\Omega: \sing Z \lra \sing X \times I$ satisfies the $\Pi$-$\Pi$ condition.
\end{definition}

\begin{remark}
If $\sing X$ has depth zero then these are the restricted cycles of \cite[Chapter 9]{Wall}, see Remark \ref{rem:Classical}.
\end{remark}

\begin{theorem}\label{thm:L1L2}
Let $\sing X$ be a stratified space (possibly with boundary).
Every element of $\tL_{\BQ}(\sing X)$ is $\cL$-bordant, relative to the boundary, to a restricted $\cL$-cycle over $\sing X.$
If a restricted $\cL$-cycle over $\sing X$ is null bordant, then it participates in a restricted null bordism.
\end{theorem}

\begin{proof}
Our proof is parallel to that of the $\Pi$-$\Pi$ theorem. When possible we will simply refer back to the latter proof.\\

We will prove by induction on the depth of the stratifications that, whenever we have
\begin{equation}\label{eq:L1L21}
	\sing M \xlra{\phi} \sing Y \xlra{\omega} \sing X
\end{equation}
with $\phi$ BQ-normal, $\omega$ BQ-transverse, and $\sing M,$ $\sing Y,$ $\sing X$ stratified spaces {\em with corners} with the same dimension functions, there is a bordism relative to the boundary  to a similar diagram
\begin{equation*}
	\sing M' \xlra{\phi'} \sing Y' \xlra{\omega'} \sing X
\end{equation*}
in which $\omega'$ satisfies the $\Pi$-$\Pi$ condition.
Specifically, there is a diagram
\begin{equation*}
	(\sing N;\sing M, \sing M', \pa \sing M \times I) \xlra{(\Phi; \phi,\phi', \phi|\times\id)} 
	(\sing Z;\sing Y, \sing Y', \pa \sing Y \times I) \xlra{(\Omega; \omega, \omega', \omega|\times \id)}
	\sing X \times I
\end{equation*}
in which $\Phi$ is BQ-normal and $\Omega$ is BQ-transverse and moreover
\begin{equation*}
	\sing M \cap \sing M' = \emptyset, \quad
	\sing M \cap (\pa \sing M \times I) = \pa \sing M, \quad 
	\sing M' \cap (\pa \sing M \times I) = \pa \sing M'
\end{equation*}
and similarly for $\sing Z.$ \\
%\textcolor{m}{I do not think
%we have given the definition of "relative to the boundary"....}\\\textcolor{b}{Perhaps not, but we do specify this in the next sentence for the reader unfamiliar with the term}\\
%\textcolor{m}{I do not think the above line would help
%the unfamiliar reader very much: what are we demanding precisely ? I am confused by the component $\pa \sing M \times I$....
%(I remember we want the bordism to be fixed along the boundary and of product type
%as you suggest, but is this encoded by the notation ? Anyway, I feel we must be more precise here...\\ I remember 
%a formal definition
%in Dovermann-Rothenberg- or Dovermann-Schultz but I cannot locate it...)}
%This handles both the case of $\cL$-cycles and $\cL$-bordisms.
%Note that we do note impose any restrictions on the behavior of $\phi$ and $\omega$ on the boundary.

We proceed by induction on the depth of the stratified space.
Our base case is for smooth manifolds with corners.  If $\dim Y<5$ then the theorem is automatic since $f$ and $\omega$ are then both diffeomorphisms.
Assuming $\dim Y\geq 5,$ this case is handled by Wall in \cite[Theorems 9.4, 9.5]{Wall}, where he shows that this can be arranged by carrying out surgery on $Y$ along $1$-handles and $2$-handles and then modifying $M$ along the preimages of the corresponding embeddings. As pointed out in \cite[Proof of Lemma 3]{Milnor:Kill}, a theorem of Whitney implies that for $\dim Y>4$ any homotopy class of maps from $\bbS^1$ or $\bbS^2$ into $Y$ contains an embedding with image contained in the interior of $Y.$ Thus all of the modifications in $M$ and $Y$ can be carried out in their interiors. 

Suppose the theorem is established for all stratified spaces with boundary 
%\footnote{\textcolor{m}{boundary or corners ?}\textcolor{b}{In the statement of the theorem $X$ only has boundary}} 
whose stratification has depth less than $k$ and consider \eqref{eq:L1L21} where $\sing Y$ (and hence $\sing M$) has a stratification of depth $k.$ Denote the subsets of depth $k$ by a ${}^\dagger$ decoration and note that these are smooth manifolds and that transversality of the maps in \eqref{eq:L1L21} implies that these subsets are preserved by these maps. 
Thus we obtain
\begin{equation*}%\label{eq:L2L12}
	(\sing{M^{\dagger}}, \pa\sing{M^{\dagger}}) \xlra{f^{\dagger}} 
	(\sing{Y^{\dagger}}, \pa\sing{Y^{\dagger}}) \xlra {\omega^{\dagger}}
	\sing X,
\end{equation*}
an element of $\cL_{\BQ}(\sing{X^{\dagger}}).$
Applying the base case to this situation we obtain an $\cL$-bordism, relative to the boundary, to an element satisfying the desired $\Pi$-$\Pi$ condition. Proceeding as in the proof of Theorem \ref{thm:pipi} we can lift this $\cL$-bordism to neighborhoods of these subsets and then graft it onto the product of \eqref{eq:L1L21} with the unit interval to obtain an $\cL$-bordism, relative to the boundary, between \eqref{eq:L1L21} and an element of $\cL_{\BQ}(\sing X)$ analogous to \eqref{eq:pipi7},
\begin{equation*}
	(\sing P, \pa\sing P) \xlra{g}
	(\sing V, \pa\sing V) \xlra{\alpha} 
	\sing X,
\end{equation*}
where $\alpha: \sing{V} \lra \sing X$ satisfies the $\Pi$-$\Pi$ condition for strata of depth $k.$

Now we remove consistent tubular neighborhoods of the subsets of depth $k$ to form $\sing P^+,$ $\sing V^+,$ and $\sing X^+$ as in the proof of Theorem \ref{thm:pipi}.
This gives a diagram
\begin{equation*}
	\triad{P^+} \xlra {g|} \triad {V^+} \xlra{\alpha|} \triad X^+
\end{equation*}
in which the stratification on $\sing V^+$ has depth less than $k.$ By our inductive hypothesis there is a bordism, relative to the boundary, between this and another diagram
\begin{equation*}
	\triad{{P^+}'} \xlra {g'} \triad{{V^+}'} \xlra{\alpha'} \triad X^+
\end{equation*}
for which $\alpha': {{V^+}'} \lra \hat X^+$ satisfies the $\Pi$-$\Pi$ condition.
Since the bordism is relative to the boundary we may glue in the previous bordism between neighborhoods of the strata of depth $k$ to obtain a bordism, relative to the boundary, between our original diagram and 
\begin{equation*}
	\sing{{P^+}'} \cup \cT_{P^\dagger}
	\xlra{g' \cup (g|)}
	\sing{{V^+}'} \cup \cT_{V^{\dagger}}
	\xlra{\alpha' \cup (\alpha|)}
	\sing X.
\end{equation*}
Finally, because of the compatibility between the stratifications and the boundary faces, the fact that $\alpha$ satisfies the $\Pi$-$\Pi$ condition between the strata of depth $k$ and $\alpha'$ satisfies the $\Pi$-$\Pi$ condition between strata of depth less than $k$ means that $\alpha' \cup (\alpha|)$ satisfies the $\Pi$-$\Pi$ condition on all strata.

\end{proof}

Our final preliminary result is to point out that the sum in $\tL^{\BQ}(\sing X \times I)$ which is induced by disjoint union can, when appropriate, be carried out by identifying boundary faces.

\begin{lemma} \label{lem:SumLemma}
Let 
\begin{equation*}
\begin{gathered}
	\alpha: (\sing M; \pa \sing M) \xlra{\phi} (\sing Y; \pa \sing Y) \xlra{\omega} \sing X \times I \\
	\beta: (\sing L; \pa \sing L) \xlra{\psi} (\sing W; \pa \sing W) \xlra{\theta} \sing X \times I
\end{gathered}
\end{equation*}
be two $\cL$-cycles over $\sing X \times I.$
Suppose that the pre-images in $\alpha$ lying above $\sing X \times \{0\}$ coincide with the pre-images in $\beta$ lying above $\sing X \times \{1\},$
\begin{equation*}
	\gamma: (\sing P; \pa \sing P) \xlra{\rho} (\sing V; \pa \sing V) \xlra{\xi} \sing X.
\end{equation*}
The class of $\alpha+\beta$ in $\tL_{\BQ}(\sing X \times I)$ is represented by the union of the diagrams along their common boundary, $\alpha \cup_{\gamma} \beta.$
\end{lemma}

\begin{proof}
Let $\sing N = (\sing M \cup_{\sing P} \sing L) \times I,$
$\sing Z = (\sing Y \cup_{\sing V} \sing W) \times I,$ and consider a diagram
\begin{equation}\label{eq:SumLemma1}
	\sing N \xlra{\Phi} \sing Z \xlra{\Omega} \sing X \times I
\end{equation}
where $\Phi$ has the form $\Phi(x,t) = (\Phi_t(x),t)$ and similarly $\Omega(x,t) = (\Omega_t(x),t).$
By definition, $\sing P$ has a collar neighborhood in each of $\sing M$ and $\sing L$ and gluing these together we have a neighborhood of the form $(-\eps,\eps) \times \sing P$ in $\sing M \cup_{\sing P} \sing L$ and of the form 
$(-\eps,\eps) \times \sing P \times I$ in $\sing N.$
Similarly, we have a neighborhood of the form $(-\eps,\eps) \times \sing V$ in $\sing Y \cup_{\sing V} \sing W$ and of the form $(-\eps,\eps) \times \sing V \times I$ in $\sing Z.$
Note that
\begin{equation*}
	\sing M \cup_{\sing P} \sing L \setminus \lrpar{ (-\eps,\eps) \times P } 
	= \sing M \sqcup \sing L, \quad
	\sing Y \cup_{\sing V} \sing W \setminus \lrpar{ (-\eps,\eps) \times V } 
	= \sing Y \sqcup \sing W.
\end{equation*}
We choose $\Phi$ so that $\Phi_0 = \phi\cup\psi$ and for $t>0,$ $\Phi_t\rest{ (-t\eps,t\eps) \times P } = \id \times \rho,$ while off of this neighborhood $\Phi$ is essentially $\phi\sqcup\psi.$ We similarly choose $\Omega_t.$
With these choices, since $\phi\rest{\pa \sing M}$ and $\psi\rest{\pa\sing L}$ are BQ-equivalences, we recognize \eqref{eq:SumLemma1} as an $\cL$-bordism between the disjoint union of $\alpha$ and $\beta$ and their union along $\gamma.$
\end{proof}

\subsection{Surgery theorem }

With the preliminary results out of the way, we can establish the fundamental result of surgery: a normal map is normal bordant to an equivalence precisely when its surgery obstruction vanishes.

\begin{theorem} [Exactness part 1] \label{thm:exact1}
Let $\sing X$ be an oriented stratified space with boundary (possibly empty) and let
\begin{equation*}
	h: \sing K \lra \pa \sing X
\end{equation*}
be a (stratum preserving) diffeomorphism.\\
%\textcolor{m}{But with the new definition of $\tN_{\BQ}(\sing X, \pa \sing X)$, should not $h$ be a stratified diffeomorphism ?}\textcolor{b}{Changed}\\
Given a stratified space $\sing M$ with boundary $\pa \sing M = \sing K,$ and a BQ-normal map
\begin{equation*}
	\phi: \sing M \lra \sing X
\end{equation*}
extending $h,$ there is a normal bordism relative to $h$ between $\phi$ and a BQ-equivalence if and only if $\phi$ is $\cL$-null-bordant.\\%\footnote{NB: Not relative to $h$! Note that in the proof the vanishing of the normal map in the {\em usual} L-group implies that the normal map is normal bordant relative to $h$}.\\

Equivalently, the sequence
\begin{equation*}
	\tS_{\BQ}(\sing X, \pa \sing X) \xlra{\eta}
	\tN_{\BQ}(\sing X, \pa \sing X) \xlra{\theta}
	\tL_{\BQ}(\sing X)
\end{equation*}
is exact.

\end{theorem}

\begin{proof}
If there is a normal bordism, relative to $h,$ between $\phi$ and a BQ-equivalence, $\phi': \sing M' \lra \sing X,$ say 
\begin{equation*}
	(\sing N; \sing M, \sing M', \sing K \times I) \xlra{\Phi}
	(\sing X\times I; \sing X \times \{0\}, \sing X \times \{1\}, \pa \sing X \times I) \xlra\id
	\sing X \times I,
\end{equation*}
then this bordism witnesses the triviality of $\phi$ in $\tL_{\BQ}(\sing X).$\\

On the other hand, if 
\begin{equation*}
	(\sing M;\sing K) \xlra{(\phi;h)} \\
	(\sing X; \pa \sing X) \xlra{\id}
	\sing X
\end{equation*}
is a null bordant $\cL$-cycle over $\sing X$ 
%\footnote{\textcolor{m}{shall we recall
%that $\pa \sing M=\sing K$ ?}\textcolor{b}{If you wish, but it is in the statement of the theorem}} 
then, since it is a restricted cycle, there is by Theorem \ref{thm:L1L2} a restricted null bordism.
That is to say, there are maps of stratified spaces with corners
\begin{equation*}
	(\sing N; \pa_1 \sing N, \pa_2 \sing N) \xlra{\Phi}
	(\sing Z; \pa_1 \sing Z, \pa_2 \sing Z) \xlra{\Omega}
	\sing X \times I
\end{equation*}
with $\Phi$ BQ-normal, $\Phi|: \pa_2 \sing N \lra \pa_2 \sing Z$ a BQ-equivalence, $\Omega$  BQ-transverse, 
\begin{equation*}
	\lrpar{
	(\pa_1\sing N, \pa_{12}\sing N) \xlra{\Phi|} 
	(\pa_1 \sing Z, \pa_{12}\sing Z) \xlra{\pi\circ\Omega|} 
	\sing X}
	=
	\lrpar{
	(\sing M; \pa \sing M) \xlra{\phi}
	(\sing X; \pa \sing X) \xlra{\id}
	\sing X},
\end{equation*}
%
%\footnote{\textcolor{m}{once again, I would maybe use round parentesis, here and in 6 lines...}\textcolor{b}{Changed}} 
and $\Omega$ satisfies the $\Pi$-$\Pi$ condition.
Now $\pa_1 \sing Z = \sing X$ and $\Omega\rest{\pa_1\sing Z} = \id,$ so the inclusion of $\pa_1\sing Z$ into $\sing Z$ satisfies the $\Pi$-$\Pi$ condition.
Since $\Phi\rest{\pa_2\sing N}$ is a BQ-equivalence we can apply the $\Pi$-$\Pi$ theorem and find a bordism (not necessarily an $\cL$-bordism) between $\sing N \xlra\Phi \sing Z \xlra\Omega \sing X \times I$ and
\begin{equation*}
	(\sing N'; \pa_1\sing N', \pa_2\sing N') \xlra{\Phi'} 
	(\sing Z; \pa_1 \sing Z, \pa_2 \sing Z) \xlra{\Omega}
	\sing X \times I
\end{equation*}
with $\Phi'$ a BQ-equivalence, and
\begin{equation*}
	\lrpar{
	\pa_2\sing N' \xlra{\Phi'|}
	\pa_2 \sing Z }
	=
	\lrpar{
	\pa_2\sing N \xlra{\Phi|}
	\pa_2 \sing Z }.
\end{equation*}
The bordism itself has the form
\begin{equation*}
	\sing L \xlra{\Psi} \sing Z \times I \xlra\Gamma \sing X \times I^2
\end{equation*}
or, more explicitly,
\begin{equation*}
	(\sing L; \sing N, \sing N', \pa_2\sing N \times I, \sing P)
	\xlra{(\Psi;\Phi, \Phi', \Phi|\times I, \Psi|)}
	(\sing Z \times I; \sing Z \times \{0\}, \sing Z \times \{1\}, \pa_2\sing Z \times I, \pa_1 \sing Z \times I)
	\xlra{\Omega \times \id}
	\sing X \times I^2,
\end{equation*}
where this diagram serves to define $\sing P.$
Note that 
\begin{equation*}
	\pa\sing P 
	= \sing P \cap (\sing N \cup \sing N' \cup \pa_2\sing N \times I) 
	= \pa_1\sing N \cup \pa_1 \sing N' \cup \pa_{12}\sing N \times I
	= \sing M \cup \pa_1\sing N'\cup \sing K \times I
\end{equation*}
and recall that $\pa_1\sing Z = \sing X,$ so if we restrict this diagram to the boundary face $\sing P,$ we find
\begin{equation*}
	(\sing P; \sing M, \pa_1 \sing N', \sing K \times I) \xlra{(\Psi|;\phi, \Phi'|, h \times \id)}
	(\sing X \times I; \sing X \times \{0\}, \sing X \times \{1\}, \pa \sing X \times I) \xlra{\id}
	\sing X \times I.
\end{equation*}
We recognize this as a normal bordism, relative to $h:\sing K \lra \pa\sing X,$ (i.e., without changing $h$) 
%\footnote{\textcolor{m}{as I have pointed out,
%we have never given a precise definition of "relative to $h$"} \textcolor{b}{Added}}
between
\begin{equation*}
	(\sing M;\sing K) \xlra{(\phi;h)} 
	(\sing X; \pa \sing X) \xlra{\id}
	\sing X
	\Mand
	(\pa_1\sing N';\sing K) \xlra{(\Phi'|;h)} 
	(\sing X; \pa \sing X) \xlra{\id}
	\sing X
\end{equation*}
which, since $\Phi'|$ is a BQ-equivalence, proves the theorem.

\end{proof}

\subsection{Wall realization}$\,$\\
In this section we follow Dovermann-Rothenberg \cite[\S 8]{Dovermann-Rothenberg}.

\begin{theorem}[Wall realization] \label{thm:WallReal}
Let $\sing X$ be a stratified space (without boundary) and $\sing L \xlra{f} \sing X$ a BQ-equivalence.
Every element $\alpha \in \tL_{\BQ}(\sing X \times I)$ has a representative of the form
\begin{equation*}
	(\sing W; \pa_-\sing W, \pa_+\sing W) \xlra{F}
	(\sing X \times I; \sing X \times \{0\}, \sing X \times \{1\}) \xlra{\id}
	\sing X \times I
\end{equation*}
with
\begin{equation*}
	\lrspar{
	\pa_-\sing W \xlra{F|} \sing X \times \{0\} } = 
	\lrspar{\sing L \xlra{f} \sing X}.
\end{equation*}
\end{theorem}

Note that this representative is an element of $\cN_{\BQ}(\sing X \times I)$ and that the restriction
\begin{equation*}
	\pa_+\sing W \xlra{F|} \sing X \times \{1\}
\end{equation*}
gives another Thom-Mather structure on $\sing X.$

\begin{proof}
Choose an $\cL$-cycle representing $\alpha,$
\begin{equation*}
	(\sing M; \pa \sing M) \xlra{\phi}
	(\sing Y; \pa \sing Y) \xlra{\omega}
	\sing X \times I
\end{equation*}
and consider the null bordant $\cL$-cycle obtained from $f,$
\begin{equation*}
	(\sing L \times I; \sing L \times \pa I) \xlra{f \times \id}
	(\sing X \times I; \sing X \times \pa I) \xlra{\id}
	\sing X \times I.
\end{equation*}
Adding these $\cL$-cycles together produces another representative of $\alpha,$
\begin{equation*}
	(\sing M \sqcup \sing L \times I; \pa\sing M \sqcup \sing L \times \pa I) \xlra{\phi \sqcup f\times \id}
	(\sing Y \sqcup \sing X \times I; \pa\sing Y \sqcup \sing X \times \pa I) \xlra{\omega \sqcup \id}
	\sing X \times I,
\end{equation*}
We can improve this representative using Theorem \ref{thm:L1L2} to obtain 
\begin{equation*}
	(\sing M'; \pa\sing M \sqcup \sing L \times \pa I) \xlra{\phi'}
	(\sing Y'; \pa\sing Y \sqcup \sing X \times \pa I) \xlra{\omega'}
	\sing X \times I,	
\end{equation*}
with $\phi'$ and $\omega'$ equal to $\phi$ and $\omega$ when restricted to the boundary, and with $\omega$ satisfying the $\Pi$-$\Pi$ condition.
Let us write
\begin{equation*}
\begin{gathered}
	\pa_1 \sing M' = \pa \sing M \sqcup \sing L \times \{1\}, \quad
	\pa_2 \sing M' = \sing L \times \{0\} = \sing L \\
	\pa_1 \sing Y' = \pa \sing Y \sqcup \sing X \times \{1\}, \quad
	\pa_2 \sing Y' = \sing X \times \{0\} = \sing X.
\end{gathered}
\end{equation*}
By commutativity of
\begin{equation*}
	\xymatrix{
	\sing X \times \{0\} \ar[d]^{\omega'|} \ar@{^(->}[r] & \sing Y' \ar[d]^{\omega'} \\
	\sing X \times \{0\} \ar@{^(->}[r] & \sing X \times I }
\end{equation*}
we see that the inclusion $\pa_2\sing Y' \hookrightarrow \sing Y'$ satisfies the $\Pi$-$\Pi$ condition.
Since moreover the map
\begin{equation*}
	\phi'|: \pa_1 \sing M'  \lra \pa_1\sing Y'
\end{equation*}
is a BQ-equivalence, we can apply Theorem \ref{thm:pipi}, the $\Pi$-$\Pi$ theorem, relative to this part of the boundary to find a BQ-equivalence
\begin{equation*}
	\sing M'' \xlra{\phi''} \sing Y'
\end{equation*}
where $\pa \sing M'' = \pa_1 \sing M'' \sqcup \pa_2\sing M'',$ with $\pa_1\sing M''=\pa_1\sing M'$ and $\phi''$ is equal to $\phi'$ on $ \pa_1\sing M',$ and $\phi'':\pa_2\sing M'' \lra \pa_2\sing Y''$ a BQ-equivalence.
The BQ-equivalence $\phi''$ is related to $\phi'$ by a bordism
\begin{equation*}
	\sing N \xlra{\Phi} \sing Y' \times I
\end{equation*}
where $\pa \sing N = \sing M' \cup \sing M'' \cup (\pa_1\sing M' \times I) \cup \sing P,$
\begin{equation*}
	\sing M' \cap \sing M'' = \emptyset, \quad
	\sing M' \cap (\pa_1\sing M' \times I) = \pa_1\sing M' = \sing M''\cap (\pa_1\sing M' \times I), \quad
	\pa\sing P = \pa_2 \sing M' \sqcup \pa_2 \sing M''.
\end{equation*}
The restriction of $\Phi$ to $\sing P$ yields an element $\gamma$ of $\cL_{\BQ}(\sing X \times I)$
\begin{equation*}
	(\sing P; \sing L, \pa_2\sing M'') \xlra{\Phi|}
	(\pa_2 \sing Y' \times I = \sing X \times I; \sing X \times \{0\}, \sing X \times \{1\})
	\xlra{\id}
	\sing X \times I
\end{equation*}
of the kind required in the statement of the theorem.
(Incidentally, note that the fact that $\Phi|_{\sing P}$ is BQ-normal and not a BQ-equivalence is why the bordism $\sing N$ is not a null bordism for $\alpha.$)

The BQ-normal map $\Phi: \sing N \lra \sing Y' \times I$ is a null bordism of the $\cL$-cycle 
\begin{equation*}
	\sing M' \cup_{\sing L} \sing P \lra
	\sing Y' \cup_{\sing X} \sing X \times I \lra
	\sing X \times I
\end{equation*}
which shows, by Lemma \ref{lem:SumLemma}, that $\alpha$ and $\gamma$ represent the same class in $\tL_{\BQ}(\sing X \times I).$

\end{proof}

\begin{corollary} \label{cor:LAct}
Given $[\alpha] \in \tL_{\BQ}(\sing X \times I)$ and $[\beta] \in \tS_{\BQ}(\sing X)$ we can use the theorem to choose representatives of the form
\begin{equation*}
\begin{gathered}
	\beta: \sing M \xlra{f} \sing X \\
	\alpha: (\sing W; \sing M, \sing M') \xlra{(\phi;\id, \phi_2)}
	(\sing M \times [0,1]; \sing M \times \{0\}, \sing M \times \{1\}) \xlra{\id}
	\sing M \times I,
\end{gathered}
\end{equation*}
and then the class of $f\circ \phi_2: \sing M' \lra \sing X$ in $\tS_{\BQ}(\sing X)$ is well-defined and denoted $\pa(\alpha)(\beta).$
The map
\begin{equation*}
	\xymatrix @R=1pt {
	\tL_{\BQ}(\sing X \times I) \times \tS_{\BQ}(\sing X) \ar[r] & \tS_{\BQ}(\sing X) \\
	([\alpha],[\beta]) \ar@{|->}[r] & \pa(\alpha)(\beta) }
\end{equation*}
defines a group action of the Browder-Quinn L-group of $\sing X \times I$ on the structure set of $\sing X.$
\end{corollary}

\begin{proof}
Since $f: \sing M \lra \sing X$ is a BQ-equivalence, $\tL_{\BQ}(\sing X\times I) = \tL_{\BQ}(\sing M \times I)$ and we can use the Wall representation theorem starting with the BQ-equivalence $\sing M \xlra{\id} \sing M$ to represent $\alpha$ as above.

If we fix the representative $\beta$ then any two representatives $\alpha,$ $\alpha'$ of $[\alpha]$ as above can be glued together along their common boundary and the result $\gamma \in \cL_{\BQ}(M \times I)$ represents the zero element of $\tL_{\BQ}(M \times I).$ It follows, from Theorem \ref{thm:exact1} applied to $\sing M \times I,$ that $\gamma$ is normal bordant relative to the boundary to a BQ-equivalence. Thus $f \circ \phi_2$ and $f \circ \phi_2'$ represent the same element of $\tS_{\BQ}(M).$\\

If the BQ-equivalence $\beta': \sing L \xlra{f'} \sing X$ represents the same class as $\beta,$ then there is a bordism between them
\begin{equation*}
	(\sing N; \sing M, \sing L) \xlra{(F;f,f')}
	(\sing X \times I; \sing X \times \{0\}, \sing X \times \{1\}).
\end{equation*}
Using the theorem we can find a representative of $[\alpha]$ of the form
\begin{equation*}
	\alpha': (\sing V; \sing L, \sing L') \xlra{(\psi:\id, \psi_2)}
	(\sing L \times [0,1]; \sing L \times \{0\}, \sing L \times \{1\}).
\end{equation*}
Now let us glue these, and $\alpha,$ together in the following order by matching the `lower boundary' of one row with the `upper boundary' of the following row,
\begin{equation*}
\xymatrix @R=5pt{
	(\sing V; \sing L, \sing L') \ar[rr]^-{(\psi:\id, \psi_2)} & &
	(\sing L \times [0,1]; \sing L \times \{0\}, \sing L \times \{1\})
	\ar[r]^-{f' \times \id} &
	\sing X \times I \\
	(\sing N; \sing M, \sing L) \ar[rr]^-{\id} & &
	(\sing N; \sing M, \sing L) \ar[r]^-{(F;f,f')} &
	\sing X \times I \\
	(\sing W^{\op}; \sing M', \sing M) \ar[rr]^-{(\phi;\phi_2, \id)} & &
	(\sing M \times [0,1]; \sing M \times \{0\}, \sing M \times \{1\}) \ar[r]^-{f \times \id} &
	\sing X \times I }
\end{equation*}
We end up with an cycle in $\cL_{\BQ}(X \times I)$ with $\pa(\alpha')(\beta')$ along the upper boundary and $\pa(\alpha)(\beta)$ along the lower boundary. Moreover, this cycle is null bordant since by Lemma \ref{lem:SumLemma}, it represents the class $[\alpha]+0-[\alpha] = 0.$ It follows as in the previous case that  $\pa(\alpha')(\beta')$ and $\pa(\alpha)(\beta)$ represent the same element in $\tS_{\BQ}(\sing X).$\\

Compatibility with the group operation on $\tL_{\BQ}(\sing X \times I)$ is easy as the operation is given by stacking normal bordisms together as in Lemma \ref{lem:SumLemma}.
\end{proof}

\begin{corollary}\label{Exactness part 2}[Exactness part 2]
Let $\sing X$ be a stratified space (without boundary).
The sequence
\begin{equation*}
	\xymatrix{
	\tN_{\BQ}(\sing X \times I, \sing X \times \pa I) \ar[r]^-{\theta} &
	\tL_{\BQ}(\sing X \times I) \ar@{..>}[r]^-{\pa} &
	\tS_{\BQ}(\sing X) \ar[r]^-{\eta} &
	\tN_{\BQ}(\sing X)
	}
\end{equation*}
is exact in that two elements of the L-group have the same action on the class of the identity map precisely when their difference is in the image of $\theta,$ and two elements in the structure set are in the same orbit precisely when they have the same image under $\eta.$
\end{corollary}

%\blue{If $\sing X$ is closed $4k-1$-dimensional orientable smooth and without boundary, then Chang-Weinberger, page 315,
%after Theorem 2, state that the image of $\theta$ is made of classes with cycles made of manifolds
%without boundary. Why is this true ? And, does it apply to our situation ? This is crucial
%for geometric applications.....You wrote something in an email of March 20 2017 to me and it would be
%nice to change notations as you suggest so as to make this statement true.}
%\red{ This is true in our situation. As noted in Remark \ref{rem:PiPiSuff}, $\tL_{\BQ, d}(\sing X) \cong \tL_{\BQ,d}(\sing X \times I).$}

\begin{proof}
Given $x_1, x_2 \in \tL_{\BQ}(\sing X \times I)$ such that $\pa(x_1)(\id) = \pa(x_2)(\id)$ we can use the Wall representation theorem to find
\begin{equation*}
	(\sing W_i; \sing X, \pa_+\sing W_i) \xlra{F_i}
	(\sing X \times I; \sing X \times \{0\}, \sing X \times \{1\})
	\xlra{\id} \sing X \times I
\end{equation*}
representing $x_i.$ Without loss of generality 
$\pa_+\sing W_1 \xlra{F_1} \sing X$ and $\pa_+\sing W_2 \xlra{F_2} \sing X$ can be taken to be the same representative of $\pa(x_1)(\id),$ so that we may form
\begin{equation*}
	(\sing W_1 \sqcup_{\pa_+\sing W_i} -\sing W_2; \sing X, \sing X) 
	\xlra{F_1 \sqcup F_2}
	(\sing X \times I; \sing X \times \{0\}, \sing X \times \{1\})
	\xlra{\id} \sing X \times I
\end{equation*}
and recognize this as a representative of a class in $\tN_{\BQ}(\sing X \times I, \sing X \times \pa I)$ representing $x_1 -x_2.$ The converse follows by similar reasoning.

%Given $[\alpha] \in \tL_{\BQ}(\sing X \times I)$ and $[\beta] \in \tS_{\BQ}(\sing X)$ we represent them as in Corollary \ref{cor:LAct}. We can take $\beta \times \id: \sing M \times \pa I \lra \sing X \times \pa I$ as boundary structure on $\pa(\sing X \times I)$ in defining $\tN_{\BQ}(\sing X \times I, \sing X \times \pa I)$ and see that $\pa(\alpha)(\beta) = [\beta]$ precisely when $\alpha$ is in $\tN_{\BQ}(\sing X \times I, \sing X \times \pa I).$

If $[\beta], [\beta'] \in \tS_{\BQ}(\sing X)$ have the same image under $\eta$ then there is a normal bordism $\alpha$ between $\beta$ and $\beta'.$ This normal bordism defines an element of $\tL_{\BQ}(\sing X \times I)$ whose action on $\tS_{\BQ}(\sing X)$ sends $[\beta]$ to $[\beta'].$
For the same reason elements in the structure set that are in an orbit of the action of an element of the L-group have the same image under $\eta.$

\end{proof}

%%%%%%%%%%%
\section{K-theory classes associated to the signature operator on Witt and Cheeger spaces}\label{sect:classes}
%%%%%%%%%%%

%%%%%%%%%%%%%
\subsection{Metric structures}$\,$\\
%%%%%%%%%%%%%
In order to do analysis we endow a stratified space with a Riemannian metric. 
Let $\sing X$ be a smoothly stratified space and $\res X$ its resolution to a manifold with corners and an iterated fibration structure.

Recall from, e.g., \cite{ALMP:Witt, ALMP:Hodge, ALMP:Novikov}, that an {\bf iterated incomplete edge metric} (briefly, an iie-metric) is a Riemannian metric on the interior of $\res X$ (or, better, a bundle metric on the iterated incomplete edge tangent bundle over all of $\res X$) that in a collar neighborhood of each collective boundary hypersurface $\bhs{Y}$ takes the form
\begin{equation*}
	dx^2 + x^2g_Z + \phi_Y^*g_Y.
\end{equation*}
Here $x$ is a boundary defining function for $\bhs{Y},$ i.e., a smooth non-negative function on $\res X$ that is positive except at $\bhs{Y} = \{x=0\}$ where it vanishes to exactly first order, and $g_Z$ and $g_Y$ are metrics with the same structure on the spaces $Z$ and $Y.$ (Thus this is really an inductive definition over the depth of a stratified space, with spaces of depth zero being assigned smooth Riemannian metrics, see {\em loc. cit.}.)

In particular we point out that an iie-metric on $\sing X$ includes a Riemannian metric on each stratum of $\sing X$ and that these metrics fit together continuously (even smoothly in that they lift to a smooth section over $\res X$). Thus endowing $\sing X$ with an iie-metric in particular gives $\sing X$ the structure of a `Riemannian Whitney (A) space' in the sense of Pflaum \cite[\S2.4]{Pflaum}. (Note that the latter concept is more general, e.g., if we were working with metrics that were asymptotically of the form $dx^2 + x^{2\ell}g_Z + \phi_Y^*g_Y$ for any $\ell$ we would still get a `Riemannian Whitney (A) space'.) In particular, by Theorem 2.4.7 in \cite{Pflaum}, the topology on $\sing X$ is that of the metric space with distance between two points given by taking the infimum over rectificable curves joining them.
As a metric space, $\sing X$ is complete and locally compact \cite[Theorem 2.4.17]{Pflaum} and hence a `length space'.

%%%%%%%%%%%
\subsection{Galois coverings}$\,$\\
%%%%%%%%%%%
Let $\sing{X}$ be a smoothly stratified pseudomanifold of arbitrary depth. 
Consider a Galois covering $\pi: \cov{\sing{X}} \to \sing{X}$
with Galois group $\Gamma$ and  fundamental domain $\widehat{\mathscr{F}}_{\Gamma}$. There is a natural 
way to define a topological stratification on $ \cov{\sing{X}}$. 
Decompose $ \cov{\sing{X}}$ into the
preimages under $\pi$ of the strata in $ \sing{X}$. 
Surjectivity of $\pi$ ensures that each stratum in the covering is mapped 
surjectively onto the corresponding stratum in $\sing{X}$. 
Since $\pi$ is  a local homeomorphism, it is 
straightforward to check that  $\cov{\sing{X}}$ and 
its fundamental domain  are again topological 
stratified spaces. \medskip

In fact, more is true: by using these local homeomorphisms we can induce a \emph{smooth}
stratification on $
\cov{\sing{X}}$ by simply pulling it up from the base,  
in either the Whitney as well as the Thom-Mather cases.
It is important to point out that, by definition, the link of a point  
$\tilde{p}\in  \cov{\sing{X}}$ is
equal to the link of its image, $p$, in the base. This construction exhibits the covering map $\pi$ as a transverse map and thus {\em if $\sing X$ belongs to a class $\sC$ as above, then so does $\cov{\sing{X}}.$}
%\footnote{FROM PIERRE: I added this to replace the proposition  \begin{proposition}\label{prop:witt-cheeger-coverings}
% If $\sing{X}$ is a Witt space (respectively, a Cheeger space) then also $ \cov{\sing{X}}$ is a Witt space (respectively, 
% a Cheeger space)
%\end{proposition}
%} In particular, if $\sing X$ is a Witt space or a Cheeger space then so is $\cov{\sing X}.$

Needless to say,  if  $ \cov{\sing{X}}$ is the universal covering 
space of $ \sing{X}$, the individual strata 
in $\cov{\sing{X}}$ need 
not be the universal covering of the 
corresponding strata in the base. 
We denote by $\cov{\sm{X}}$ the regular stratum of 
$\cov{\sing{X}}$ and observe that it is a Galois covering of the regular stratum $\sm{X}$ of 
$\sing{X}$ with fundamental domain $\mathscr{F}_\Gamma$ equal to the regular part of $\widehat{\mathscr{F}}_{\Gamma}$.
Let $g$ be an admissible incomplete edge metric on $\sm{X}$. We can lift $g$ to the Galois covering
$\cov{\sm{X}}$ where it becomes a $\Gamma$-invariant admissible incomplete 
edge metric $\widetilde{g}$. 
Moreover, there is an isometric embedding of $\mathscr{F}_\Gamma$ into $\sm{X}$ with complement of measure 
zero. We denote  by $\De$ the signature operator on $\cov{\sm{X}}$ associated to such a metric. 
 
% \begin{proposition}\label{prop:witt-cheeger-coverings}
% If $\sing{X}$ is a Witt space (respectively, a Cheeger space) then also $ \cov{\sing{X}}$ is a Witt space (respectively, 
% a Cheeger space)
%\end{proposition}
%
%\begin{proof}
%For Witt spaces, this is immediate from the fact that the  link of $\widetilde{p}\in  \cov{\sing{X}}$ is equal to the link of $p:=\pi(\tilde{p})$ in $\sing{X}$. For Cheeger spaces, note that the pull-back of a mezzoperversity is again a mezzoperversity. Indeed, a non-Witt stratum 
%over each non-Witt stratum 
%
%
% we might have to look into a proper
%definition of a cocompact $\Gamma$-Cheeger space; in particular if $\cW$ is a self-dual mezzoperversity for $D$, i.e. adapted to $g$,
%then we shall need to talk about the  associated self-dual
%$\Gamma$-equivariant mezzoperversity $\cW_\Gamma$ for $\De$, i.e. adapted to $\tilde{g}$; I assume that this is the lift, in a suitable sense, of $\cW$; alternatively, we can talk about a $\Gamma$-equivariant self-dual
%mezzoperversity inducing a self-dual mezzoperversity on the $\Gamma$-compact quotient. 
%Let us leave all this aside for the time being.
%Notice that we expect  that the same statement will be true for any $\sC$-space (I am using your notation).
%\end{proof}
 
% In the 
% Witt case this is an essentially self-adjoint operator.

%%%%%%%%%%%
\subsection{$C^*$ and $D^*$ algebras}$\,$\\
%%%%%%%%%%%
First of all, we need to fix an Hilbert space $H$ with a unitary action of $\Gamma$ and a $C^*$-representation from $C_0 (\cov{\sing{X}})$
to $\sB (H)$ intertwining the two actions of $\Gamma$. Notice that the representation is associated to the {\em stratified} Galois covering $\cov{\sing{X}}$ (and not to its regular part $\cov{\sm{X}}$).
%\footnote{FROM PAOLO: This is similar to our analytic K-homology signature class in $K_* (\sing{X})$, with $\sing{X}$ a compact stratified Witt space: see our paper in AS ENS ({\bf we must take care of the corrigendum !!!}), Theorem 6.2. See also the corresponding treatment for Cheeger spaces.}
We take $H=L^2 (\cov{X}, \Lambda^* \cov{X})$;
%\footnote{FROM PAOLO: once again, this is precisely what we have done, but in the base of the covering, with the K-homology class for a Witt-space and 
%for a self-dual mezzoperversity}
 the representation is given by the multiplication operator associated to the restriction 
of a function to the regular part $\cov{\sm{X}}$. To these data we can associate two $C^*$-algebras:
the Roe algebra $C^* (\cov{\sing{X}},H)^\Gamma$, obtained as the closure of the $\Gamma$-equivariant finite propagation locally compact
bounded operators on $H$, and the Higson-Roe algebra 
$D^* (\cov{\sing{X}},H)^\Gamma$, obtained as the closure of the $\Gamma$-equivariant finite propagation pseudolocal bounded operators 
on $H$. Since we shall be eventually interested in the K-theory groups of these $C^*$-algebras and since the $K$-theory
groups are independent of the choice of the (adequate) $\Gamma$ equivariant $C_c (\cov{\sing{X}})$-module $H$, we shall 
adopt the notation $C^* (\cov{\sing{X}})^\Gamma$ and $D^* (\cov{\sing{X}})^\Gamma$ for these two $C^*$-algebras.

\footnote{For technical reasons having to
do with functoriality one actually takes $H=L^2 (\cov{X}, \Lambda^* \cov{X})\otimes \ell^2 (\bbN).$}

Then we have the following fundamental

\begin{proposition}\label{prop:classes-0}
Let $(\sing{X},g)$ and $(\cov{\sing{X}},\tilde{g})$ as above. Assume that $\sing{X}$
%and hence $\cov{\sing{X}}$ 
is a Cheeger-space and let $\cW$ be
a  self-dual mezzoperversity for $D$. Then:
\begin{itemize}
\item there exists a closed $\Gamma$-equivariant self adjoint extension of $\De$ associated to $\cW$, denoted $\De^{\cW}$.
\item if $\phi\in C_0 (\bbR)$, then $\phi (\De^{\cW})\in C^* (\cov{\sing{X}})^\Gamma$. 
\item if $\chi$ is a chopping function (i.e. $\chi:\bbR\to [-1,1]$ is odd and $\lim_{x\to \pm\infty} \chi(x)=\pm 1$)
then $\chi  (\De^{\cW})\in D^* (\cov{\sing{X}})^\Gamma$. 
\end{itemize}
\end{proposition}

\begin{proof}
(1) 
The pull-back of the mezzoperversity $\cW$ along the covering map $\pi$ is a mezzoperversity on $\cov X,$ which we briefly denote $\cov{\cW}.$ The definition of the domain associated to a mezzoperversity in \cite{ALMP:Hodge} applies in the setting of $\cov X,$ as the asymptotic expansions on which it relies are carried out in distinguished neighborhoods of points on the singular strata and these are the same on $\sing X$ or $\cov{\sing X}.$ Similarly we can define $\cD_{\cov{\cW}}(d)$ and $\cD_{\cov{\cW}}(\delta)$ as in \cite{ALMP:Hodge} and see that they are mutually adjoint and that
\begin{equation*}
	\cD_{\cov{\cW}}(\De^{\cW}) = \cD_{\cov{\cW}}(d) \cap \cD_{\cov{\cW}}(\delta)
\end{equation*}
so that $\De^{\cW}$ with this domain is self-adjoint.

The analysis of \cite{ALMP:Hodge} that establishes that $\cD_{\cW}(D)$ includes compactly into $L^2( X; \Lambda^*X)$ implies that, for any compact subset $K \subseteq \cov{\sing X},$
\begin{equation*}
	\{ u \in \cD_{\cov{\cW}}(\De^{\cW}) : \supp (u) \subseteq K \}
\end{equation*}
includes compactly into $H.$

(2)
As far as the second item is concerned we initially tackle the local compactness of $\phi  (\De^{\cW})$.
We have to prove that if $g\in C_c (\cov{\sing{X}})$ then $g \phi  (\De^{\cW})$ and $ \phi  (\De^{\cW})g$   are compact operators.
By taking adjoints it suffices to see that $g \phi  (\De^{\cW})$ is compact.
Using the Stone-Weistrass theorem it suffices to establish this property for the 
function $\phi(x)=(i+x)^{-1}$. As this maps $H$ into $\cD_{\cov{\cW}}(\De^{\cW}),$ the local compactness of the inclusion of the latter into $H$ implies that of $\phi(\De^{\cW}).$

Next we consider the finite propagation property: by a density argument it suffices to see such a property 
 for smooth functions $\phi$ that are of rapid decay and have compactly
supported Fourier transform. Thus, let $\hat{\phi}$ be the Fourier transform of a smooth rapidly decaying 
$\phi$ and let
us assume that the support of $\hat{\phi}$ is contained in $[-R/2,R/2]$. 
We must show that there exists $S\in\bbR^+$ such that
$f \phi  (\De^{\cW}) g=0$ whenever the distance between the support of $f$ and $g$ is greater than $S$.
Proceeding precisely as in  \cite[Theorem 5.3]{ALMP:Novikov} we know that there exists a $\delta$ such that
$\exp(i s \De^{\cW})$
has propagation $|s|$ if $|s| < \delta$; thus
$f \exp(is \De^{\cW}) g=0$ if the distance of the supports of $f$ and $g$ is greater than $|s|$, with $s$ in the range 
$(-\delta,\delta)$.  Recall now
that, by functional calculus, we can write
$$\phi  (\De^{\cW})=\frac{1}{2\pi}\int  \exp(i s \De^{\cW}) \hat{\phi} (s) ds $$
where the integral converges weakly:
$$\ang{ \phi  (\De^{\cW})u,v} =\frac{1}{2\pi}\int  \ang{\exp(i s \De^{\cW})u,v} \hat{\phi} (s) ds $$
for each pair of compactly supported sections on $\cov{\sm{X}}$.
Assume initially that $R<\delta$. Then, from the above integral representation, we see that $\phi  (\De^{\cW})$
has finite propagation (in fact, propagation $R$) which is what we wanted to prove. For the general case we use a trick from \cite{roe-foliation}. Write $\phi=\sum_j f_j$ where the sum is finite and where $f_j$ has Fourier transform supported in 
$(T_j-\delta/2, T_j+\delta/2)$. Consider $g_j (x)= \exp(-iT_j x)f_j (x) $. Then $g_j (\De^{\cW})$ has propagation  $\delta$
by what we have just seen. Write now $$\exp(iT_j x)=\prod_{\ell =1}^k \exp (i\tau_\ell x)\quad\text{with}\quad |\tau_\ell |<\delta\,.$$ We have then 
$f_j (x)= \prod_{\ell =1}^k \exp (i\tau_\ell x) g_j (x)$ and thus 
$$f_j (\De^{\cW} )= \exp(i\tau_1 \De^{\cW})\circ \cdots \circ \exp (i\tau_k\De^{\cW}) \circ g_j (\De^{\cW})\,.$$
All the operators appearing  on the right hand side have finite propagations and we know that the composition of two
operators of finite propagation is again of finite propagation. Thus $f_j  (\De^{\cW} )$ has finite propagation. The proof of item  2
is complete.

(3) Let $\chi$ be a chopping function. 
Recall from \cite[Section 10.6]{higson-roe-book} that for every $t>0$ there exists a chopping function $\chi$
with (distributional) Fourier transform 
supported in $(-t,t)$. Moreover, if $\chi_0$ and $\chi$ are two arbitrary chopping function, then $\chi_0-\chi=\phi$,
with $\phi\in C_0 (\bbR)$. This implies immediately that if $\chi_0 (\De^{\cW})$ is of finite propagation
then  $\chi_1 (\De^{\cW})$ is a limit of  finite propagation operators; moreover $\chi_0 (\De^{\cW})g-\chi (\De^{\cW})g$
is a compact operator for any $g\in C_c (\cov{\sing{X}})$. We choose a chopping function $\chi_0$ with Fourier transform supported in 
$(-\delta/2,\delta/2)$. Then we know, from the previous arguments, that $\chi_0 (\De^{\cW})$ is of  propagation $\delta$.
Hence $\chi (\De^{\cW})$ is a limit of  finite propagation operators for each chopping function $\chi$.
It remains to see that $\chi (\De^{\cW})$ is pseudolocal, i.e. $[f, \chi (\De^{\cW})]$
is compact for any $f \in C_0 (\cov{\sing{X}})$. 
By Kasparov's Lemma, see \cite[Lemma 5.4.7]{higson-roe-book} and \cite[Lemma 10.6.4]{higson-roe-book}, we know that $\chi (\De^{\cW})$ is pseudolocal if and only if $f \chi (\De^{\cW}) g$
is compact for any choice of $f\in C (\cov{\sing{X}})$ bounded  and $g\in C_c  (\cov{\sing{X}})$ with disjoint supports.
Now, if $\eta>0$ is the distance between 
the support of $f$ and the support of $g$ and if we choose a chopping function $\chi_0$ with Fourier transform 
supported in $(-\eta/2,\eta/2)$ then we know that   $f \chi_0 (\De^{\cW}) g=0$. But then for an arbitrary chopping function
$\chi$ we have 
$$f \chi (\De^{\cW}) g= f\chi_0 (\De^{\cW}) g + f (\chi (\De^{\cW}) -\chi_0 (\De^{\cW}) )g = 0 + f (\chi (\De^{\cW}) -\chi_0 (\De^{\cW}) )g$$
and since $f$ is bounded and $(\chi (\De^{\cW}) -\chi_0 (\De^{\cW}) )g$
is compact we see that on the right hand side we do have a compact operator as required. The proof of item 3 is now complete.

\end{proof}

\subsection{K-homology classes}$\,$\\
Proposition \ref{prop:classes-0} allows us to recover, in the bounded picture, the fundamental classes that were defined in \cite{ALMP:Witt}
for Witt spaces  and in 
\cite{ALMP:Novikov} for Cheeger spaces. More precisely:

\begin{proposition}\label{prop:k-homology-class}
If $\sing{X}$ is an $n$-dimensional  Cheeger space endowed with a rigid iterated conic metric $g$ and if $\cW$ is a
self-dual mezzoperversity adapted to $g$ then there is a well defined $K$-homology signature class $[ D^{\cW} ]\in K_n (\sing{X})$.
\end{proposition}
\begin{proof}
Let $\chi$ be a chopping function; then  $\chi^2 -1 $ is an element in $C_0 (\bbR)$
and thus  $\chi  (\De^{\cW})$
is an involution in the quotient $D^* (\cov{\sing{X}})^\Gamma/C^* (\cov{\sing{X}})^\Gamma$. Thus, using also
the grading in even dimension, one defines an element
in $K_{n+1} (D^* (\cov{\sing{X}})^\Gamma/C^* (\cov{\sing{X}})^\Gamma)$ which is precisely $K_{n} (\sing{X})$
by Paschke duality.
\end{proof}

\begin{remark}\label{remark:equality-of-indeces}
The class we have just defined does coincide with the one defined
in \cite{ALMP:Novikov}: this follows from the proof of \cite[Theorem 10.6.5]{higson-roe-book}
and the correspondence between the unbounded and bounded picture for K-homology.
\end{remark}

\begin{remark}\label{remark:independence0}
The class $[ D^{\cW} ]_{\bbQ}\in K_n (\sing{X})\otimes \bbQ$ is independent of the choice of self-dual mezzoperversity
$\cW$; indeed the homological Chern character of $[ D^{\cW} ]_{\bbQ}$, in $H_* (\sing X,\bbQ)$ is equal to
the $L$-class of the Cheeger space, see \cite[Thorem 5.6]{ALMP:Novikov},  and  we know that the $L$-class
is  independent of the choice of $\cW$, see \cite[Section 5.1]{ALMP:Novikov}. We shall come back to this point later on.
\end{remark}

\subsection{Higson-Roe sequences associated to a Thom-Mather space}$\,$\\
If $\sing{X}$ is a Thom-Mather stratified space and $\cov{\sing X}$ is a Galois covering with structure group $\Gamma$,  then there is a  short exact sequence of $C^*$-algebras 
$$0\rightarrow  C^*(\cov{\sing X})^\Gamma\rightarrow 
D^*(\cov{\sing X})^\Gamma\rightarrow 
D^*(\cov{\sing X})^\Gamma/C^*(\cov{\sing X})^\Gamma\rightarrow 0$$ and thus a 6-term long exact sequence in K-theory:
\begin{equation}\label{eq:HR-stratified}
\cdots\rightarrow K_{m+1}  (C^*(\cov{\sing X})^\Gamma)
\rightarrow K_{m+1}  (D^*(\cov{\sing X})^\Gamma)
\rightarrow  K_{m+1}  (D^*(\cov{\sing X})^\Gamma/C^*(\cov{\sing X})^\Gamma)
\rightarrow K_m (C^*(\cov{\sing X})^\Gamma)\rightarrow \cdots
\end{equation}
%$$\cdots\rightarrow K_{n+1}  (C^*(\widetilde{V})^\Gamma)
%\rightarrow K_{n+1}  (D^*(\widetilde{V})^\Gamma)
%\rightarrow K_{n+1}  (D^*(\widetilde{V})^\Gamma/C^*(\widetilde{V})^\Gamma)
%\xrightarrow{\pa} K_n (C^*(\widetilde{V})^\Gamma)\rightarrow \cdots$$
This is the {\em analytic surgery sequence} of Higson and Roe associated to the $\Gamma$-compact 
$\Gamma$-space $\cov{\sing X}$.
Since we have  the canonical isomorphism 
$K_{*+1}  (D^*(\cov{\sing X})^\Gamma/C^*(\cov{\sing X})^\Gamma)= K_{*} (\sing X)$ 
we can also rewrite \eqref{eq:HR-stratified} as 
\begin{equation}\label{eq:HR-stratified-bis}
\cdots\rightarrow K_{m+1}   (C^*(\cov{\sing X})^\Gamma)
\rightarrow K_{m+1}  (D^*(\cov{\sing X})^\Gamma)
\rightarrow  K_{m}  (\sing X)
\rightarrow K_m (  C^*(\cov{\sing X})^\Gamma )\rightarrow \cdots
\end{equation}
Moreover, recall that we have a canonical isomorphism  $K_{*}  (C^*(\cov{\sing X})^\Gamma)=K_* (C^*_r \Gamma)$.

Now, in particular, all of the above is true with $\cov{\sing X}$ equal to the universal covering of $\sing X$
and $\Gamma=\pi_1 (\sing X)$. Consider now a closed stratum $\sing Y$ of $\sing X$: thus if $Y$ is a stratum of
$\sing X$ then 
\begin{equation*}
	\sing Y:= \overline{Y}= \bigcup \{ \sm Y_i : \sm Y_i \leq \sm Y \}
\end{equation*}
Consider $\Gamma (\sing Y):=\pi_1 (\sing Y)$; then we have a 6-term exact sequence similar to \eqref{eq:HR-stratified-bis}
but associated to the universal covering, $\Gamma (\sing Y) -  \sing Y_{\Gamma (\sing Y)} \to \sing Y$ of $\sing Y$.

%\footnote{Maybe our notation is not so good !!}

\subsection{Index  classes}\label{subsect:index-classes}$\,$\\
Let now $\sing X$ be an $n$-dimensional  Cheeger space and let us choose a self-dual mezzoperversity $\cW$.
Then, by Proposition \eqref{prop:k-homology-class}, we have a K-homology class 
$[ D^{\cW} ]\in K_n (\sing{X})=K_{n+1} (D^* (\cov{\sing{X}})^\Gamma/C^* (\cov{\sing{X}})^\Gamma)$ and thus an index  class 
\begin{equation}\label{eq:index-class}\Ind (\De^{\cW}):=\partial [ D^{\cW} ]\in K_{n} (C^*(\cov{\sing X})^\Gamma),
\end{equation} with $\partial$ the connecting homomorphism in the Higson-Roe surgery sequence.
Following the proof given in \cite[Proposition 2.1]{PS-Stolz} this class corresponds to the one considered
in \cite{ALMP:Hodge} through the canonical isomorphism $K_{*}  (C^*(\cov{\sing X})^\Gamma)=K_* (C^*_r \Gamma)$; notice that the class
defined in  \cite{ALMP:Hodge} is a Mishchenko class, obtained by twisting the signature operator by the Mishchenko
bundle $\sing X_\Gamma \times_\Gamma C^*_r \Gamma$.

Both in the Higson-Roe formalism, see \cite{higson-roe4}, and in the Mishchenko formalism, we can also
consider the index class with values in the maximal version of our $C^*$-algebras. We denote the maximal group $C^*$-algebra
associated to $\Gamma$ as $C^*\Gamma$.

Finally, if $\sing Y$ is a closed $m$-dimensional stratum with fundamental group $\Gamma ({\sing Y})$ then $\cW$ induces a self-dual mezzoperversity $\cW(\sing Y)$ 
for $\sing Y$ 
 and we obtain a K-homology class  $[D^{\cW(\sing Y)}]\in K_m (\sing Y)$
and thus an Index class 
$\Ind(D^{\cW(\sing Y)}_{\Gamma(\sing Y)})\in K_m (C^*_r (\Gamma(\sing Y))$. %\footnote{Help ! Too heavy  notation !}

\subsection{Rho classes associated to trivializing perturbations}$\,$\\
Let $\sing X$ be a Cheeger space endowed with an $\operatorname{iie}$ metric $g$.
We initially assume that $\sing X$ is odd dimensional. Let $\cW$ be a
self-dual mezzoperversity for $\sing X$ and let $D^{\cW}$ be the corresponding signature operator,
an unbounded self-adjoint operator on $L^2 (\sm{X},\Lambda^* \sm{X})$.
(Recall that $D^{\cW}$  is a short notation for the pair $(D,\cD_{\cW}(D))$, the (extension of the) signature operator on $(\sm{X},g)$, the regular part
of $\sing X$ endowed with the Riemannian metric $g$, with domain defined by
the self-dual mezzoperversity $\cW$.) Given a Galois $\Gamma$-covering $\cov{\sing X}$ of $\sing X$,
we also have the $\Gamma$-equivariant signature operator $\De^{\cW}$, a self-adjoint unbounded operator 
on $L^2 (\cov{X}, \Lambda^* \cov{X})$.
Let now $A$ be a bounded $\Gamma$-equivariant self-adjoint operator  on $L^2 (\cov{X}, \Lambda^* \cov{X})$.
Then $\De^{\cW}+A$, with domain equal to the domain of $\De^{\cW}$, is also
self-adjoint. Following \cite[Section 2B]{PS-akt} we make the assumption that $\De^{\cW}+A$
is $L^2$-invertible and that $A\in \mathfrak{M}(C^* (\cov{\sing X})^\Gamma)$,
the multiplier algebra of $C^* (\cov{\sing X})^\Gamma$. We refer to $A$ as a {\em trivializing perturbation}.
Then, using Proposition \ref{prop:classes-0} and 
\cite[Proposition 2.8]{PS-akt}, we see that 
\begin{equation}\label{rho-0}
\frac{\De^{\cW}+A}{|\De^{\cW}+A|} \;\;\text{is an element in}\;\; D^* (\cov{\sing X})^\Gamma\,.
\end{equation}
Moreover, $\De^{\cW}+A/|\De^{\cW}+A|$  is clearly an involution and thus 
$$ \frac{1}{2}\left( \frac{\De^{\cW}+A}{|\De^{\cW}+A|}  + 1 \right)$$
is a projection in $D^* (\cov{\sing X})^\Gamma$. We define the rho class 
associated to $\De^{\cW}+A$ as 
\begin{equation}\label{rho-odd}
\rho(\De^{\cW}+A):= \left[ \frac{1}{2}\left( \frac{\De^{\cW}+A}{|\De^{\cW}+A|}  + 1 \right) \right]\;\;\text{in}\;\; K_0 (D^* (\cov{\sing X})^\Gamma)\,.
\end{equation}
In the even dimensional case we consider the grading associated to the Hodge $\star$ operator and we demand
that the trivializing perturbation $A\in \mathfrak{M}(C^* (\cov{\sing X})^\Gamma)$ be odd with respect to this grading; thus $\De^{\cW}+A$
can be written as 
$$
    \begin{pmatrix}
      0 & \De^{\cW,-} + A^-\\  \De^{\cW,+} + A^+ & 0
    \end{pmatrix}
    $$
    We now fix a chopping function $\chi$ equal to the sign function on the spectrum of the invertible operator 
    $\De^{\cW}+A$; we also fix a $\Gamma$-equivariant isometry  $u: \Lambda_- (^{\operatorname{iie}}T^* X_\Gamma)
    \to \Lambda_+ (^{\operatorname{iie}}T^* X_\Gamma) $ and consider the induced bounded $\Gamma$-equivariant
    operator on the space of $L^2$ sections of these bundles, call it $U$. Observing that $\chi (\De^{\cW}+A)$ is also
    odd (see \cite[Lemma 10.6.2]{higson-roe-book}) we consider $U\chi( \De^{\cW}+A)_+$ which is a unitary
    in $D^* (\cov{\sing X})^\Gamma$. We then define the rho class in the even dimensional case 
    an 
    \begin{equation}\label{rho-even}
\rho(\De^{\cW}+A):= [U\chi( \De^{\cW}+A)_+]\;\;\text{in}\;\; K_1(D^* (\cov{\sing X})^\Gamma)
\end{equation}
As explained in \cite[page 118]{PS-akt} this is  well defined, independent of the choice of $u$. 

\section{Bordisms and associated K-theory classes}\label{sect:bordism}

\subsection{Bordisms of Cheeger spaces}$\,$\\
We recall here some fundamental facts established in \cite{ALMP:Novikov}. Assume that 
 $M$ is a topological space and denote by $\operatorname{Sig}_{n} (M)$ the bordism group of four-tuples 
 $(\sing X, g, \cW, r:\sing X \lra M)$ where $\sing X$ is an oriented Cheeger space of dimension $n,$ 
 $g$ is an adapted iterated incomplete conic metric  metric (briefly an $\operatorname{iie}$ metric), $\cW$ is a self-dual Hodge mezzoperversity adapted to $g$  and $r:\sing X \lra M$ is a continuous map. An admissible bordism between $(\sing X, g, \cW, r:\sing X\lra M)$ and $(\sing X',g',\cW', r':\sing X' \lra M)$ is a four-tuple $(\sing{\sX}, G, \sW, R: \sing{\sX} \lra M)$  consisting of:\\
i) a smoothly stratified, oriented, compact pseudomanifold with boundary $\sing{\sX}$, whose boundary is $\sing X \sqcup (-\sing X'),$
and whose strata near the boundary are collars of the strata of $\sing X$ or $\sing X',$\\
ii) an $\operatorname{iie}$ metric $G$ on $\sing{\sX}$ that near the boundary
is of the collared form $dx^2 +g$ or $dx^2 + g',$ \\
iii) an adapted self-dual mezzoperversity $\sW$ that extends, in a collared way, that of $\sing X$ and $\sing X',$\\
iv) a map $R: \sing{\sX} \lra M$ that extends $r$ and $r'.$\\
%\end{definition}
We shall briefly say that $(\sing X, g, \cW, r:\sing X\lra M)$ and $(\sing X',g',\cW', r':\sing X' \lra M)$ are Cheeger-bordant
through $(\sing{\sX}, G, \sW, R: \sing{\sX} \lra M)$.
We are mainly interested in the case $M=B\Gamma$, so that a map $r:\sing X\to B\Gamma$ defines a Galois $\Gamma$-covering $\sing X_\Gamma$.
We  have the following important results, see \cite{ALMP:Novikov} for proofs:

\begin{theorem}\label{theo:main-bordism}
If $(\sing X,g,\cW,r)$ and $(\sing X',g',\cW',r')$ are $n$-dimensional and Cheeger-bordant through
$(\sing{\sX}, G, \sW, R: \sing{\sX} \lra M)$ then:
\item{1]} the numeric Fredholm indices associated to $D^{\cW}$ and $(D^\prime)^{\cW'}$ are equal;
\item{2]} there exists a well defined relative K-homology class $[D^{\sW}]\in K_{n+1} (\sing{\sX},\partial \sing{\sX})$;
\item{3]} if $\partial:  K_{n+1} (\sing{\sX},\partial \sing{\sX})\to K_n (\partial \sing{\sX})\equiv K_n (\sing X\cup (-\sing X'))$
is the connecting homomorphism associated to the long exact sequence of the pair $(\sing{\sX},\partial\sing{\sX})$ then 
$$\partial [D^{\sW}]= [D^{\cW}]-[(D^\prime)^{\cW^\prime}]\;\;\text{ in }\;\;K_n (\partial \sing{\sX})\otimes_{\bbZ} \bbZ[\frac{1}{2}];$$
\item{4]} the  signature
index classes  associated to $(\sing X, g, \cW, r:\sing X\lra B\Gamma)$ and $(\sing X',g',\cW', r':\sing X' \lra B\Gamma)$ are equal 
 in $K_* (C^*_r \Gamma)\otimes_{\bbZ} \bbZ[\frac{1}{2}]$.
 \item{5]}  If $\cW$ and $\cW'$ are adapted  to $g$ and $g'$ on the same Cheeger space $ \sing X$ and $r:X\to M$
 is a continuous map  then 
$ (\sing X,g,\cW,r:\sing X\to M)$ is Cheeger-bordant to $(\sing X,g',\cW',r:\sing X\to M)$. In particular, the numeric Fredholm index, in $\bbZ$,
and the signature index
class, in $K_* (C^*_r \Gamma)\otimes_{\bbZ} \bbZ[\frac{1}{2}]$, are independent of the choice of self-dual mezzoperversity.
\end{theorem}

\begin{remark}\label{rmk:12Enough}
The statements in \cite{ALMP:Novikov} are given with values in 
$K_0(C^*\Gamma)\otimes_{\bbZ} \bbQ$ but it is easy to see that the arguments given there establish the same results in $K_0(C^*\Gamma)\otimes_{\bbZ} \bbZ[\tfrac{1}{2}]$.
\end{remark}

\medskip
\noindent
The main idea behind the formulation and the proof of item 5] is due to Markus Banagl, see \cite{banagl-annals}.
%The independence of the homology L-class in $H_*(\sing, \bbQ)$ on the choice of mezzo-pervesity follows from item 5.

\smallskip
\noindent
{\bf Notation.} Since the signature index class $\Ind(D_\Gamma^{\cW})\in K_* (C^*_r \Gamma)\otimes_{\bbZ} \bbZ[\frac{1}{2}]$ associated to a Galois covering 
$r:\sing X\to B\Gamma$ is in fact independent of the cloice of $\cW$, we shall often denote it simply by
$\Ind(D_\Gamma)$ or even $\Ind(\sing X_{\Gamma})$.

\subsection{The signature operator on Cheeger spaces with cylindrical ends}$\,$\\
Let  $\Gamma - \cov{\sing{\sX}}\rightarrow \sing{\sX}$ be a Galois $\Gamma$-covering of an even dimensional  Cheeger space with boundary.
We consider a rigid $\operatorname{iie}$ metric $g$ on the regular part $\sX$ which is collared near $\pa \sX$
and we lift it to a $\Gamma$-equivariant rigid $\operatorname{iie}$ metric on $\cov{\sX}$.
We also consider
the Cheeger spaces with cylindrical ends, $\sing{\sX}_\infty$, $ \sing{\sX}_{\Gamma,\infty}$, obtained by attaching $(-\infty,0]\times \pa\sing{\sX}$ and $(-\infty,0]\times \pa \cov{\sing{\sX}}$ to 
$\sing{\sX}$, $ \sing{\sX}_{\Gamma}$ respectively. 
We endow $(-\infty,0]\times \pa\sX$ and $(-\infty,0]\times \pa \cov{\sX}$ with product metrics and we obtain in this way 
 global metrics on $\sX_\infty$ and  $\sX_{\infty,\Gamma}$.
If $\sW$ is a self-dual mezzoperversity on $\sX$ then we obtain in a natural way a self-dual mezzoperversity
$\sW_\infty$ on $\sm{X}_\infty$ and thus, by lifting, a $\Gamma$-equivariant self-dual mezzoperversity 
$\sW_{\Gamma,\infty}$ on  $\sm{X}_{\infty,\Gamma}$. 
We denote by $\pa \sW$ the self-dual mezzoperversity induced on the boundary, see Theorem \ref{theo:main-bordism}.
Finally, we denote by $P_0$ the multiplication operator by the characteristic function of the attached semi-cylinder.

\begin{proposition}\label{prop:cylindrical-domain}
Let $D_\infty$ and $D_{\infty,\Gamma}$ be the  the signature operators on $\sm{X}_\infty$ and  $\sm{X}_{\infty,\Gamma}$
respectively. By employing $\sW_\infty$  and 
$\sW_{\Gamma,\infty}$ we can define self-adjoint extensions $D_\infty^{\sW}$ and $D_{\infty,\Gamma}^{\sW}$
\end{proposition}

\begin{proof}
Extend the iterated fibration structure from $\sm X$ to $\sm X_\infty$ by including the cylindrical direction in the base of each fiber bundle. Define
\begin{multline*}
	D_{\infty}^{\sW} = (D_{\infty}, \cD_{\sW_\infty}(D_{\infty})), \;
	\Mwhere
	\cD_{\sW,\infty}(D) = \{ u\in \sD_{\max}(D_{\infty}) : \text{ at each singular stratum, }\\
	\text{ $u$ satisfies the ideal boundary condition corresponding to $\sW_{\infty}$} \}.
\end{multline*}
It is easy to see that this is a self-adjoint domain.
Indeed, this domain is localizable (an element is in the domain if and only if it is in the domain after multiplying by any function in $\CI_{\Phi}(\sm X),$ see \cite[\S2]{ALMP:Hodge} and, e.g., the discussion after assumption 3.8 in \cite{ALMP:Hodge}.)
 and so it suffices to show that the corresponding domain on the full cylinder $\pa \sm X \times \bbR$ is self-adjoint; here we could either use Fourier transform in the $\bbR$-factor to reduce to the self-adjointness of $D^{\pa\sW},$ or alternately recognise $\pa \sm X \times \bbR$ as a cover of $\pa \sm X \times \bbS^1,$ consider the pull-back of $\pa\sW$ to this product, and then appeal to Proposition \ref{prop:classes-0} above.

\end{proof}

\subsection{Perturbations and coarse APS-index classes}$\,$\\
In this subsection we shall define APS-index classes associated to the self-adjoint operator $D_{\infty,\Gamma}^{\sW}$.
%As in \cite{PS-Stolz} (and then \cite{PS-akt}) we shall first give a coarse description of the APS index class and then discuss
%an equivalent $b$-version of it.
%
%\medskip
%\noindent
%{\bf The coarse APS index class.} 
We consider the $C^*$-algebra
$C^* (\sing{\sX}_\Gamma\subset \sing{\sX}_{\Gamma,\infty})^\Gamma$%\,,\;\;D^* (\cov{\sing{\sX}}\subset \cov{\sing{\sX}}^\infty)\;;
obtained by closing the  operators $T$ in  $C^*_c (\sing{\sX}_{\Gamma,\infty})^\Gamma$ with 
the additional property that $\exists$ $R>0$ such that $\phi T=0=T\phi$ whenever $\phi\in C_c (\sing{\sX}_{\Gamma,\infty})$
and $d(\operatorname{supp}(\phi), \sing{\sX}_\Gamma)>R$.
See for example \cite[Definition 1.7]{PS-Stolz}. A similar definition can be given for 
$D^* (\sing{\sX}_\Gamma\subset \sing{\sX}_{\Gamma,\infty})^\Gamma$. One can prove that the inclusion $c:
\sing{\sX}_\Gamma\hookrightarrow\sing{\sX}_{\Gamma,\infty}$ induces K-theory isomorphisms:
\begin{equation}\label{relative-higson-roe}
K_* (C^* (\sing{\sX}_\Gamma)^\Gamma)\simeq K_* (C^* (\sing{\sX}_\Gamma\subset \sing{\sX}_{\Gamma,\infty})^\Gamma)\,;\quad
K_* (D^* (\sing{\sX}_\Gamma)^\Gamma)\simeq K_* (D^* (\sing{\sX}_\Gamma\subset \sing{\sX}_{\Gamma,\infty})^\Gamma)\,.
\end{equation}
Notice that $C^* (\sing{\sX}_\Gamma\subset \sing{\sX}_{\Gamma,\infty})^\Gamma)$ and  $D^* (\sing{\sX}_\Gamma\subset \sing{\sX}_{\Gamma,\infty})^\Gamma$ are ideals in 
$D^* (\sing{\sX}_{\Gamma,\infty})^\Gamma$.

We assume the existence of a trivializing perturbation $C_\pa$  for the signature operator 
 $D_{\Gamma}^{\pa \sW}$ on $\pa \sX_\Gamma$: this means, as before, that $C_\pa$ is bounded,
 that $D_{\Gamma}^{\pa \sW}+C_\pa$ (with domain equal to the domain of $D_{\Gamma}^{\pa \sW}$)
 is $L^2$-invertible and that 
 $C_\pa\in \mathfrak{M}(C^* (\pa\sing{\sX}_\Gamma)^\Gamma)$. $C_{\pa}\otimes \Id_{\bbR}$ then defines
 a bounded operator on $L^2 (\pa\sing{\sX}_\Gamma\times \bbR)$.
We can then graft this perturbation on $ \sX_{\Gamma,\infty}$ and obtain a bounded perturbation $C_\infty$
for $D_{\infty,\Gamma}^{\sW}$. 
In the case of interest to us it will be the case that $C_\infty$
is a limit of finite propagation operators  %\mathfrak{M}(C^* (\sing{\sX}_\Gamma)^\Gamma)$
and so we assume this property in what follows.
In fact, we might more generally consider a perturbation $B_\infty$ which is  a limit of finite propagation operators and such that 
\begin{equation}\label{perturbation-B}
% B_\infty \in  %\mathfrak{M}(C^* (\sing{\sX}_{\Gamma,\infty})^\Gamma)
% C^* (\sing{\sX}_\Gamma)^\Gamma
 %\;\;\;\text{and}\;\;\; 
 P_0 B_\infty P_0 - P_0 C_\infty P_0 \in C^* (\sing{\sX}_\Gamma\subset \sing{\sX}_{\Gamma,\infty})^\Gamma\,. \end{equation}

\begin{proposition}\label{prop:relative}
Let $C_{\pa}$, $C_\infty$ and $B_\infty$ be as above.
If $\phi\in C_0 (\bbR)$ and if 
$\chi$ is a chopping function equal to the sign function on the spectrum of $D_{\Gamma}^{\pa \sW} + C_\pa$, then:
\item{1]} $\phi  (D_{\Gamma,\infty}^{\sW}+B_\infty)\in C^* ( \sing{\sX}_{\Gamma,\infty})^\Gamma$;
\item{2]} $\chi (D_{\Gamma,\infty}^{\sW}+B_\infty)\in D^* ( \sing{\sX}_{\Gamma,\infty})^\Gamma$;
\item{3]}  $\chi (D_{\Gamma,\infty}^{\sW}+B_\infty)$ is an involution modulo $C^* (\sing{\sX}_\Gamma\subset \sing{\sX}_{\Gamma,\infty})^\Gamma$.
\end{proposition}

\begin{proof}
For 1] and 2] we use Proposition \ref{prop:classes-0} and the purely functional analytic arguments given in 
\cite[Lemma 2.25]{PS-akt}.
For the third item we use the proof 
of Proposition 2.26 in \cite{PS-akt}, which is once again purely functional analytic.
%
%; Proposition 2.26
%  is stated in \cite{PS-akt} under the  assumptions that $C_\pa$
%is a norm limit of finite propagation operators, $B_\infty$ is a norm limit of finite propagation operators
%and, as above,  $P_0 B_\infty P_0 - P_0 C_\infty P_0 \in C^* (\sing{\sX}_\Gamma\subset \sing{\sX}_{\Gamma,\infty})^\Gamma$
%and this assumption is certainly satisfied by elements in $C^* (\sing{\sX}_\Gamma)^\Gamma$.
%%However, an inspection of the proof shows that the same arguments work under the more general hypothesis 
%%stated here.

\end{proof}

Given $C_\pa$ as above, choosing $B_\infty=C_\infty$ and using Proposition \ref{prop:relative}
we can define a coarse relative index class
\begin{equation*}%\label{coarse-rel-index}
\Ind^{\operatorname{rel}}( D_{\Gamma,\infty}^{\sW}+C_\infty):= \pa [\chi (D_{\Gamma,\infty}^{\sW}+C_\infty)]\in K_* (C^* (\sing{\sX}_\Gamma\subset \sing{\sX}_{\Gamma,\infty})^\Gamma)
\end{equation*}
and thus, using \eqref{relative-higson-roe}, a coarse APS-index class 
\begin{equation*}%\label{coarse-rel-index}
\Ind( D_{\Gamma}^{\sW},C):= c_*^{-1} (\Ind^{\operatorname{rel}}( D_{\Gamma,\infty}^{\sW}+C_\infty))\in K_* (C^* (\sing{\sX}_\Gamma)^\Gamma)\simeq K_* (C^*_r \Gamma)
\end{equation*}
Notice that the left hand side is just notation; we have not really defined a perturbation $C$ on $\sing{\sX}_\Gamma$.\\
One can prove, following the arguments in the proof of \cite[Proposition 2.33]{PS-akt}, that for $C_\pa$, $C_\infty$
and $B_\infty$ as above:
\begin{equation*}%\label{coarse-rel-index=}
\Ind^{\operatorname{rel}}( D_{\Gamma,\infty}^{\sW}+C_\infty)= \Ind^{\operatorname{rel}}( D_{\Gamma,\infty}^{\sW}+B_\infty)
\in K_* (C^* (\sing{\sX}_\Gamma\subset \sing{\sX}_{\Gamma,\infty})^\Gamma)
\end{equation*}
where the right hand side is well defined because of item 3] of Proposition \ref{prop:relative}.\\

%\begin{itemize}
%\item define $b$-index class in the even case and state equality with coarse APS index class.
%\item define odd $b$-index class by suspension; state equality with Dirac suspended family ; prove equality
%with coarse Dirac suspension
%\item look at the compatibility of KK-description of index classes and coarse index classes
%\end{itemize}

\subsection{The delocalized APS index theorem}$\,$\\
Let $\sX$, $\sX_\Gamma$, $\sX_\infty$, $\sX_{\Gamma,\infty}$, 
$D^{\sW}$, $D_\Gamma^{\sW}$, $D_\infty^{\sW}$, $D_{\Gamma,\infty}^{\sW}$,
$D^{\pa \sW}$, $D^{\pa \sW}_{\Gamma}$ and $C_{\pa}$  as in the previous subsections. We assume $\sX$ to be even dimensional.
By assumption $C_\pa$
is a trivializing perturbation  for $D_{\Gamma}^{\pa \sW} $; assume that  $C_\pa\in C^* ({\pa \sW})^\Gamma)$, so that
 $C_\pa$ is a norm limit of finite propagation operators. Consequently $C_\infty$ is also a norm limit of finite propagation operators. 
We can consider 
the rho class $\rho (D_{\Gamma}^{\pa \sW} + C_\pa)\in K_0 (D^* (\pa \sX_{\Gamma})^\Gamma)$
and the coarse-APS index class $\Ind( D_{\Gamma}^{\sW},C)\in K_0 (C^* (\sX_{\Gamma})^\Gamma)$.
Let $\iota: C^* (\sX_{\Gamma})^\Gamma \hookrightarrow D^* (\sX_{\Gamma})^\Gamma$ be the natural inclusion and consider 
$j_* : K_0 (D^* (\pa \sX_{\Gamma})^\Gamma))\to K_0 (D^* ( \sX_{\Gamma})^\Gamma)$
induced by the inclusion of $\pa \sX_{\Gamma}$ into $\sX_{\Gamma}$. Our main tool in the next section will be the 
{\it delocalized APS index theorem
for  perturbed signature operators on Cheeger spaces}:

\begin{theorem}
\label{theo:k-theory-deloc} 
If the trivializing perturbation $C_\pa$ is a norm limit of finite propagation
operators, then 
the following equality holds
\begin{equation}\label{k-theory-deloc}
\iota_* ( \Ind( D_{\Gamma}^{\sW},C))= j_*(\rho( D_{\Gamma}^{\pa \sW} + C_{\pa})) \quad\text{in}\quad K_{0}
(D^*(  \sX_{\Gamma})^\Gamma).
\end{equation}
\end{theorem}

\begin{proof}
All the arguments given in  \cite{PS-Stolz} and then \cite{PS-akt} are functional analytic with the exception
of the proof of Proposition 5.33 in  \cite{PS-akt}. However, the alternative proof of this particular Proposition given by Zenobi, see
Proposition 3.20 in  \cite{Zenobi:Mapping} , applies verbatim to the present context.
\end{proof}

Let now  $\sX$  be odd  dimensional. After inverting 2 we can reduce the delocalized  APS index theorem 
on $\sX$ to the one on $\sX\times S^1$  by a suspension argument. This is discussed carefully in \cite[\S5]{Zenobi:Mapping} where a different description of the group $\tK_*(D^*(\sing X_{\Gamma})^{\Gamma})$ is given for metric spaces with $\Gamma$-actions. These arguments apply in our situation largely unchanged.

%%%%%%%%%%%%%
\section{Stratified homotopy equivalences and associated perturbations}\label{sect:HS}
%%%%%%%%%%%%%

%%%%%%%%%%%%%
\subsection{The Hilsum-Skandalis replacement}$\,$\\
%%%%%%%%%%%%%
Let $\sing X$ be a Cheeger space, $r: \sing X \lra B\Gamma$ the classifying map for the universal cover of $\sing X,$ $\sG(r)$ the Mishchenko bundle associated to $r,$ and $\cW_X$ a self-dual mezzoperversity on $\sing X.$
If $\sing M$ is another Cheeger spaces and $f: \sing M \lra \sing X$ a stratified homotopy equivalence then (see \cite[Theorem 4.6]{ALMP:Novikov}), there is a  `Hilsum-Skandalis replacement' for the pull-back of differential forms by $f,$
\begin{equation*}
	HS(f): 
	L^2(X; \Lambda^* {}^{\iie}T^*X \otimes \sG(r)) \lra
	L^2(M; \Lambda^* {}^{\iie}T^*M \otimes \sG(f\circ r)),
\end{equation*}
that we can use to define a self-dual mezzoperversity $\cW_M=f^{\sharp}(\cW_X)$ on $\sing M.$ These data satisfy
\begin{itemize}
\item $HS(f)d^{\sG(r)} = d^{\sG(f\circ r)}HS(f)$ and 
	$HS(f)(\cD_{\cW_X}(d^{\sG(r)})) \subseteq \cD_{\cW_M}(d^{\sG(f\circ r)}),$
\item There is an $L^2$-bounded operator $\Upsilon$ acting on $\cD_{\cW_X}(d^{\sG(r)}),$ such that
\begin{equation*}
	\Id - HS(f)'HS(f) = d_{\sG(r)}\Upsilon + \Upsilon d_{\sG(r)},
\end{equation*} 
where $HS(f)'$ denotes the adjoint with respect to the quadratic form defined
by the Hodge operator.
\end{itemize}
We point out that the boundedness of $HS(f)$ on $L^2(X; \Lambda^* {}^{\iie}T^*X \otimes \sG(r)),$ together with the first of these properties, implies that $HS(f)$ is bounded as a map
\begin{equation*}
	HS(f): \cD_{\cW_X}(d^{\sG(r)}) \lra \cD_{\cW_M}(d^{\sG(f\circ r)}),
\end{equation*}
when these spaces are endowed with the respective $d$-graph norm.
Similarly $\Upsilon$ is bounded as an operator on the Hilbert space $\cD_{\cW_X}(d^{\sG(r)}).$
Note however that $HS(f)$ does not map $L^2$ differential forms into the maximal domain of $d;$ indeed, if a differential form extends to be smooth on the closure of $\res X$ and its exterior derivative fails to be in $L^2,$ then the same will be true of its image under $HS(f).$

%%%%%%%%%%%%%
\subsection{The compressed Hilsum-Skandalis replacement}$\,$\\
%%%%%%%%%%%%%
Following \cite{PS1}, we will also make use of a compressed version of the Hilsum-Skandalis replacement. In this case the replacement will make use of a fixed mezzoperversity and will have the property that it maps all of the $L^2$ differential forms into the domain of $d.$
Recall that one of the main
results in \cite{ALMP:Novikov} is that  the resolvents
of $D^{\cW}_{\cG(r)}$ and $D^{f^{\sharp}\cW}_{G(r\circ f)}$ are $C^*_r \Gamma$-compacts.

\begin{definition}
Let $\sing X$ be a Cheeger space, $\cW_X$ a self-dual mezzoperversity, $r: \sing X \lra B\Gamma$ the classifying map for the universal cover of $\sing X$ and $f: \sing M \lra \sing X$ be a smooth stratified map.
For each $\mu: \bbR \lra \bbR$ an even, rapidly decreasing function, we define the {\em compressed Hilsum-Skandalis replacement of $f$} to be the operator
\begin{equation*}
\begin{gathered}
	HS_{\mu}(f): 
	L^2(X; \Lambda^* {}^{\iie}T^*X \otimes \sG(r)) \lra
	L^2(M; \Lambda^* {}^{\iie}T^*M \otimes \sG(f\circ r)),
	\\
	HS_{\mu}(f) = \mu(D^{f^{\sharp}\cW_X}) \circ HS(f) \circ \mu(D^{\cW_X}).
\end{gathered}
\end{equation*}
\end{definition}
As elements of the functional calculus we know that, e.g., 
\begin{equation*}
	\mu(D^{\cW_X}): 
	L^2(X; \Lambda^* {}^{\iie}T^*X \otimes \sG(r)) \lra
	L^2(X; \Lambda^* {}^{\iie}T^*X \otimes \sG(r))
\end{equation*}
commutes with $D^{\cW_X}$ and is a bounded operator with range contained in the domain of  $D^{\cW_X}.$ In fact the range is contained in the domain
\begin{equation*}
	\cD_{\cW_X}(D^{\infty}) = \bigcap_{\ell \in \bbN} 
	\{\omega \in \cD_{\cW_X}(D):
	D\omega, \ldots, D^{\ell}\omega \in \cD_{\cW_X}(D) \}
\end{equation*}
as $x^{\ell}\mu(x)$ is a rapidly decreasing function for any $\ell \in \bbN.$
Since this domain is compactly included in 
$L^2(X; \Lambda^* {}^{\iie}T^*X \otimes \sG(r)),$ it follows that $\mu(D^{\cW_X})$ is a compact operator. Moreover since $\mu$ is even and $d$ commutes with $(d+\delta)^2,$ $d$ commutes with $\mu(D).$
Thus,
\begin{equation*}
	HS_{\mu}(f) \text{ is a compact operator and }
	HS_{\mu}(f) d^{\cW_X} = d^{f^{\sharp}\cW_X} HS_{\mu}(f).
\end{equation*}
The compressed Hilsum-Skandalis replacement satisfies properties similar to those of $HS(f),$ see Lemma 9.7 in \cite{PS1}.

%%%%%%%%%%%%%
\subsection{The Hilsum-Skandalis perturbation}$\,$\\
%%%%%%%%%%%%
On $\sm X \sqcup -\sm M$ consider the operators 
\begin{equation*}
	d_{\sm X \sqcup -\sm M}
	= \begin{pmatrix}
	d_{\sm X} & 0 \\
	0 & d_{\sm M}
	\end{pmatrix}, \quad
	\tau_{\sm X \sqcup -\sm M}
	= \begin{pmatrix}
	\tau_{\sm X} & 0 \\
	0 & -\tau_{\sm M}
	\end{pmatrix}, 
\end{equation*}
and, for $t \in [0,1],$ the operator
\begin{equation}\label{eq:DefL}
\begin{gathered}
	\cL_t:\cD_{\cW_X \sqcup f^{\sharp}\cW_X}(d_{\sm X \sqcup -\sm M})
	\lra
	\cD_{\cW_X \sqcup f^{\sharp}\cW_X}(d_{\sm X \sqcup -\sm M})\\
	\cL_t = 
	\begin{pmatrix}
	\Id - HS(f)'HS(f) & (1-it\gamma\Upsilon)\circ HS(f)'\\
	HS(f)\circ (1+ it\gamma \Upsilon) & \Id
	\end{pmatrix}.
\end{gathered}
\end{equation}
We point out that $\cL_t$ is bounded as an operator on the space $\cD_{\cW_X \sqcup f^{\sharp}\cW_X}(d_{\sm X \sqcup -\sm M})$ endowed with its $d_{\sm X \sqcup -\sm M}$-graph norm, and let $|\cL_t| = \sqrt{\cL_t^*\cL_t}$ denote the operator defined by the functional calculus on this Hilbert space (or equivalently as a bounded operator on $L^2$-differential forms).

As in \cite{HiSka}, the Hilsum-Skandalis replacement can be used to construct a perturbation of the signature operator
\begin{equation*}
	D^{\cW_X \sqcup f^{\sharp}\cW_X} 
	= 
	\begin{pmatrix}
	D^{\cW_X} & 0 \\
	0 & -D^{\cW_M}
	\end{pmatrix}
	\Mon \sing X \sqcup (-\sing M)
\end{equation*}
that results in an invertible operator. Indeed, for sufficiently small $t,$ the operator
\begin{equation*}
	D^{\cW_X \sqcup f^{\sharp}\cW_X} + \cC_t(f)
	= \tfrac1i U_t\circ \underline{D}_t \circ U_t^{-1}
\end{equation*}
is invertible, where $\underline{D}_t$ is the operator obtained from $D^{\cW_X \sqcup f^{\sharp}\cW_X}$ by making two replacements: 
\begin{equation*}
	d_{\sm X \sqcup -\sm M}
	\mapsto
	\begin{pmatrix}
	d_X & tHS(f)' \\ 0 & d_M
	\end{pmatrix}, \quad
	\tau_{\sm X \sqcup -\sm M}
	\mapsto
	\mathrm{sign}\lrpar{
	\tau_{\sm X \sqcup -\sm M}
	\circ \cL_t } = 
	\tau_{\sm X \sqcup -\sm M} \circ \mathrm{sign}(\cL_t)
\end{equation*}
and 
\begin{equation*}
	U_t = \abs
	{\tau_{\sm X \sqcup- \sm M}
	\circ \cL_t }^{1/2} = |\cL_t|^{1/2}.
\end{equation*}

\begin{lemma}\label{lem:UncompPert}
The operator $\cC_t(f)$ is a bounded operator relative to $D^{\cW_X \sqcup f^{\sharp}\cW_X};$ that is, $\cC_t(f)$ is bounded as a map
\begin{equation*}
	\cC_t(f): \cD_{\cW_X \sqcup f^{\sharp}\cW_X}(D) \lra 
	L^2(X \sqcup -M;\Lambda_*(X \sqcup -M)).
\end{equation*}
The operator $D^{\cW_X \sqcup f^{\sharp}\cW_X} + \cC_t(f)$ is invertible for small enough $t>0.$
\end{lemma}

\begin{proof}
The boundedness of $\cC_t(f)$ relative to $D^{\cW_X \sqcup f^{\sharp}\cW_X}$ follows from the fact that $\cL_t$ is a bounded operator on $\cD_{\cW_X}(d).$ 
With notation similar to \cite[Proof of Proposition 3.4]{Zenobi:Mapping}, we can write 
\begin{equation*}
	E_t =
	\begin{pmatrix}
	0 & tHS(f)' \\ 0 & 0
	\end{pmatrix}, \quad
	\mathrm{sign}(\cL_t) = \Id + G_t, \quad
	U_t = \Id+ H_t', \quad
	U_t^{-1} = \Id +F_t'
\end{equation*}
with $E_t,$ $G_t,$ $H_t',$ $F_t'$ bounded operators on $\cD_{\cW_X\sqcup f^{\sharp}\cW_X}(d),$ with its $d$-graph norm as well as on 
$L^2(X \sqcup -M;\Lambda_*(X \sqcup -M)).$
Then, e.g., in the even dimensional case we can write
\begin{equation*}
	D_t = \tfrac1i (1 + F_t')\circ
	((d+E_t)
	+ \tau_{\sm X \sqcup -\sm M} \circ(1+ G_t) \circ
	(d+E_t) \circ 
	\tau_{\sm X \sqcup -\sm M} \circ(1+ G_t) )
	\circ (1 + H_t') \\
	= D + \cC_t(f)
\end{equation*}
and it follows that $\cC_t(f)$ is bounded as a map 
from $\cD_{\cW_X \sqcup f^{\sharp}\cW_X}(D)$ to $L^2(X \sqcup -M;\Lambda_*(X \sqcup -M)).$\\
Note that $\cL_0$ satisfies
\begin{equation*}
	\cL_0 = R'R, \quad R =
	\begin{pmatrix}
	\Id & 0 \\ HS(f) & \Id
	\end{pmatrix}
\end{equation*}
and, since $R$ is invertible, this shows that $\cL_0$ is invertible and hence $\cL_t$ is invertible for small enough $t.$
The invertibility of $D^{\cW_X \sqcup f^{\sharp}\cW_X} + \cC_t(f)$ as an unbounded operator with domain $\cD_{\cW_X \sqcup f^{\sharp}\cW_X}(D_{X \sqcup -M})$ now follows as in \cite[Lemme 2.1]{HiSka}, \cite[\S3]{Wahl_higher_rho}. 
\end{proof}

A similar result holds for the signature operator on $X\sqcup (-M)$, with mezzoperversity 
given by $\cW_X$ and $  f^{\sharp}\cW_X$ and   twisted by the Mishchenko bundle
$\sG(r)$ on $X$ and $\sG(f\circ r)$ on $M$. In this case we use the Hilsum-Skandalis replacement
$HS(f): \cD_{\cW_X}(d^{\sG(r)}) \lra \cD_{\cW_M}(d^{\sG(f\circ r)})$.
%%%%%%%%%%%%%%
\subsection{The compressed Hilsum-Skandalis perturbation} $\,$\\
%%%%%%%%%%%%%%
We can repeat the argument from the previous subsection replacing $HS(f)$ by $HS_{\mu}(f).$ The resulting perturbation, which we denote $\cC_{t,\mu}(f)$ and refer to as the {\em compressed Hilsum-Skandalis perturbation}, satisfies an improved version of Lemma \ref{lem:UncompPert}.

\begin{lemma}
The operator $\cC_{t,\mu}(f)$ extends from $\cD_{\cW_X \sqcup f^{\sharp}\cW_X}(D_{\sm X \sqcup -\sm M})$ to a compact operator
\begin{equation*}
	\cC_{t,\mu}(f): 
	L^2(X \sqcup -M;\Lambda_*(X \sqcup -M)) \lra 
	L^2(X \sqcup -M;\Lambda_*(X \sqcup -M)).
\end{equation*}
The operator $D^{\cW_X \sqcup f^{\sharp}\cW_X} + \cC_{t,\mu}(f)$ is invertible.
\end{lemma}

\begin{proof}
If $\cL_{t,\mu}$ is the operator obtained as in \eqref{eq:DefL} but using $HS_{\mu}(f),$ then it is an invertible operator of the form $\Id+H_{t,\mu}$ with $H_{t,\mu}$ a compact operator such that both $H_{t,\mu}$ and its adjoint map $L^2(X \sqcup -M;\Lambda_*(X \sqcup -M))$ into $\cD_{\cW_X \sqcup f^{\sharp}\cW_X}(D_{\sm X \sqcup -\sm M}^{\infty}).$  It follows from, e.g., the argument used in Lemma A.12 of \cite{PS1}, that each of the operators $E_{t,\mu},$ $G_{t,\mu},$ $H_{t,\mu}',$ $F_{t,\mu}'$ defined as in the proof of Lemma \ref{lem:UncompPert} will also have this property. Hence $\cC_{t,\mu}(f)$ will be a compact operator.

The invertibility of the perturbed signature operator follows from \cite[Lemme 2.1]{HiSka}.
\end{proof}

Also in this case we can extend the whole analysis to the signature operators twisted by the appropriate
Mishchenko bundles; we state and use this result in Proposition \ref{prop:compr-HS-perturbation} below.

\subsection{Passing to the Roe algebra} $\,$\\
Let $\bbB(\cE)$ denote the  operators acting on the Hilbert $C^*_r \Gamma$-module
$\mathcal{E}:=L^2 (X,\Lambda^* X\otimes \cG(r))\oplus L^2 (M,\Lambda^* M\otimes G(r\circ f))$. 
Recall that
there is a $C^*$-homomorphism $$L_\pi: \bbB( \cE)\to \mathcal{B}(L^2 (X_\Gamma,\Lambda_* X_\Gamma)\oplus 
L^2 (M_\Gamma,\Lambda_* M_\Gamma))$$ and that $L_\pi$ induces an isomorphism between $\bbK(\cE)$
and the Roe algebra $C^* (\cov{\sing X}\sqcup (- \cov{\sing{M}}))^\Gamma$ and between $\bbB(\cE)$
and the multiplier algebra $\mathfrak{M}(C^* (\cov{\sing X}\sqcup (- \cov{\sing{M}}))^\Gamma)$ of the Roe algebra.
%Consequently if $C(f):= L_\pi (\cC(f))$ then 
%\begin{equation}\label{hs-multiplier}
%C(f)\in \mathfrak{M}(C^* (\cov{\sing X}\sqcup (- \cov{\sing{M}}))^\Gamma)\,.
%\end{equation}
%If $f$ is a stratified homotopy equivalence, then the perturbed signature operator is $L^2$-invertible.

 \begin{proposition}\label{prop:compr-HS-perturbation}
The compressed Hilsum-Skandalis perturbation  $\cC_{\mu,t} (f)$ is an element in $ \bbK (\cE)$.
Consequently, if $C_{\mu,t} (f):=L_\pi (\cC_{\mu,t} (f))$, then
\begin{equation}\label{compr-hs}
	C_{\mu,t} (f)\in C^* (\cov{\sing X}\sqcup (- \cov{\sing{M}}))^\Gamma))\,.
\end{equation}
\end{proposition}

\subsection{APS-index classes associated to degree one normal maps}$\,$\\
%%%%%%%%%%
Let $\sM$ and $\sX$ Cheeger spaces with boundary. We denote 
$\sing M:=\pa \sM$ and $\sing X:= \pa \sX$. We let $\sZ:= \sX\sqcup (-\sM)$. We assume the existence 
of a smooth transverse stratified map $F:\sM\to \sX$; we further assume that $F$ sends $\pa \sM$ into $\pa \sX$ 
%\footnote{From Paolo: this
%should be automatic}
and that
$F_{\pa}:=F|_{\pa \sY}$ is a stratified homotopy equivalence from $\pa \sM\equiv \sing M$ into $\pa \sX\equiv \sing X$, as in the previous
subsection. Finally we assume the existence 
of a classifying map $\sX\to B\Gamma$; the latter, together with $F$, defines a $\Gamma$-covering $\Gamma\to \sZ_\Gamma\to \sZ$.
We fix a self-dual mezzoperversity $\sW$ on $\sX$ and consider the induced mezzo-perversity $F^{\sharp} \sW$
on $\sM$. This gives $\sZ$, and thus $\sZ_\Gamma$, a self-dual mezzo-perversity $\sW\sqcup F^{\sharp} \sW$.
We consider now $\sing Z:=\pa \sZ\equiv \sing X\sqcup (-\sing M)$ 
and $\cov{\sing Z}:= \pa \sZ_\Gamma$; this gives a Galois $\Gamma$-covering of Cheeger-spaces without boundary
$$\Gamma - \cov{\sing Z}\to \sing Z \;\;\equiv\;\; \Gamma- \pa \sZ_\Gamma \to \pa \sZ\,.$$
By our discussion above there is a well defined (compressed) Hilsum-Skandalis perturbation $C(F_\partial)\in C^* (\cov{\sing Z})$ (for simplicity, we will no longer include the $t,\mu$ sub-indices in the notation for the perturbation);
this is a trivializing perturbation for the signature operator on $\pa \sZ_\Gamma$ with domain fixed by $\pa \sW\sqcup f^{\sharp}
\pa \sW$. 
We thus obtain a well defined APS coarse index class that we shall denote as  $\Ind (D_\Gamma^{\sW\sqcup F^{\sharp}\sW}, C(F_{\pa}))$
in $K_* (C^* (\sZ_\Gamma)).$

\begin{proposition}\label{prop:extension}
If $F$ is a global stratified transverse homotopy equivalence, then
\begin{equation*}
\Ind (D^{\sW\sqcup F^{\sharp}\sW}_\Gamma,C(F_{\pa}) )=0
\end{equation*}
\end{proposition}

\begin{proof}
There is a well-defined Hilsum-Skandalis perturbation, $C(F),$ on the manifold with cylindrical ends, $\sZ_{\Gamma,\infty}.$
(Notice that this is an `un-compressed' perturbation.)

Proceeding as in \cite{HiSka} and \cite{Wahl_higher_rho}, one can show that the associated perturbed signature operator is {\it invertible}.
One can then proceed as in \cite{Wahl_higher_rho} and show that the index class in the statement of the proposition,  viz.
$\Ind (D_\Gamma^{\sW\sqcup F^{\sharp}\sW},C(F_{\pa}) ),$ is equal to the index class of the invertible operator and hence vanishes. This is done in two steps. The first is a spectral flow argument, which is purely functional analytic and hence applies in our setting. The second is a relative index theorem following Bunke, of the type we will discuss below in Proposition \ref{gluing}, and also applies in our setting.
\end{proof}

\section{Mapping the Browder-Quinn surgery sequence to analysis}\label{sect:mapping}

\subsection{The rho class of a stratified homotopy equivalence}$\,$\\
	Let $f:\sing M\to \sing X$ be a transverse stratified homotopy equivalence.
Let $\Gamma$ be $\pi_1 (\sing X)$.
Let $\sing Z:= \sing X \sqcup (- \sing M)$. The Cheeger space $\sing Z$ comes equipped with two maps
induced respectively by $f$ and the identity and by $f$ and the classifying map for $\sing X$:
\begin{equation*}
\phi: \sing Z\to \sing X\,,\quad u: \sing Z\to B\Gamma\,.
\end{equation*}
 In particular, there is a well defined $\Gamma$ covering $\Gamma - \cov{\sing Z}\to \sing Z$.
We let $u_\Gamma:\cov{\sing Z}\to E\Gamma$ be the $\Gamma$-equivariant lift of $u$.
We fix a self-dual mezzoperversity $\cW$ on $\sing X$ and consider the associated self-dual mezzoperversity $f^{\sharp} \cW$
on $-\sing M$. We call $\cW\sqcup f^{\sharp} \cW$ the resulting self-dual mezzoperversity on $\sing Z$. We then have  self-adjoint
extensions $D^{\cW\sqcup f^{\sharp} \cW}$ on $Z$, $\De^{\cW\sqcup f^{\sharp} \cW}$ on $Z_\Gamma$  and, by Proposition \ref{prop:compr-HS-perturbation}, a well defined (compressed) Hilsum-Skandalis perturbation
$C_f\in C^* (\cov{\sing X}\sqcup (- \cov{\sing{M}}))^\Gamma\equiv C^* (\cov{\sing Z})$

\begin{definition}\label{def:rho-of-an-he}
The rho-class $\rho (\sing M\xrightarrow{f} \sing X,\cW)$ associated to $f:\sing M\to \sing X$ and the 
self-dual mezzoperversity $\cW$  is given by 
\begin{equation}\label{eq:rho}
\rho (\sing M\xrightarrow{f} \sing X,\cW):=\phi_* (\rho (D^{\cW\sqcup f^{\sharp}\cW}_\Gamma+C_f))\in K_{\dim X+1} (D^* (\cov{\sing X})^\Gamma)
\end{equation}
The universal rho class is, by definition,
\begin{equation}\label{eq:universal-rho}
\rho_\Gamma (\sing M\xrightarrow{f} \sing X,\cW):=(u_\Gamma)_* (\rho (D^{\cW\sqcup f^{\sharp}\cW}_\Gamma+C_f))\in K_{\dim X+1} 
(D^*_\Gamma)
\end{equation}
\end{definition}

We shall see in the next subsection that the rho class of a stratified homotopy equivalence 
is independent of $\cW$ and descends to $ \tS_{\BQ}(\sing X)$.

\subsection{The rho map  from $\tS_{\BQ}(\sing X)$   to $K_{\dim \sing X + 1}(D^* (\cov{\sing X})^\Gamma)$} \label{sec:RhoMap}

\begin{proposition}\label{prop:well-def-on-S}
The rho class associated to a transverse stratified homotopy equivalence $f:\sing M\to \sing X$ and a
self-dual mezzoperversity $\cW$ on $\sing X$ satisfies the following properties:
\item{1]} 
it is independent of the choice of $\cW$;
\item{2]} it gives a well-defined map 
\begin{equation}\label{well-def-on-S}
\rho: \tS_{\BQ}(\sing X)\longrightarrow K_{\dim \sing X + 1}(D^* (\cov{\sing X})^\Gamma)
\end{equation}
We denote by $\rho [\sing M\xrightarrow{f} \sing X]$ the image of  $ [\sing M\xrightarrow{f} \sing X]$ through the rho map.
\end{proposition}

\begin{proof}
Let $g$ and $g^\prime$ two $\operatorname{iie}$-metrics on $\sing X$ and let $\cW$ and $\cW^\prime$ be two self-dual mezzoperversities adapted respectively to $g$ and $g^\prime$. Let $r:\sing X \to B\Gamma$ be a classifying map.
Recall, following Banagl, how it is proved that $(\sing X,g,\cW, r)$ is Cheeger-bordant to $(\sing X,g^\prime,\cW^\prime,r)$;
we refer the reader to \cite[Section 4.4]{ALMP:Novikov} for the details.
We consider the pseudomanifold with boundary 
\begin{equation*}
	\sX = \sing X \times [0,1]_t.
\end{equation*}
Instead of the product stratification,  we stratify $\sX$ using the strata of $\sing X$ as follows:\\
i) The regular stratum $\sm X$ of $\sing X$ contributes $X \times [0,1]$\\
ii) Every singular stratum of $\sing X,$ $Y^k,$ contributes three strata to $\sX,$
\begin{equation*}
	Y^k \times [0,1/2), \quad Y^k \times (1/2, 1], \quad Y^k \times \{1/2 \}.
\end{equation*}
The link of $\sX$ at $Y^k \times [0,1/2)$ and $Y^k \times (1/2, 1]$ is equal to $Z^k,$
while the link of $\sX$ at $Y^k \times \{ 1/2 \}$ is seen to be the (unreduced) suspension of $Z^k,$ $S Z^k.$
Since the lower middle perversity intersection homology of $S Z^k,$ when $\dim Z^k = 2j-1,$ is given by
\begin{equation*}
	I^{\bar m}H_i(S Z^k) = 
	\begin{cases}
		I^{\bar m}H_{i-1}(Z^k) & i>j \\
		0 & i=j \\
		I^{\bar m}H_i(Z^k) & i <k
	\end{cases}
\end{equation*}
we see that $\sX$  satisfies the Witt condition at the strata $Y^k \times \{ 1/2 \}.$ Put it differently, we do not need 
to fix a self-dual
mezzoperversity at this stratum.

Let us endow $\sX$ with any $\operatorname{iie}$ metric $G$ such that,
for some $t_0>0,$ 
\begin{equation*}
	G\rest{X \times [0,t_0)} = g + dt^2, \quad
	G\rest{X \times (1-t_0, 1]} = g' + dt^2.
\end{equation*}
Next we  endow $\sX$ with a self-dual mezzoperversity $\sW$ as follows:
let $Y^1, \ldots, Y^T$ be an ordering of the strata of $\hat X$ with non-decreasing depth.
Denote 
\begin{equation*}
	\cW = \{ W^1 \lra Y^1, \ldots, W^T \lra Y^T \}, \quad
	\cW' = \{ (W^1)' \lra Y^1, \ldots, (W^T)' \lra Y^T \}
\end{equation*}
and denote the fiber of, e.g., $W^j \lra Y^j$ at the point $q \in Y^j,$ by $W^j_q.$
Let us define
\begin{equation*}
	W^1_- \lra Y^1 \times [0,1/2)
\end{equation*}
by requiring that the Hodge-de Rham isomorphism identifies all of the fibers.
Once this is done, we can define $W^2_- \lra Y^2 \times [0,1/2)$ in the same way, and inductively define $W^3_- \lra Y^3 \times [0,1/2), \ldots, W^T_- \lra Y^T \times [0,1/2).$

We define $W^j_+ \lra Y^j \times (1/2 \times 1]$ in the same way to obtain
\begin{multline*}
	\sW = \{ W^1_- \lra Y^1 \times [0,1/2), W^1_+ \lra Y^1 \times (1/2, 1], 
		\ldots, \\ 
	W^T_- \lra Y^T \times [0,1/2), W^T_+ \lra Y^T \times (1/2, 1] \},
\end{multline*}
a self-dual mezzoperversity over $\sX.$
So, in words, we extend the metrics $g$ and $g'$ arbitrarily to an $\operatorname{iie}$ metric $G$ without changing them in collar neighborhoods of the boundary, and then we choose a Hodge mezzoperversity by extending the de Rham mezzoperversities trivially from $Y^i$ to $Y^i \times [0,1/2)$ on the left and from $Y^i$ to $Y^i \times (1/2, 1]$ on the right.
Since the strata induced by $Y^k \times [0,1/2)$ are disjoint from the strata induced by $Y^k \times (1/2,1],$ there is no compatibility required between the corresponding mezzoperversities.\\
Finally, define $R:\sX \lra B\Gamma$ by $R(\zeta, t) = r(\zeta).$
The result is a Cheeger-bordism
\begin{equation*}
	(\sX, G, \sW, R:\sX \lra B\Gamma),
\end{equation*}
between $(\hat X,g,\cW, r:\sing X \lra B\Gamma)$ and $(\hat X',g',\cW', r:\sing X \lra B\Gamma).$

Let us go back to the proof of our Proposition. 
Let $f:\sing M \to \sing X$ be a transverse stratified homotopy equivalence. We want to show
that the rho class is independent of the choice of the self-dual mezzoperversity $\cW$ on $\sing X$.
Let $g$, $\cW$ and $g^\prime$, $\cW^\prime$ as above and consider $f^*g$, $f^{\sharp} W$ and
$f^* (g^\prime)$, $f^{\sharp} \cW^\prime$ on $\sing M$. We can consider $\sM:=\sing M\times [0,1]$, stratified as above.
%There is a well defined self-dual mezzoperversity on $\sY$, call it $\sV$ for the time being, producing the Cheeger bordsim between $(\sing M,f^*g, f^{\sharp} W,(r\circ f): \sing M\to B\Gamma)$ and
%$(\sing M,f^* (g^\prime),f^{\sharp} \cW^\prime, (r\circ f): \sing M\to B\Gamma)$. 
%Since $f$ is a homotopy equivalence, 
%by employing
%the Hilsum-Skandalis perturbation on the two boundaries of $(-\sY)\sqcup \sX$
%we get a well defined APS-index class 
Remark now that, by definition,
the map $F:\sM\to \sX$, $F(\zeta,t)=f(\zeta)$ is such  that $F^{\sharp} \sW$, adapted to $F^* G$, is precisely equal to the self-dual mezzoperversity 
producing the Cheeger bordism between $(\sing M,f^*g, f^{\sharp} W,(r\circ f): \sing M\to B\Gamma)$ and
$(\sing M,f^* (g^\prime),f^{\sharp} \cW^\prime, (r\circ f): \sing M\to B\Gamma)$. Moreover, $F$ is a (transverse) stratified 
homotopy equivalence between $\sM$ and $\sX$.  

We thus have a stratified Cheeger-space with boundary,
$$\sZ:= (-\sM)\sqcup \sX$$
which is the disjoint union of two stratified Cheeger spaces with boundary, endowed with a stratified homotopy equivalence
$F:\sM\to \sX$, with self-dual mezzoperversities $F^{\sharp} \sW$ on $\sM$ and $\sW$ on $\sX$ and with a classifying map
into $B\Gamma$, the latter producing a Galois $\Gamma$-covering $\Gamma - \cov{\sZ}\to \sZ$; moreover, by construction, the self-dual mezzoperversity on the manifold with boundary $ (-\sM)\sqcup \sX$
 restricts to give $f^{\sharp}\cW\sqcup \cW$
on one boundary, the one corresponding to $t=0$, and  $f^{\sharp}\cW^\prime\sqcup \cW^\prime$ on the other boundary, the one 
corresponding to $t=1$. For later use we denote by $j_0$ and $j_1$ the obvious inclusions of 
$(-\cov{\sing M})\sqcup \cov{\sing X}$ into $\sZ_\Gamma$ as the $t=0$ and $t=1$ boundary respectively.
We now apply Proposition \ref{prop:extension} and obtain that
$$\Ind (D^{\sW\sqcup F^{\sharp} \sW},C(F_{\pa}))=0 \quad\text{in}\quad K_* (C^* (\sZ_\Gamma)^\Gamma)\,.$$
By applying the delocalized APS-index theorem we then obtain that
\begin{equation}\label{0=difference}
0=(j_0)_* (\rho (D^{\cW\sqcup f^{\sharp} \cW}+C_f)) - (j_1)_* (\rho (D^{\cW^\prime \sqcup f^{\sharp} \cW^\prime}+C^\prime_f)) \quad\text{in}\quad K_* (C^* (\sZ_\Gamma)^\Gamma)
\,.
\end{equation}
Observe now that there is an obvious $\Gamma$-equivariant map $\sZ_\Gamma\to \sX_\Gamma=\cov{\sing X}\times [0,1]$, induced by $F$ and the identity, and thus, by projecting onto the first factor,  a $\Gamma$-equivariant map $\sZ_\Gamma\to \cov{\sing X}$. We can push-forward the equality \eqref{0=difference}  through this map and use functoriality in order to obtain
$$0= \rho (\sing M\xrightarrow{f} \sing X,\cW) - \rho (\sing M\xrightarrow{f} \sing X,\cW^\prime)\in K_{*+1} (D^* (\cov{\sing X})^\Gamma)\;\;\text{with}\;\; *=\dim\sing X\,;$$
this shows indeed that  the rho class is independent of the choice of  self-dual mezzoperversity.

The proof of item 2] is very similar. 

\end{proof}

\subsection{The map from  $\tN_{\BQ}(\sing X)$ to $K_{\dim \sing X} (\sing X)$}$\,$\\
We have defined $\tN_{\BQ}(\sing X)$ as equivalence classes of {\em transverse} degree one normal maps into $\sing X$  which are diffeomorphisms when restricted to strata of dimension less than five. Our task is to
map an element $[\sing M\xrightarrow{f} \sing X]\in \tN_{\BQ}(\sing X)$ to $K_{\dim \sing X} (\sing X)$, or, more precisely, to
$K_{\dim \sing X} (\sing X)\otimes\bbZ[1/2]$.
Following the original treatment of Higson and Roe in the smoth setting, we shall in fact
forget about the normal data encoded in $[\sing M\xrightarrow{f} \sing X]\in \tN_{\BQ}(\sing X)$.
%We fix a self-dual mezzoperversity $\cW$ on $\sing X$ and consider $f^{\sharp}\cW$ on $\sing M$. We set \footnote{\textcolor{m}{From Paolo:
%there might be powers of $2$ to be included in the definition} \red{I wrote this because I was looking at Higson-Roe 3
%where they multiply by a power of $1/2$;
%however  they also do this for the map out of the L-groups; I guess they have different conventions...Let us leave it
%as it is} }
\begin{equation}\label{def-of-beta}
\beta [\sing M\xrightarrow{f} \sing X]:= f_* [D^{f^{\sharp}\cW}]- [D^{\cW}]\;\;\in\;\; K_{\dim \sing X} (\sing X)\otimes\bbZ[1/2]
\end{equation}
We then have the following

\begin{proposition}
\item{1]} The right hand side of \eqref{def-of-beta} is independent of the choice of 
self-dual mezzoperversity $\cW$.
\item{2]} The map $\beta$ is well defined: if $ [\sing M_0\xrightarrow{f_0} \sing X]= [\sing M_1\xrightarrow{f_1} \sing X]$
in $\tN_{\BQ}(\sing X)$, then 
\begin{equation}\label{def-of-beta-bis}
(f_0)_* [D^{f_0^{\sharp}\cW}]- [D^{\cW}]=(f_1)_* [D^{f^{\sharp}_1\cW}]- [D^{\cW}]\;\;\in\;\; K_{\dim \sing X} (\sing X)\otimes\bbZ[1/2]
\end{equation}
\end{proposition}

\begin{proof}
We establish both statements by adapting an argument due to Higson and Roe
and by making use of Theorem \ref{theo:main-bordism} above, item 3].

Thus let $\cW$ and $\cW^\prime$ be two self-dual mezzoperversities, adapted to $\operatorname{iie}$ metrics $g$ and $g^\prime$ respectively. We must show that
$$ f_* [D^{f^{\sharp}\cW}]- [D^{\cW}]- (f_* [D^{f^{\sharp}\cW^\prime}] - [D^{\cW^\prime}])=0
\;\;\in\;\; K_{\dim \sing X} (\sing X)\otimes\bbZ[1/2]
$$
We initially follow the construction exploited in the previous subsection.
Thus we 
consider $\sX:=\sing X\times [0,1]$ and 
$\sM:=\sing M\times [0,1]$, both stratified \`a la Banagl.
We consider the transverse map $F:\sM\to \sX$, $F(\zeta,t)=f(\zeta)$ 
and consider $G$, $F^*G$, $\sW$ and $F^{\sharp} \sW$ as in the previous subsection.
We thus have a stratified Cheeger-space with boundary,
$$\sZ:= (-\sM)\sqcup \sX\equiv (-(\sing M\times [0,1]))\sqcup (\sing X\times [0,1])$$
which is the disjoint union of two stratified Cheeger spaces with boundary, endowed with a stratified transverse map
$F:\sM\to \sX$, with self-dual mezzoperversities $F^{\sharp} \sW$ on $\sM$ and $\sW$ on $\sX$; moreover, by construction, the self-dual mezzoperversity $F^{\sharp}\sW\sqcup \sW$ on the manifold with boundary $ (-\sM)\sqcup \sX$
 restricts to give $f^{\sharp}\cW\sqcup \cW$
on one boundary, the one corresponding to $t=0$, and  $f^{\sharp}\cW^\prime\sqcup \cW^\prime$ on the other boundary, the one 
corresponding to $t=1$. 

Remark first of all that the K-homology groups of a disjoint union of two spaces $A\sqcup B$ is equal to the direct sum
of the individual K-homology groups. We define two group homomorphisms
\begin{equation*}
\begin{gathered}
\Phi: K_* (\sZ,\partial \sZ)\to K_* (\sX,\partial \sX)=K_* (\sing X\times [0,1],\sing X\times\{0,1\})\,,\\
\phi: K_* (\partial \sZ) \to K_* (\pa \sX)=K_* (\sing X\times\{0,1\})
\end{gathered}
\end{equation*}
as follows:
\begin{equation*}
\Phi (\alpha_{\sM},\beta_{\sX})=F_* \alpha_{\sM} - \beta_{\sX} \,,\quad
\phi(\alpha_0,\alpha_1,\beta_0,\beta_1)= (f_* \alpha_0 -\beta_0,f_* \alpha_1 - \beta_1)
\end{equation*}
It is easy to check, using the functoriality properties of the connecting homomorphism in the long exact sequence of 
a pair, that the following diagram is commutative:
\begin{equation*}
   	\xymatrix{
	K_{*+1} (\sZ,\partial \sZ) \ar[rr]^-{\partial_{\sqcup}}  \ar[d]^{\Phi} & & K_* (\partial \sZ) \ar[d]^{\phi} \\
	K_{*+1} (\sX,\partial \sX) \ar[rr]^-{\pa} & & K_* (\pa \sX)}
\end{equation*}
The bottom horizontal homomorphism is part of the long exact sequence
$$ K_{*+1} (\sX,\partial \sX) \xrightarrow{\pa}  K_* (\pa \sX)\xrightarrow{\iota} K_* (\sX)$$
which can be rewritten as
\begin{equation}\label{exact-portion}
 K_{*+1} (\sing X\times [0,1],\sing X\times\{0,1\})\xrightarrow{\pa} K_* (\sing X\times\{0,1\}) \xrightarrow{\iota}
K_* (\sing X\times [0,1])\,.
\end{equation}
Notice that  there is a natural group homomorphism 
$$\psi: K_* (\sing X\times\{0,1\})\to K_* (\sing X)\,,\quad \psi(\gamma_0,\gamma_1)=\gamma_0 - \gamma_1 $$
and that $\psi$ factors as follows:

\begin{equation*}
\xymatrix @R=1pt @C=50pt { 
	 K_* (\sing X\times\{0,1\})\ar[rd]^-{\psi} \ar[dd]_-{\iota} \\
	& K_* (\sing X) \\
	K_* (\sing X\times [0,1]) \ar[ru]_-{\pi} }
\end{equation*}
	with $\pi$ induced by the projection onto the first factor.
Using these remarks and Theorem \ref{theo:main-bordism}, which in the present context 
states that
$$ \pa_{\sqcup} ([D^{F^{\sharp}\sW}\sqcup \sW])=([D^{f^{\sharp}\cW}],[D^{f^{\sharp} \cW^\prime}],[D^{\cW}],[D^{\cW^\prime}])
$$
 we then have
\begin{equation*}
\begin{aligned}
&f_* [D^{f^{\sharp}\cW}]- [D^{\cW}]- (f_* [D^{f^{\sharp}\cW^\prime}] - [D^{\cW^\prime}])\\&= 
\psi ( f_* [D^{f^{\sharp}\cW}]- [D^{\cW}], f_* [D^{f^{\sharp}\cW^\prime}] - [D^{\cW^\prime}])\\
&=\pi\circ \iota ( f_* [D^{f^{\sharp}\cW}]- [D^{\cW}], f_* [D^{f^{\sharp}\cW^\prime}] - [D^{\cW^\prime}])\\
&=\pi\circ\iota\circ \phi ([D^{f^{\sharp}\cW}],[D^{f^{\sharp} \cW^\prime}],[D^{\cW}],[D^{\cW^\prime}])\\
&=\pi\circ\iota\circ\phi \circ\pa_{\sqcup} ([D^{F^{\sharp\sW}\sqcup \sW}])\\
&=\pi\circ\iota\circ \pa \circ \Phi  [D^{F^{\sharp}\sW\sqcup \sW}] =0
\end{aligned}
\end{equation*}	
where in the last step we have used the exactness of \eqref{exact-portion}.\\
This establishes item 1]. Item 2] is similar, but easier.
	\end{proof}

\subsection{The index map from  $\tL_{\BQ}({\sing X})$   to $K_* (C^* (\cov{\sing X})^\Gamma)$}$\,$\\
We finally consider the (APS) index homomorphisms 
$$\tL_{\BQ}(\sing X\times [0,1])\xrightarrow{\Ind_{{\rm APS}}} K_{\dim \sing X+1} (C^* (\cov{\sing X})^\Gamma)\,,\qquad 
\tL_{\BQ}(\sing X)\xrightarrow{\Ind_{{\rm APS}}}  K_{\dim \sing X} (C^* (\cov{\sing X})^\Gamma)\,.$$
First, we  restate some of the constructions above in a way  that will
be useful to the present task.

\begin{lemma}\label{lemma:special-cycles}
Every $[\alpha] \in \tL_{\BQ}(\sing X \times [0,1])$ can be represented by a diagram of the form 
\begin{equation*}
	\alpha: (\sing M; \sing X, \sing X') \xlra{(\phi;\id, \psi)}
	(\sing X \times [0,1]; \sing X \times \{0\}, \sing X \times \{1\}) \xlra{\id}
	\sing X \times I.
\end{equation*}
Let $\alpha_1$ and $\alpha_2$ be two such diagrams representing the same class and $\beta$ the diagram obtained by gluing $\alpha_1$ and $-\alpha_2,$
\begin{equation*}
	\beta:
	(\sing W; \sing X_2', \sing X_1') \xlra{(\Phi;\psi_2, \psi_1)}
	(\sing X \times [0,1]; \sing X \times \{0\}, \sing X \times \{1\}) \xlra{\id}
	\sing X \times I,
\end{equation*}
then there are stratified spaces with corners $\sing P,$ $\sing Q$ together with a BQ-normal map $\Gamma,$
\begin{equation*}
	(\sing P; \pa_0\sing P, \pa_1 \sing P) 
	\xlra{(\Gamma; \gamma_0, \gamma_1)}
	(\sing Q; \pa_0\sing Q, \pa_1 \sing Q)
\end{equation*}
such that 
\begin{equation*}
	\lrpar{\pa_0\sing P \xlra{\gamma_0} \pa_0 \sing Q}
	=
	\lrpar{\sing W \xlra{\Phi} \sing X \times I}
\end{equation*}
and $\gamma_1$ is a BQ-equivalence.
\end{lemma}

%\footnote{\textcolor{m}{FROM PAOLO: do I understand correctly that the total boundary
%of, for example, $\sing P$ is equal to the union of
%$\pa_0\sing P$ and $\pa_1 \sing P$ ? Or are there other boundary hypersurfaces ?}
%\textcolor{b}{Yes, $\pa\sing P = \pa_0\sing P \cup \pa_1\sing P$}}

\begin{proof}
We can represent $[\alpha]$ in this way directly from the Wall representation theorem (as in Corollary \ref{cor:LAct}).
Taking
\begin{equation*}
	\sing W = \sing M_1 \bigsqcup_{\sing X} -\sing M_2.
\end{equation*}
and gluing $\alpha_1$ and $-\alpha_2$ along their common boundary yields $\beta,$ an element of $\cN_{BQ}(\sing X \times [0,1]).$ Since Lemma \ref{lem:SumLemma} implies that $\beta$ represents $[\alpha_1]-[\alpha_2] =0$ in $\tL_{BQ}(\sing X \times [0,1]),$ Theorem \ref{thm:exact1}, implies that the class of $\beta$ in $\tN_{BQ}(\sing X \times [0,1], \pa(\sing X \times [0,1]))$ is in the image of 
$\tS_{BQ}(\sing X \times [0,1], \pa(\sing X \times [0,1])).$
It follows that there is a normal bordism between $\beta$ and an element in $\tS_{{\rm BQ}}(\sing X \times [0,1], \pa(\sing X\times [0,1])),$ which yields $\Gamma: \sing P \lra \sing Q$ as above.
\end{proof}

We can now define the homomorphism $\Ind: \tL_{BQ}(\sing X\times [0,1])\to K_* (C^* (\cov{\sing X})^\Gamma)$.
Consider an element $[\alpha]\in \tL_{BQ}(\sing X\times [0,1])$ and assume, thanks to the previous Lemma,
that  \begin{equation*}
	\alpha: (\sing M; \sing X, \sing X') \xlra{(\phi;\id, \psi)}
	(\sing X \times [0,1]; \sing X \times \{0\}, \sing X \times \{1\}) \xlra{\id}
	\sing X \times I.
\end{equation*}
We have a classifying map $\sing X\to B\Gamma$, with $\Gamma=\pi_1 (\sing X)$. This induces, through $\phi:\sing M\to \sing X\times [0,1]$
a $\Gamma$-Galois covering $\sing Z_\Gamma$
on $\sing Z:= (-\sing M)\cup (\sing X\times [0,1])$ 
%Also, in the previous sections we 
%would have used the notation $\sZ$ for $\sing Z$ (in general $\sX$, and not $X$, would be a Cheeger space with boundary...).

We now fix an incomplete iterated conic metric on $\sing X$ and we choose an adapted
 mezzoperversity $\cW$ for the resulting signature operator. 
 We take the associated {\it product mezzoperversity} $\cW\times [0,1]$
 on $\sing X\times [0,1]$. We endow $-\sing M$ with the induced metric and the induced mezzoperversity 
 $\phi^{\sharp} (\cW\times [0,1])$. We lift all these data 
 to $\sing Z_\Gamma$.
 Recall at this time that $\phi$ is {\it not} an homotopy equivalence.
 On the other hand, since $\psi$ (and, of course, the identity
$\sing X\xrightarrow{\id} \sing X$)  {\it is} a stratified homotopy equivalence 
we see that there exists a well defined (APS) index class, obtained by using the compressed Hilsum-Skandalis
perturbation associated to $\psi\sqcup\id$ and to $(\psi^{\sharp}\cW\sqcup \cW)\sqcup (\cW\sqcup \cW)$  on 
$\partial \sing Z$: this index class belongs to $K_* (C^* (\sing Z_\Gamma)^\Gamma)$.
We define $\Ind_{{\rm APS}}  (\alpha,\cW)$ as the push-forward of the above index class
to $ K_* (C^* (\sing X_\Gamma)^\Gamma)$. In what follows we shall use the canonical isomorphism
 $K_* (C^* (\sing X_\Gamma)^\Gamma)\simeq K_* (C^*_r \Gamma)$.

\begin{lemma}\label{lemma:independence-L}
If $\cW$ and $\cW^\prime$ are two mezzoperversities on $\sing X$ then
\begin{equation*}%\label{equality-aps}
\Ind_{{\rm APS}} (\alpha,\cW)=  \Ind_{{\rm APS}} (\alpha,\cW^\prime)\quad\text{in}\quad % K_* (C^* (\sing X_\Gamma)^\Gamma)=
K_* (C^*_r \Gamma).
\end{equation*}
\end{lemma}

\begin{proof}
Consider the stratified manifolds with corners $$\sM:=\sing M\times [0,1]_t\quad\text{and}\quad \sX:=(\sing X\times [0,1])\times [0,1]_t.$$
We stratify $(\sing X\times [0,1])\times [0,1]_s$ as we did in the proof of Proposition \ref{prop:well-def-on-S};
the  mezzoperversity $\cW$ on $\sing X$ induces a product mezzoperversity, denoted $\cW\times [0,1]$, on 
$(\sing X\times [0,1])\times \{0\}$; similarly, the mezzoperversity $\cW^\prime$ on $\sing X$
induces a mezzoperversity, denoted $\cW^\prime \times [0,1]$ on 
$(\sing X\times [0,1])\times \{1\}$. We know that there is a mezzoperversity $\sW$ on $\sX\equiv (\sing X\times [0,1])\times [0,1]_t$ interpolating between 
$\cW\times [0,1]$ and $\cW^\prime\times [0,1]$.
Similarly, we stratify $-(\sing M\times [0,1]_t)$ as in  Proposition \ref{prop:well-def-on-S}; let $\Phi:\sM\to \sX$
the map $\Phi(m,t)=\phi(m)$; the mezzoperversity $\Phi^{\sharp} \sW$ interpolates between $\phi^{\sharp} (\cW\times [0,1])$
and $\phi^{\sharp} (\cW^\prime \times [0,1])$.
Consider now the stratified manifold with corners
$$ \sZ:= -\sM\sqcup \sX\equiv \left( -(\sing M\times [0,1]_t )\right)\sqcup \left((\sing X\times [0,1])\times [0,1]_t\right).$$
This  is a Cheeger space with corners. Consider the following  boundary hypersurfaces of $\sZ$:
$$F= \left( -\sing M\times \{s=0\} \sqcup ((\sing X\times [0,1])\times \{s=0\}) \right) \sqcup \left( -\sing M\times \{s=1\} \sqcup ((\sing X\times [0,1])\times \{s=1\}) \right)
$$
and
$$G= -(\pa \sing M\times[0,1]_s) \sqcup (\pa (\sing X\times [0,1])\times [0,1]_s)
.$$
Consider the signature operator on $\sZ:=-\sM\sqcup \sX$; using appropriate Hilsum-Skandalis
perturbations we can perturb this operator and make it invertible at $G$. To fix notation, let us assume that
$\sing X$ is odd dimensional, so that $\sing M$ and $\sing X\times [0,1]$
 are even dimensional.
We can define a bivariant class 
$B\in KK_1 (C_F(\sZ),C^*_r \Gamma)$, 
with $C_F(\sZ)$ denoting the continuous functions on $\sZ$
which vanish on $F$ \footnote{Without further
hypothesis we could only define a
bivariant class in $KK_1 (C_{\pa 
\sZ} (\sZ),C^*_r \Gamma)$}. Consider $\pi^F: F\to {\rm point}$
and $\pi^{\sZ}:\sZ\to {\rm point}$. 
Denote by $\iota$ the natural inclusion $F\hookrightarrow \sZ:=-\sM\sqcup \sX$
and by $q:\sZ\to F$ the restriction map to $F$. Obviously $\pi^F=\pi^{\sZ}\circ \iota.$
Consider the semi-split short exact sequence
$$0\to C_F(\sZ)\xrightarrow{j} C (\sZ)\xrightarrow{q}  C(F)\to 0$$
Part of the associated long exact sequence in K-theory is
$$KK_1 (C_F(\sZ),C^*_\Gamma)\xrightarrow{\delta} KK_0 (C (F),C^*_r \Gamma)\xrightarrow{\iota_*}
KK_0 (C(\sZ),C^*_\Gamma)\,.$$
In particular, by exactness,  $\iota_* \circ \delta=0$.
Then,  on the one hand a classic argument shows that 
$$\pi^F_* (\delta B)= \Ind_{{\rm APS}} (\alpha,\cW)- \Ind_{{\rm APS}} (\alpha,\cW^\prime)\quad\text{in}\quad  KK_0 (\bbC,C^*_r \Gamma)=
K_0 (C^*_r \Gamma)$$
and, on the other hand, $\pi^F_* (\delta B)=\pi^{\sZ}_* \circ \iota_* (\delta B)=\pi^{\sZ}_* \circ \iota_* \circ \delta (B)=0$ by exactness.
Thus 
$$\Ind_{{\rm APS}} (\alpha,\cW)- \Ind_{{\rm APS}} (\alpha,\cW^\prime)=0$$
as required.
\end{proof}

We shall need a generalization of the previous Lemma. Fix an incomplete iterated conic metric 
on $\sing X\times [0,1]$ (product-type near the boundary) and assume that $\cW_{\sing X\times [0,1]}$
is a mezzoperversity for the resulting signature operator on $\sing X\times [0,1]$; thus $\cW_{\sing X\times [0,1]}$
is {\it not} necessarily the product mezzoperverity associated to a mezzoperversity on $\sing X$.
We can still define an index class $\Ind_{{\rm APS}} (\alpha, \cW_{\sing X\times [0,1]})$ by considering $\phi^{\sharp} (\cW_{\sing X\times [0,1]})$ on $\sing M$ and proceeding as above.

\begin{lemma}\label{lemma:independence-L-bis}
If $\cW_{\sing X\times [0,1]} $ and $\cW_{\sing X\times [0,1]} ^\prime$ are two mezzoperversities on $\sing X\times [0,1]$ then
\begin{equation*}%\label{equality-aps}
\Ind_{{\rm APS}} (\alpha,\cW_{\sing X\times [0,1]} )=  \Ind_{{\rm APS}} (\alpha,\cW_{\sing X\times [0,1]} ^\prime)\quad\text{in}\quad  K_* (C^*_r \Gamma).
\end{equation*}
\end{lemma}

\begin{proof}
The proof given for Lemma \ref{lemma:independence-L} can be easily adapted to this more general case.
\end{proof}

We now describe a gluing theorem following Bunke \cite{Bunke}.\\
Let $\sing Z$ be a stratified space with boundary and $\sing H$ a compact hypersurface transverse to the stratification that does not meet the boundary of $\sing Z.$ We can view $\sing Z$ as two stratified spaces with boundary glued along $\sing H,$
\begin{equation*}
	\sing Z = \sing Z^1 \bigcup_{ \sing H}  \sing Z^2
\end{equation*}
We decompose a   $\Gamma$-cover of $ \sing Z$ accordingly:
\begin{equation*}
	 \sing Z_{\Gamma} =  \sing Z_{\Gamma}^1 \bigcup_{ \sing H_\Gamma}  \sing Z_{\Gamma}^2.
\end{equation*}
We assume that $ \sing Z$ is Cheeger, we fix an iterated incomplete edge metric $g$ which is
of product type near $ \sing H$; finally,  we fix a selfdual mezzoperversity $\mathcal{W}$ adapted to $g$.
This restricts to a selfdual mezzoperversity on the hypersurface $\sing H$ which we denote $\mathcal{W}^H$. 
Similarly, we obtain selfdual mezzoperversities $\sW^1$ on $\sing Z^1$ and  $\sW^2$ on $\sing Z^2$.
We lift all these structures to the $\Gamma$-covers with minimal change of notation.

Let $D_{\Gamma}$ be the signature operator  on $Z_\Gamma$, the regular part of $\sing Z_\Gamma$. 
We assume that a trivializing perturbation $Q_{\pa}$ of the boundary operator has been fixed; the latter gives a grafted
perturbation $Q_\infty$ on the associated manifold with cylindrical ends and thus an index class $\Ind_{{\rm APS}} (D_\Gamma^{\mathcal{W}},Q_\pa )\in K_* (C^*_r \Gamma)$, where the canonical 
isomorphism $K_* (C^* (\sing Z_\Gamma)^\Gamma)\simeq  K_* (C^*_r \Gamma)$ has been used. We can assume,
without loss of generality, that the perturbation $Q_\infty$ is localized  away from $\sing H_\Gamma$.
 
The signature 
operator near $H_\Gamma$, the regular part of  $\sing H_\Gamma$, will decompose in the usual
way, given that the metric is of product type near $H_\Gamma$.
Let $C_H$ be a perturbation of $D_{\Gamma}^{\mathcal{W}^H}$ such that $D_{\Gamma}^{\mathcal{W}^H}+ C_H$ is invertible
\footnote{a simple argument using the cobordism invariance of the signature index class with
Cheeger boundary conditions shows that such a  perturbation always exists.}
and let
\begin{equation*}
	D^{\sW^1}_{\Gamma,\infty} + C_{H,\infty}^1, \quad
	D^{\sW^2}_{\Gamma,\infty} + C_{H,\infty}^2
\end{equation*}
be perturbed differential operators on the spaces obtained from $\sing Z^1_{\Gamma},$ $\sing Z^2_{\Gamma}$ by attaching an infinite half-cylinder along $\sing H_{\Gamma}.$ We obtain in this way well defined index classes 
$$\Ind_{{\rm APS}} (D^{\sW^1}_{\Gamma}, Q^1_\pa \sqcup C_{H})\,,\;\;
	\Ind_{{\rm APS}}(D^{\sW^2}_{\Gamma}, Q^2_\pa\sqcup C_{H})\quad\text{in}\quad K_* (C^*_r \Gamma)$$
	where $Q^j_\pa$ is $Q_\pa$ restricted  to $\pa \sing Z_\Gamma\cap \sing Z^j_\Gamma$.

\begin{proposition}\label{gluing}(Gluing)
With notation as above, the index classes satisfy
\begin{equation*}
	\Ind_{{\rm APS}} (D_{\Gamma}^{\cW}, Q_\pa) = 
\Ind_{{\rm APS}} (D^{\sW^1}_{\Gamma}, Q^1_\pa \sqcup C_{H}) + 
	\Ind_{{\rm APS}} (D^{\sW^2}_{\Gamma}, Q^2_\pa\sqcup C_{H})\quad\text{in}\quad K_* (C^*_r \Gamma)
\end{equation*}
If $\sing Z$ is without boundary, then  
\begin{equation*}
	\Ind(D_{\Gamma}^{\cW}) = 
\Ind_{{\rm APS}} (D^{\sW^1}_{\Gamma},  C_{H}) + 
	\Ind_{{\rm APS}} (D^{\sW^2}_{\Gamma},  C_{H})\quad\text{in}\quad K_* (C^*_r \Gamma)
\end{equation*}
\end{proposition}

\begin{proof}
A happy byproduct of the functional analytic nature of Bunke's proof is that it applies almost unchanged to our setting.
Let $D$ denote either of $D^{\sW^j}_{\Gamma,\infty} + C_{H,\infty}^j,$ $j\in\{1,2\},$ with $\cD(D)$ its self-adjoint domain.
Replace the definition of the spaces $H^\ell,$ $\ell\geq 0,$ in \cite[(2)]{Bunke} by
\begin{equation*}
	H^\ell = \cD(D^\ell), \quad
	\norm{\phi}_{\ell}^2 = \sum_{k=0}^{\ell} \norm{D^k\phi}_{L^2}^2
\end{equation*}
where the norm on the right hand side is the pointwise norm coming from the Hilbert $C^*$-module structure.

Note that since $D^2+\Id$ is invertible we can take a compact exhaustion of the regular part of our space, approximate the constant function one, and find a non-negative $f \in \CIc(Z^j_{\Gamma,\infty})$ such that $D^2+f$ is invertible, thereby satisfying Bunke's assumption 1.

With these conventions, the analytical results in \S1.2 of Bunke now hold verbatim save that the expressions $R(\lambda)\mathrm{grad}(f)R(\lambda)$ should be replaced by 
\begin{equation*}
	R(\lambda)(\mathrm{grad}(f)+ [C,f])R(\lambda)
\end{equation*}
where $C$ is the perturbation at $H.$ This replacement is still a compact operator and hence the argument in Proposition 1.13 of Bunke yields well-defined index classes. The argument in Theorem 1.15 of Bunke then yields the equality of the index classes we seek, once we take into account that the index of a translation invariant operator on the infinite cylinder vanishes.
\end{proof}

Recall that an articulated stratified space without boundary $\sing L$ is the (entirety of the) boundary of a stratified space with corners. Thus $\sing L$ is a finite union of stratified spaces with corners together with identifications of their boundary faces and the absence of boundary says that there are no unmatched faces. 
If $\sing L$ is the boundary of a Cheeger space, $\sing L = \pa \sing X,$ so in particular each of its constituent stratified spaces with corners is a Cheeger space, then a choice of mezzoperversity on $\sing X$ induces compatible mezzoperversities on the constituents of $\sing L,$ and the boundary identifications (which are stratified diffeomorphisms) give rise to Hilsum-Skandalis perturbations, and so we have an index class,
\begin{equation*}
	\Ind((\pa\sing X)_{\Gamma}) \in K_*(C^*_r\Gamma),
\end{equation*}
where $\Gamma$ is the fundamental group of $\sing X.$

\begin{lemma}\label{lem:CobInvCorners}
If $\sing X$ is a Cheeger space with corners then
\begin{equation*}
	\Ind((\pa \sing X)_{\Gamma})  =0 \Min K_*(C^*_r\Gamma)[\tfrac12].
\end{equation*}
\end{lemma}

\begin{proof}
Our convention is that every boundary hypersurface $\sing M$ of $\sing X$ is collared, i.e., has a neighborhood of the form $[0,1)_{\rho_{\sing M}} \times \sing M$ in $\sing X,$ consistent with the stratification of $\sing X.$ We refer to $\rho_{\sing M}$ as a boundary defining function for $\sing M.$
By a `total boundary defining function' for $\sing X,$ we mean a function $\rho_{\pa \sing X}$ obtained by taking the product of boundary defining functions, one per boundary hypersurface of $\sing X.$
Since the boundary hypersurfaces are collared, for all $\eps>0$ sufficiently small the set $\{ \rho_{\pa \sing X} \geq \eps \}$ is a stratified space with boundary and $\pa \sing X$ can be obtained from $\pa\{ \rho_{\pa \sing X} \geq \eps \}$ by `introducing corners' (i.e., partitioning it and considering as an articulated manifold).
Cobordism invariance of the signature on Cheeger spaces \cite[Theorem 4.8]{ALMP:Novikov} implies that the signature of $\pa\{ \rho_{\pa \sing X} \geq \eps \}$ vanishes in $K_*(C^*_r\Gamma)[\tfrac12]$ (see Remark \ref{rmk:12Enough}) and then Proposition \ref{gluing} implies that the signature of $\pa \sing X$ vanishes.
\end{proof}

The following Proposition is the main result of this Subsection:

\begin{proposition} 
Let $\cW$ be a mezzoperversity for $\sing X$. Let
%\begin{equation*}
\begin{align*}
&\alpha_1=\left( (\sing M_1; \sing X, \sing X'_1) \xlra{(\phi_1;\id, \psi_1)}
	(\sing X \times [0,1]; \sing X \times \{0\}, \sing X \times \{1\}) \xlra{\id}
	\sing X \times I \right) \\
	&\alpha_2=\left( (\sing M_2; \sing X, \sing X'_2) \xlra{(\phi_2;\id, \psi_2)}
	(\sing X \times [0,1]; \sing X \times \{0\}, \sing X \times \{1\}) \xlra{\id}
	\sing X \times I \right)
\end{align*}
%\end{equation*}
be two elements in $\cL_{BQ}(\sing X\times [0,1])$. Assume that $[\alpha_1]=[\alpha_2]$ in  $\tL_{BQ}(\sing X\times [0,1])$.
Then
$$\Ind_{{\rm APS}} (\alpha_1,
\cW)=\Ind_{{\rm APS}}(\alpha_2,\cW).$$
Consequently, using also the Lemma  \ref{lemma:independence-L}, we can define a  map
$$\tL_{BQ}(\sing X\times [0,1])\ni \zeta\longrightarrow \Ind_{{\rm APS}} (\zeta)\in K_{\dim \sing X + 1} (C^*_r \Gamma)$$
by setting $\Ind_{{\rm APS}}(\zeta):= \Ind_{{\rm APS}}(\alpha,\cW)$ for any representative  $\alpha$ of $\zeta$, $[\alpha]=\zeta$,
and any choice
of mezzoperversity $\cW$. This map is a homomorphism of abelian groups.
\end{proposition}

\begin{proof}
By Lemma \ref{lemma:special-cycles} and the gluing formula, Proposition \ref{gluing}, it suffices to show that if $\alpha_1$ is null bordant then $\Ind_{APS}(\alpha_1, \cW)=0.$

Let 
\begin{equation*}
	(\sing N; \pa_1 \sing N, \pa_2 \sing N) \xlra{\Phi}
	(\sing Z; \pa_1 \sing Z, \pa_2 \sing Z) \xlra{\Omega}
	\sing (X \times I) \times I
\end{equation*}
be a null bordism of $\alpha.$ Thus, $\Phi$ is BQ-normal, $\Phi|:\pa_2\sing N \lra \pa_2\sing Z$ is a BQ-equivalence, $\Omega$ is BQ-transverse, and
\begin{multline*}
	\lrpar{
	(\pa_1\sing N, \pa_{12}\sing N) \xlra{\Phi|} 
	(\pa_1 \sing Z, \pa_{12}\sing Z) \xlra{\pi\circ\Omega|} 
	\sing X}\\
	=
	\lrpar{
	(\sing M_1; \sing X, \sing X'_1) \xlra{(\phi_1;\id, \psi_1)}
	(\sing X \times [0,1]; \sing X \times \{0\}, \sing X \times \{1\}) \xlra{\id}
	\sing X \times I 
	}.
\end{multline*}
By Lemma \ref{lem:CobInvCorners}, we know that 
\begin{equation*}
	\Ind ((\pa_1 \sing N\xrightarrow{\pa_1\Phi}\pa_1 \sing Z)
		\cup (\pa_2 \sing N\xrightarrow{\pa_2\Phi}\pa_2 \sing Z),
	(\pa_1\Phi^{\sharp} (\cW^{\pa_1\sing Z})\sqcup 
	\cW^{\pa_1\sing Z})
	\sqcup (\pa_2\Phi^{\sharp} (\cW^{\pa_2\sing Z})
	\sqcup \cW^{\pa_2 \sing Z}))=0
\end{equation*}
for any mezzoperversity $\cW^{\sing Z}$ on $\sing Z.$
By Proposition \ref{gluing}, we can write this as the sum of two APS indices,
\begin{equation*}
	\Ind_{APS}(\pa_1 \sing N\xrightarrow{\pa_1\Phi}\pa_1 \sing Z, 
	\pa_1\Phi^{\sharp} (\cW^{\pa_1\sing Z})\sqcup 
	\cW^{\pa_1\sing Z})
	+\Ind_{APS}(\pa_2 \sing N\xrightarrow{\pa_2\Phi}\pa_2 \sing Z, 
	\pa_2\Phi^{\sharp} (\cW^{\pa_2\sing Z})\sqcup 
	\cW^{\pa_2\sing Z})
\end{equation*}
However, the second summand is equal to zero since $\Phi|:\pa_2\sing N \lra \pa_2\sing Z$ is a BQ-equivalence, and hence so is the first summand.

The fact that $\Ind_{{\rm APS}} $ is a homomorphism of abelian groups follows from  Lemma \ref{lem:SumLemma}
and the gluing formula. The Proposition is proved.

\end{proof}

We can use the gluing result in Proposition \ref{gluing} to simplify the APS index map from an $\cL$-cycle $\alpha \in \cL_{\BQ}(\sing X \times I)$ if it restricts to be a diffeomorphism on the boundary.

\begin{proposition}\label{aps-versus-closed}
Let $\alpha \in \cL_{\BQ}(\sing X \times I)$ be given by
\begin{equation*}
	\alpha
	=\left( (\sing M; \sing X, \sing X') \xlra{(\phi;\id, \psi)}
	(\sing X \times [0,1]; \sing X \times \{0\}, \sing X \times \{1\}) \xlra{\id}
	\sing X \times I \right)
\end{equation*}
where $\psi$ is a diffeomorphism $\sing X' \lra \sing X.$
Let $G(\alpha) \in \cL_{\BQ}(\sing X \times \bbS^1)$ be the $\cL_{\BQ}$-cycle given by
\begin{equation*}
	G(\alpha)
	=\left( G(\sing M) = \sing M/ (\sing X\sim_{\psi}\sing X') \xlra{G(\phi)}
	\sing X \times [0,1]/ (\sing X \times \{0\}\sim \sing X \times \{1\}) \xlra{\id}
	\sing X \times \bbS^1 \right).
\end{equation*}
Given a mezzoperversity $\cW$ on $\sing X$ we have
\begin{equation*}
%\Ind_{\Gamma,{\rm APS}} (\alpha, \cW) = \Ind_{\Gamma} (G(\alpha), \cW) \Min K_*(C^*_r\Gamma)
	\Ind_{{\rm APS}} (\alpha, \cW) = \Ind  (G(\alpha), \cW) \Min K_*(C^*_r\Gamma)
\end{equation*}
where on the left $\cW$ is lifted to $\sing X \times [0,1]$ and pulled-back to $\sing M,$ while on the right $\cW$ is lifted to $\sing X \times \bbS^1$ and pulled-back to $G(\sing M).$ Note that the index map on the right does not require boundary conditions; put differently, this is the index class of a cycle involving Cheeger spaces without boundary.
\end{proposition}

\begin{remark}
The index of $(G(\alpha), \cW)$ is, by definition, the index of the signature operator on 
\begin{equation*}
	G(\sing M) \cup (\sing X \times \bbS^1)
\end{equation*}
(twisted using a reference map to $B\Gamma,$ with $\Gamma = \pi_1\sing X$). Another application of Proposition \ref{gluing} shows that this index coincides with that of the signature operator on
\begin{equation*}
	\sing M \bigsqcup_{\substack{
	\sing X \sim \sing X \times \{0\}\\\sing X'\sim_{\psi} \sing X \times \{1\}}}
	\sing X \times I,
\end{equation*}
as in the original definition of Higson-Roe \cite{higson-roeIII}.
\end{remark}

%\footnote{
%\textcolor{m}{From Paolo: doesn't the same proof establishes the equality $	\Ind_{{\rm APS}} (\alpha, \cW) = \Ind(H(\alpha), \cW)$ with 
%$$H(\alpha)
%	=\left( H(\sing M) = \sing M\cup_{(\id,\psi)}  
%	\sing X \times [0,1]\xlra{H(\phi)} \sing X \times \bbS^1 \xlra{\id}
%	\sing X \times \bbS^1 \right). \,?$$
%	\red{I think my statement is maybe better, in that it is connected with Higson-Roe. Can you please look into this ?}}
%	\blue{I don't think this is quite the same as Higson-Roe. I think our definition of $\Ind$ would assign to $H(\alpha)$ the index of an operator on 
%	$$ H(\sing M) \cup \sing X \times \bbS^1
%	= ( \sing M\cup_{(\id,\psi)}  
%	\sing X \times [0,1] ) \cup (\sing X \times \bbS^1) $$
%	}I've added a remark which says what I think you wanted to say.}
%
\begin{proof}
Let $\sing N \subseteq G(\sing M)$ and $\sing Y \subseteq \sing X \times \bbS^1$ be the image of the boundary of $\sing M,$ respectively of $\sing X \times [0,1],$ under the identification maps $\sing M \lra G(\sing M),$ $\sing X \times [0,1] \lra \sing X \times \bbS^1.$
Without loss of generality we assume that $G(\phi)$ is collared near these subsets, i.e., that there are neighborhoods on which
\begin{equation*}
	G(\phi)|: \sing N \times (-1,1) \lra \sing Y \times (-1,1)
\end{equation*}
is the identity on the second factor. We assume that the stratifications respect the product structure of these neighborhoods.

We will apply our gluing result to this situation. Let $\sing Z = G(\sing M) \sqcup (\sing X \times \bbS^1)$ endowed with the natural map to $B\Gamma,$ $\Gamma = \pi_1\sing X,$ and let 
\begin{equation*}
	\sing H 
	= (\sing N \times \{-\tfrac12,\tfrac12\})
	\sqcup (\sing Y \times \{-\tfrac12,\tfrac12\}).
\end{equation*}
Given a mezzoperversity $\sW$ on $\sing X$ adapted to a wedge metric on $\sing Z$ that respects the product decomposition of the neighborhood 
$(\sing N \sqcup \sing Y) \times (-1,1),$ we can apply Proposition \ref{gluing} to see that
\begin{equation*}
	\Ind(D_{\Gamma}^{\cW};\sing Z)
	=
	\Ind_{{\rm APS}}(D_{\Gamma}^{\cW^1}, C_H;\sing M \sqcup (\sing X \times [0,1]))
	+\Ind_{{\rm APS}}(D_{\Gamma}^{\cW^1}, C_H;(\sing N \sqcup \sing Y) \times (-\tfrac12,\tfrac12))
\end{equation*}
where $C_H$ is the Hilsum-Skandalis perturbation
and we have used more explicit notation than is our wont.
Finally note that the first summand on the right hand side is $\Ind(\alpha, \cW),$ and the second term vanishes, as it is the index of an invertible translation-invariant operator on an infinite cylinder (indeed, as we have already remarked,
even though we denote our classes as
APS classes, they are really classes on manifolds with cylindrical ends).
\end{proof}

\subsection{Mapping stratified surgery to analysis}$\,$\\
At this point we have shown that all of the maps in the following diagram are well defined and independent of the choice of mezzoperversity:
\begin{equation}\label{eq:SurgToAnOne}
	\xymatrix{
	\tL_{\BQ}(\sing X \times I) \ar@{..>}[r] \ar[d]^-{\Ind_{{\rm APS}}} & 
	\tS_{\BQ}(\sing X) \ar[r]^-{\eta}  \ar[d]^-{\rho} &
	\tN_{\BQ}(\sing X) \ar[r]^-{\theta} \ar[d]^-{\beta} &
	\tL_{\BQ}(\sing X) \ar[d]^-{\Ind_{{\rm APS}}}\\
	\tK_{\dim \sing X+1}(C^*_r\Gamma)[\tfrac12] \ar[r] &
	\tK_{\dim \sing X+1}(D^* (\cov{\sing X})^\Gamma)[\tfrac12] \ar[r] &
	\tK_{\dim \sing X}(\sing X)[\tfrac12] \ar[r] &
	\tK_{\dim \sing X}(C^*_r\Gamma)[\tfrac12] }
\end{equation}
We now establish the commutativity of this diagram.

The key fact for establishing commutativity of the first square is the behaviour of the rho class under composition (\cite[Theorem 9.1]{Wahl_higher_rho}, \cite[(4.14)]{PS-akt}):

\begin{proposition}\label{prop:composition-rho}
Let $\sing L,$ $\sing M,$ $\sing V,$ be Cheeger spaces and
\begin{equation*}
	\sing M \xlra f \sing V, \quad
	\sing L \xlra g \sing M
\end{equation*}
transverse stratified homotopy equivalences. We fix a self-dual mezzoperversity $\cW$ on $\sing V$
and we consider the induced mezzoperversities $f^{\sharp} \cW$ on $\sing M$ and $g^{\sharp} (f^{\sharp} \cW)$
on $\sing L$. If $\sing V_\Gamma$ is a $\Gamma$-covering of $\sing V$ then we lift these mezzoperversities
to the induced coverings $f^* \sing V_\Gamma$ on $M$ and $g^* (f^* (\sing V_\Gamma))$ on $L$.
The following identity holds in $K_* (D^* (\sing V_\Gamma)^\Gamma)$:
\begin{equation*}
	\rho(\sing L \xlra{f\circ g} \sing V, \cW)] 
	+ \wt f_*(\rho( \sing M \xlra{\id} \sing M, f^{\sharp} \cW ))
	= \wt f_*(\rho (\sing L \xlra{g} \sing M,f^{\sharp} \cW )) 
	+ \rho(\sing M \xlra f \sing V,\cW].
\end{equation*}
where in the first summand on the right-hand side it is  the rho class of the perturbed signature operator 
on the covering  $g^* (f^* (\sing V_\Gamma)) \sqcup f^* (\sing V_\Gamma)$ that appears.\\
Consequently, with a small abuse of notation, we have
\begin{equation}\label{eq:compose-rho}
	\rho[\sing L \xlra{f\circ g} \sing V] 
	+ \wt f_*(\rho [\sing M \xlra{\id} \sing M])
	= \wt f_*(\rho [\sing L \xlra{g} \sing M]) 
	+ \rho[\sing M \xlra f \sing V].
\end{equation}
\end{proposition}

\begin{remark}
It can be shown that the $\rho$-invariant of the identity map vanishes, but we will not need this here.
\end{remark}

\begin{proof}
The proof given in  \cite{PS-akt}, based in turn on the proof of \cite[Proposition 7.1]{Wahl_higher_rho}
and on the delocalized APS index theorem, applies to the present situation.
\end{proof}

Recall now how $\tL_{\BQ}(\sing X \times I)$  acts  on $ \tS_{\BQ}(\sing X)$.
If $[\alpha] \in \tL_{\BQ}(\sing X \times I)$ and $[\beta] \in \tS_{\BQ}(\sing X)$ then we can choose representatives of the form
\begin{equation*}
\begin{gathered}
	\beta: \sing M \xlra{f} \sing X \\
	\alpha: (\sing W; \sing M, \sing M') \xlra{(\phi;\id, \phi_2)}
	(\sing M \times [0,1]; \sing M \times \{0\}, \sing M \times \{1\}) \xlra{\id}
	\sing M \times I,
\end{gathered}
\end{equation*}
and then the class of $f\circ \phi_2: \sing M' \lra \sing X$ in $\tS_{\BQ}(\sing X)$ is well-defined and denoted $\pa(\alpha)(\beta).$
The map
\begin{equation*}
	\xymatrix @R=1pt {
	\tL_{\BQ}(\sing X \times I) \times \tS_{\BQ}(\sing X) \ar[r] & \tS_{\BQ}(\sing X) \\
	([\alpha],[\beta]) \ar@{|->}[r] & \pa(\alpha)(\beta) }
\end{equation*}
defines the  group action of the Browder-Quinn L-group of $\sing X \times I$ on the structure set of $\sing X.$
In order to show that the first diagram in \eqref{eq:SurgToAnOne} commutes we need to show that 
\begin{equation}\label{first-diagram}
\rho (\pa(\alpha)(\beta))- \rho (\beta)=\iota_* (\Ind_{{\rm APS}} (\alpha))
\end{equation}
with $\Ind_{{\rm APS}} (\alpha)\in \tK_{\dim \sing X+1}(C^* (\sing X_\Gamma)^\Gamma)[\tfrac12] =
\tK_{\dim \sing X+1}(C^*_r\Gamma)[\tfrac12] $ and $\iota: C^* (\sing X_\Gamma)^\Gamma\to D^* (\sing X_\Gamma)^\Gamma$
the natural inclusion.
The left hand side of \eqref{first-diagram} is, by definition,
$$\rho [ \sing M' \xrightarrow{f\circ \phi_2} \sing X] - \rho [ \sing M \xlra{f} \sing X]\,.$$
We now apply Proposition \ref{prop:composition-rho} and obtain that this difference equals:
$$\tilde{f}_* \rho [\sing M' \xrightarrow{\phi_2} \sing M] - \tilde{f}_* \rho [\sing M \xrightarrow{\id} \sing M]$$
and a direct  application of the delocalized APS index theorem shows that this difference is precisely
equal to $\iota_* (\Ind_{\Gamma,{\rm APS}} (\alpha))$. This establishes the commutativity of the first square in the
diagram.\\
The second square  in the diagram is proved to commute exactly as in \cite{PS-akt}.\\
Finally, for the third square, we observe that the image of a class in $\tN_{\BQ}(\sing X)$
is a union of two closed Cheeger spaces and that for such an element in $\tL_{\BQ}(\sing X)$
the APS-index class is just the index class of Subsection \ref{subsect:index-classes}; the commutativity
of the third square then follows  by the functoriality of the boundary map in the Higson-Roe surgery
sequence.

\subsection{Mapping stratified surgery to analysis on all strata}\label{sect:all-strata}$\,$\\
The use of transverse maps in the definition of the Browder-Quinn surgery sequence implies that there are well-defined restriction maps from the long exact sequence of a stratified space to the corresponding sequence of a singular stratum.

Recall from \S\ref{sec:TransMaps} that if $Y \in \cS(\sing X)$ then the closure of $Y$ in $\hat X$ is a stratified space denoted $\sing Y.$ An $\cL$-cycle over $\sing X$ restricts to an $\cL$-cycle over $\sing Y,$ and a null bordism over $\sing X$ restricts to an $\cL$-cycle over $\sing Y.$ The restriction of a normal invariant or a Thom-Mather structure from $\sing X$ to $\sing Y$ is an $\cL$-cycle of the same type. Thus we have commutative diagrams
\begin{equation*}
	\xymatrix{
	\tL_{\BQ}(\sing X \times I) \ar@{..>}[r] \ar[d] & 
	\tS_{\BQ}(\sing X) \ar[r]^-{\eta}  \ar[d] &
	\tN_{\BQ}(\sing X) \ar[r]^-{\theta} \ar[d] &
	\tL_{\BQ}(\sing X) \ar[d] \\
	\tL_{\BQ}(\sing Y \times I) \ar@{..>}[r] &
	\tS_{\BQ}(\sing Y) \ar[r]^-{\eta}  &
	\tN_{\BQ}(\sing Y) \ar[r]^-{\theta} &
	\tL_{\BQ}(\sing Y) }
\end{equation*}
which we can extend arbitrarily to the left.
Note that the vertical arrows are generally neither injective nor surjective.

Let us introduce the abbreviations,
\begin{equation*}
\begin{gathered}
	\tK_{[j]}(C^*;{\cS(\sing X)})[\tfrac12] 
	= \tK_{\dim \sing X+j}(C^*_r\Gamma)[\tfrac12] \oplus
	\bigoplus_{Y \in \cS(X)}
	\tK_{\dim \sing Y+j}(C^*_r\Gamma(\sing Y))[\tfrac12], \\
	\tK_{[j]}(D^*;{\cS(\sing X)})[\tfrac12] 
	= 
	\tK_{\dim \sing X+j}(D^* ({\sing X}_{\Gamma})^{\Gamma})[\tfrac12] 
	\oplus
	\bigoplus_{Y \in \cS(X)}
	\tK_{\dim \sing Y+j}
		(D^* ({\sing Y}_{\Gamma(\sing Y)})^{\Gamma(\sing Y)})[\tfrac12] \\
	\tK_{[j]}(D^*/C^*;{\cS(\sing X)})[\tfrac12] 
	= \tK_{\dim \sing X+j-1}(\sing X)[\tfrac12]
	\oplus
	\bigoplus_{Y \in \cS(X)}
	\tK_{\dim \sing Y+j-1}(\sing Y)[\tfrac12].
\end{gathered}
\end{equation*}
By restricting to each singular stratum an making use of their respective commutative diagram \eqref{eq:SurgToAnOne} we end up with a combined diagram
\begin{equation}\label{eq:SurgToAnAll}
	\xymatrix{
	\tL_{\BQ}(\sing X \times I) \ar@{..>}[r] \ar[d]^-{\oplus\Ind}& 
	\tS_{\BQ}(\sing X) \ar[r]^-{\eta}  \ar[d]^-{\oplus\rho} &
	\tN_{\BQ}(\sing X) \ar[r]^-{\theta} \ar[d]^-{\oplus\beta} &
	\tL_{\BQ}(\sing X) \ar[d]^-{\oplus\Ind} \\
	\tK_{[1]}(C^*;{\cS(\sing X)})[\tfrac12] \ar[r] &
	\tK_{[1]}(D^*;{\cS(\sing X)})[\tfrac12] \ar[r] &
	\tK_{[1]}(D^*/C^*;{\cS(\sing X)})[\tfrac12] \ar[r] &
	\tK_{[0]}(C^*;{\cS(\sing X)})[\tfrac12] }
\end{equation}
$ $\\

\begin{remark}\label{rmk:SupNormal}
In \cite[\S12.4]{W-book}, for Witt spaces with simply connected links (also known as `supernormal spaces', see 
\cite{CW-CRAS}) %\footnote{these spaces are called {\em supernormal} \cite{W-book}, see also CW-CRAS}
we find that
\begin{equation*}
	\tL_{\BQ}(\sing X)[\tfrac12]
	= \tL(\bbZ\pi_1(\sing X))[\tfrac12] 
	\oplus
	\bigoplus_{Y \in \cS(X)}
	\tL(\bbZ\pi_1(\sing Y))[\tfrac12] 
\end{equation*}
(where $\sing Y,$ the `closed stratum', is the closure of $Y$ in $\sing X$ with the induced stratification)
so that in this case the vertical arrows in \eqref{eq:SurgToAnAll} map from the {\em algebraic} $L$-groups.
\end{remark}

\section{Further considerations}\label{sect:applications}

In \cite{Chang-Weinberger}, Chang and Weinberger use the surgery exact sequence of a manifold to show that torsion in its fundamental group implies the existence of infinitely many homotopy equivalent manifolds that are homeomorphically distinct.

In this section we use the argument in {\em loc. cit.} as a launching pad to discuss related topics. First, as an answer to `how to map in to/out of a BQ L-group?' we establish a couple of long exact sequences. Secondly we establish Atiyah's $L^2$-signature theorem for Cheeger spaces. Finally we combine these to discuss a version of the result of \cite{Chang-Weinberger} for Cheeger spaces.

%%%%%%%%%%%
\subsection{A long exact sequence for the Browder-Quinn L-groups} \label{sec:LESLBQ}
%%%%%%%%%%%

It is interesting to connect the Browder-Quinn L-groups with the usual (Wall) L-groups of a smooth manifold. If $\sing X$ is a smoothly stratified space then, on the one hand, recall that we can identify the regular part of $\sing X$ with the (interior of) the resolution of $\sing X,$ $\res X,$ and so the inclusion, $i,$ of the regular part induces
\begin{equation*}
	\tL_k(\res X) \xlra{i_*} \tL_{\BQ, d^{\res X}(k)}(\sing X),
\end{equation*}
where $d^{\res X}(k)$ is the dimension function that is equal to $k$ on the inverse image of the regular part of $\sing X.$
%\footnote{\textcolor{m}{So, any dimension function with that property is allowed ?} \textcolor{b}{Yes; but actually there's only one dimension function with this property, thanks to transversality.} \red{I am not sure I understand this explanation...}
%\textcolor{b}{Let's say that we have a transverse map $\sing M \xlra{F} \sing X$ and I tell you that the dimension of the stratum of $\sing M$ that maps into $\sing X^{\reg}$ is $q.$ If $Y$ is any stratum of $\sing X$ and $N$ is the stratum of $\sing M$ that maps into $Y,$ the dimension of $N$ is determined: it has to be $q-\dim Z,$ where $Z$ is the fiber of $Y$ in $\sing X$ (and by transversality also the fiber of $N$ in $\sing M.$ }}
On the other hand, if $X^{\dagger}$ is a minimal stratum (a stratum of greatest depth), and hence a closed manifold, then the fact that the Browder-Quinn L-groups are defined using transverse maps means that we have a natural restriction map
\begin{equation*}
	\tL_{\BQ,d}(\sing X) \xlra{R} \tL_{d_{X^{\dagger}}}(X^{\dagger}),
\end{equation*}
where $d_{X^{\dagger}}$ is the restriction of the dimension function to $X^{\dagger}.$\\

%\footnote{
%\textcolor{m}{Question B: does there exist a map, one way or the other, between the L-theory of  $X^\dagger$ and the
%L-theory of $\res X$ ?}
%\textcolor{b}{Yes; for example, if $\hat X$ has only one singular stratum, you could lift an L-cycle over $X^{\dagger}$ to an L-cycle over $\pa \res X$ and then use the inclusion of $\pa \res X$ into $\res X$ to get an L-cycle over $\res X.$ I have no idea how this is related to the maps in the sequences.} \red{OK, but is this dramatic ? (I mean the fact that we do not know
%how to relate to the sequences...)}
%\textcolor{b}{I'm not sure what you mean. Are you saying, this is a natural map we should be able to understand what it does topologically? or?}\red{I mean: we have problems in "mapping in" but if we compose this map you have described
%and the one from the L-theory of $\res X$ into the BQ-L-group
%of $\sing X$ we get a map from the L-group of $X^{\dagger}$ to the BQ-L-group of $\sing X$;
%maybe this is the zero map ? Or maybe is not well defined ? Or what ?}}
Both of these maps fit into a long exact sequence of Browder-Quinn L-groups. An example of the former is found in \cite{BQ} and of the latter in \cite[\S 6]{W-book}. We treat these as special cases of long exact sequences in L-groups associated with inclusions.\\

We recall from, e.g., \cite[Lemma 8.3.1]{Wall:DiffTop}, a standard construction of long exact sequences in cobordism. Suppose that $\alpha$ and $\beta$ denote two possible types of structure a manifold can have, and that a $\beta$ structure implies an $\alpha$ structure (for example, if $Y \subseteq X$ and $\alpha$ structure could be `is endowed with a map to $X$' and an $\beta$ structure `is endowed with a map to $Y$). Denote by $\Omega_n^{\alpha},$ $\Omega_n^{\beta}$ the cobordism groups of $n$-dimensional manifolds with the corresponding structure, and denote by $\Omega_n^{\alpha,\beta}$ the cobordism group of manifolds with boundary with an $\alpha$ structure and a compatible $\beta$ structure on the boundary then, with the obvious maps, there is a long exact sequence
\begin{equation*}
	\ldots \Omega_n^{\beta} 
	\lra \Omega_n^{\alpha} 
	\lra \Omega_n^{\alpha,\beta}
	\xlra{\pa}
	\Omega_{n-1}^{\beta} \lra \ldots.
\end{equation*}
The proof of exactness in {\em loc. cit.} does not depend on the specific structures $\alpha$ and $\beta$ and adapts easily to the situations we consider below.\\

{\bf One type of relative L-group.}$\,$\\
In \cite{Wall}, Wall explains how to associate to a map $h:V \lra W$ between two manifolds an L-group that moreover fits into a long exact sequence with the L-groups of $V$ and $W.$ We now observe that the same is true if $\sing V$ and $\sing W$ are smoothly stratified spaces and $h:\sing V \lra \sing W$ is a transverse map between them.
Let us denote by $\cL_{\BQ,d}(\sing V \xlra h \sing W)$ the set of commutative diagrams of the form
\begin{equation}\label{eq:RelDiags}
	\xymatrix{
	(\sing S; \pa_1 \sing S, \pa_2 \sing S) \ar[r]^-\phi &
	(\sing T; \pa_1 \sing T, \pa_2 \sing T) \ar[r]^-\eta &
	\sing W \\
	\pa_2 \sing S \ar@{^(->}[u] \ar[r]^-{\phi|} &
	\pa_2 \sing T \ar@{^(->}[u] \ar[r]^-{\eta|} &
	\sing V \ar[u]^h }
\end{equation}
where $\sing S$ is a smoothly stratified space with dimension function $d,$ and boundary $\pa_1\sing S \cup \pa_2\sing S,$ and similarly $\sing T,$ $\phi$ is a BQ-normal map and restricts to a BQ-equivalence between $\pa_1\sing S$ and $\pa_1\sing T$ and $\eta$ is BQ-transverse. As usual $\tL_{\BQ,d}(\sing V \xlra h \sing W)$ then denotes cobordism classes of such cycles.
As above, these groups fits into a long exact sequence
\begin{equation*}%\label{eq:LESRelL}
	\ldots \xlra{\pa}
	\tL_{\BQ,d}(\sing V) \xlra{h_*} 
	\tL_{\BQ,d}(\sing W)
	\xlra{j}
	\tL_{\BQ,d}(\sing V \xlra h \sing W)
	\xlra{\pa}
	\tL_{\BQ,d-1}(\sing V) \xlra{h_*} \ldots
\end{equation*}
where the maps are given by
\begin{equation*}
\begin{gathered}
	h_*\lrspar{
	(\sing S'; \pa \sing S') \xlra{\phi'} (\sing T', \pa \sing T') \xlra{\eta'} \sing V}
	=
	\lrspar{
	(\sing S'; \pa \sing S') \xlra{\phi'} 
	(\sing T', \pa \sing T') 
	\xlra{h\circ \eta'} \sing W}, \\
	j\lrspar{
	(\sing S''; \pa \sing S'') \xlra{\phi''} (\sing T'', \pa \sing T'') \xlra{\eta''} \sing W}
	=
    \left[
    \vcenter{\xymatrix{ 	
	(\sing S''; \pa \sing S'', \emptyset) \ar[r]^-{\phi''} &
	(\sing T''; \pa \sing T'', \emptyset) \ar[r]^-{\eta''} &
	\sing W \\
	\emptyset \ar@{^(->}[u] \ar[r] &
	\emptyset \ar@{^(->}[u] \ar[r] &
	\sing V \ar[u]^h }} \right], \\
	\pa \left[
    \vcenter{\xymatrix{ 	
	(\sing S; \pa_1 \sing S, \pa_2 \sing S) \ar[r]^-\phi &
	(\sing T; \pa_1 \sing T, \pa_2 \sing T) \ar[r]^-\eta &
	\sing W \\
	\pa_2 \sing S \ar@{^(->}[u] \ar[r]^-{\phi|} &
	\pa_2 \sing T \ar@{^(->}[u] \ar[r]^-{\eta|} &
	\sing V \ar[u]^h }} \right]
	= \lrspar{ (\pa_2\sing S; \pa(\pa_2 \sing S)) \xlra{\phi|}
	(\pa_2 \sing T; \pa(\pa_2\sing T)) \xlra{\eta|}
	\sing V }.
\end{gathered}
\end{equation*}
%
%Exactness follows from a standard argument, e.g.,\cite[Lemma 8.3.1]{Wall:DiffTop} (see also \cite[Theorem 8.4.4]{Wall:DiffTop}, \cite[Theorem 9.6]{Wall}).\\

If we apply this to the inclusion of the regular part, $i: \res X \lra \sing X,$ we obtain
\begin{equation}\label{eq:LESalaBQ}
	\ldots \xlra{\pa}
	\tL_{k}(\res X) \xlra{i_*} 
	\tL_{\BQ,d^{\res X}(k)}(\sing X)
	\xlra{j}
	\tL_{\BQ,d^{\res X}(k)}(\res X \xlra i \sing X)
	\xlra{\pa}
	\tL_{k-1}(\res X) \xlra{i_*} \ldots
\end{equation}
(cf. \cite[Proposition 4.8]{BQ}).\\

{\bf Another type of relative L-group.}$\,$\\
(For this sequence cf. \cite[Theorem 5.4]{Dovermann-Schultz}.)
Let $\sing X$ be a stratified space, $\sing \Sigma \subseteq \sing X$ a closed subset of $\sing X$ made up of a union of strata.
Let $\cL_{\BQ,d}(\sing X; \sing \Sigma)$ denote the $\cL$-cycles over $X$ with dimension function $d$ whose restriction to $\sing \Sigma$ is a BQ-equivalence, and $\tL_{\BQ,d}(\sing X;\sing\Sigma)$ the corresponding bordism classes.
We allow $\sing\Sigma = \emptyset$ for which
\begin{equation*}
	\tL_{\BQ,d}(\sing X;\emptyset) = 
	\tL_{\BQ,d}(\sing X).
\end{equation*}

Analogously to the above, if $\sing\Sigma' \subseteq \sing X$ is another closed subset of $\sing X$ made up of a union of strata, with $\sing \Sigma \subseteq \sing\Sigma'$ then there is a relative group $\tL_{\BQ,d}(\sing X;\sing \Sigma \subseteq \sing \Sigma')$ with classes represented by diagrams of the form
\begin{equation}\label{eq:DiagRel}
	(\sing M;\pa_0 \sing M, \pa_1\sing M) 
	\xlra{\phi} 
	(\sing N; \pa_0 \sing N, \pa_1 \sing N) 
	\xlra{\omega} 
	\sing X
\end{equation}
such that $\phi$ is BQ-normal, $\omega$ is BQ-transverse, $\pa_0\phi$ is a BQ-equivalence, $\pa_1\phi$ restricted to the preimage of $\sing\Sigma'$ is a BQ-equivalence, and $\phi$ restricted to the preimage of $\sing\Sigma$ is a BQ-equivalence.

There are natural inclusion maps
\begin{equation*}
	\tL_{\BQ,d}(\sing X; \sing\Sigma')
	\lra \tL_{\BQ,d}(\sing X; \sing\Sigma), \quad
	\tL_{\BQ,d}(\sing X; \sing\Sigma) \lra
	\tL_{\BQ,d}(\sing X;\sing \Sigma \subseteq \sing \Sigma')	
\end{equation*}
which fit into a long exact sequence
\begin{equation*}
	\ldots \lra \tL_{\BQ,d}(\sing X; \sing\Sigma')
	\lra \tL_{\BQ,d}(\sing X; \sing\Sigma)
	\lra
	\tL_{\BQ,d}(\sing X;\sing \Sigma \subseteq \sing \Sigma')	
	\xlra{\pa_1} \tL_{\BQ,d-1}(\sing X; \sing\Sigma') \lra \ldots.
\end{equation*}
Exactness of this sequence follows from the the usual construction of relative sequences in bordism, see \cite{Wall:DiffTop}.

Let us consider the case where $\sing\Sigma = \emptyset$ and $\sing \Sigma' = X^{\dagger},$ a minimal stratum of $\sing X.$ We point out that there are compatible restriction maps
\begin{equation*}
	\tL_{\BQ,d}(\sing X;\emptyset)
	\lra
	\tL_{d(Y)}(X^{\dagger}), \quad
	\tL_{\BQ,d}(\sing X;\emptyset \subseteq X^{\dagger})		
	\lra
	\tL_{d(Y)}(X^{\dagger}),
\end{equation*}
both of which are onto, since any $\cL$-cycle over $X^{\dagger}$ has a lift to an $\cL$-cycle over $\sing X,$
\begin{equation*}
	\lrpar{
	(M;\pa M) \xlra\phi (N, \pa N) \xlra\omega X^{\dagger}} \\
	\mapsto
	\lrpar{
	(\omega\circ\phi)^*\cT_{X^{\dagger}} \lra
	\omega^*\cT_{X^{\dagger}} \lra
	\lrpar{ \cT_{X^{\dagger}} \hookrightarrow } \sing X }
\end{equation*}
(where $\cT_{X^{\dagger}}$ is a tubular neighborhood of $X^{\dagger}$ in $\sing X$).
The restriction map
$\tL_{\BQ,d}(\sing X;\emptyset \subseteq X^{\dagger})	\lra \tL_{d(Y)}(X^{\dagger})$ is also injective.
Indeed, assume that \eqref{eq:DiagRel} is such that $\phi$ restricts to $X^{\dagger}$ to be a BQ-equivalence and consider
\begin{equation}\label{eq:DiagRelTriv}
	M\times [0,1] \xlra{\phi \times \id}
	N \times [0,1] \xlra{\omega\times \id}
	X \times [0,1].
\end{equation}
Since $\phi \times \id$ restricted to $\pa M \times [0,1] \cup M \times \{1\}$ is a BQ-equivalence over $X^{\dagger},$ we recognize \eqref{eq:DiagRelTriv} as a null bordism for \eqref{eq:DiagRel}. Thus we have established the long exact sequence
\begin{equation}\label{eq:LESalaW}
	\ldots \lra \tL_{\BQ,d}(\sing X; X^{\dagger})
	\lra \tL_{\BQ,d}(\sing X)
	\lra
	\tL_{d(X^{\dagger})}(X^{\dagger})
	\lra \tL_{\BQ,d-1}(\sing X; X^{\dagger}) \lra \ldots.
\end{equation}
%

%\footnote{
%\textcolor{m}{Question C: so, using the above sequence and composing the last map in \eqref{eq:LESalaW}
%with the next one, not appearing above, we have a homomorphism 
%$\tL_{d(X^{\dagger})}(X^{\dagger})\to \tL_{\BQ,d-1}(\sing X)$; is this correct ? Do you have a feeling on this homomorphism
%from $\tL_{d(X^{\dagger})}(X^{\dagger})$ to $\tL_{\BQ,d-1}(\sing X)$ ?}
%\textcolor{b}{The sequence is exact, so this homomorphism is the zero map.}
%%\textcolor{m}{ I assume we also have a homomorphism, call it $\phi:\tL_{d(X^{\dagger})+1}(X^{\dagger})\to \tL_{\BQ,d}(\sing X)$; is this correct ?}
%}

%\footnote{
%\textcolor{m}{Question D: the $L^1$-description of the algebraic L-theory of $\bbZ\pi_1 (\sing X)$, does it 
%give a homomorphism from $L_k  (\bbZ\pi_1 (\sing X))$ to the BQ-L group of $\sing X$ itself for an appropriate dimension function
%? }
%\textcolor{b}{The $L^1$-description of the algebraic L-theory of $\bbZ\pi_1 (\sing X)$ would be $\tL(W)$ for some closed manifold $W$ with $\pi_1(W) = \pi_1(\sing X).$ I don't know how to get a map between $W$ and $\sing X,$ and I don't know how else to get an induced map between BQ-L groups.}
%}

%\footnote{
%\textcolor{m}{Question D-bis: how are we going to use the long exact sequence you have proved ?
%For example for sufficient conditions ensuring the the maps $i_*$ and $\phi$ considered in the next subsection are indeed injective ?}
%\textcolor{b}{Excellent question. I was just focused on getting the long exact sequences based on our division of labor from my last day in Rome.}}
%
%%%%%%%%%
\subsection{Atiyah's $L^2$ signature theorem} \label{sec:Atiyah}$\,$\\
%%%%%%%%%
If $M$ is a closed even-dimensional manifold and $L \lra M$ is a regular cover with transformation group $\Gamma$ then given any elliptic differential operator on $M,$ Atiyah's $L^2$-index theorem asserts the equality of its index with the $\Gamma$-equivariant $L^2$-index of its lift to $L.$ Let $C^*\Gamma$ denote the maximal group $C^*$-algebra associated to $\Gamma$.
Both of these numeric indices can be obtained from the $C^*\Gamma$ index of the elliptic operator, an element
in $K_0 (C^*\Gamma)$, by applying
two traces: $\tau_{\Gamma}: K_{0} (C^*\Gamma)\to \bbC$ is obtained by extending  the trace 
$\tau_\Gamma:\bbC\Gamma\to \bbC$ given by
$$\tau_{\Gamma} (\sum_\gamma \alpha_\gamma \,\gamma):= \alpha_e$$
whereas $
	\tau: K_{0} (C^*\Gamma)\to \bbC$
is obtained by extending  the trace $\tau: \bbC\Gamma\to \bbC$,
$$\tau (\sum_\gamma \alpha_\gamma \,\gamma):= \sum_\gamma \alpha_\gamma\,.$$
Atiyah's theorem can be summarized as
\begin{equation*}
	(\tau_\Gamma - \tau) \circ \Ind_{\Gamma}=0.
\end{equation*}
In \cite[Theorem 6.5]{ALMP:Hodge}, it is shown that the signature operator admits a parametrix 
which is $\epsilon$-local; once we have this key information, Atiyah's original proof carries over to the setting of Cheeger spaces essentially unchanged. However, as some of the arguments will be useful below, we present instead a proof that follows \cite[Appendix]{Chang-Weinberger}.\\

Recall that whenever $\sing X$ is an even-dimensional Cheeger space and $r: \sing X \lra B\Gamma$ is the classifying map of a regular $\Gamma$-cover we have a signature class
\begin{equation*}
	\Ind(\sing X_{\Gamma}):=
	\Ind(D^{\sG(r)}) \in K_0(C^*\Gamma)\otimes_{\bbZ}\bbZ[\tfrac{1}{2}]\equiv
	K_0(C^*\Gamma)[\tfrac{1}{2}]
\end{equation*}
where $\sG(r)$ denotes the flat bundle of $C^*\Gamma$-modules corresponding to $r,$ and $D^{\sG(r)}$ is the twisted signature operator. (As already remarked around Theorem \ref{theo:main-bordism}, the index class
 in $ K_0(C^*\Gamma)$  is defined using a choice of a mezzoperversity but it is in fact independent of this choice
  in $ K_0(C^*\Gamma)[\frac{1}{2}]$, which is why we omit it from the notation.) Also in {\em loc. cit.}, for any topological space $L$ there are bordism groups
$\Omega^{\Che}_n(L)$ of pairs $(\sing M, f:\sing M \lra L)$ where $\sing M$ is a Cheeger space of dimension $n$ and $f$ is continuous. 
The signature class $\Ind(D^{\sG(r)})$ only depends on the bordism class 
\begin{equation*}
	[(\sing X, r:\sing X \lra B\Gamma)] \in
	\Omega^{\Che}_{\dim \sing X}(B\Gamma),
\end{equation*}
and defines a group homomorphism \cite[Corollary 5.11]{ALMP:Novikov}
\begin{equation*}
	\sigma_{\Gamma}: 
	\Omega^{\Che}_{\dim \sing X}(B\Gamma) \lra
	 K_0(C^*\Gamma)[\frac{1}{2}]
\end{equation*}

\begin{proposition}\label{prop:IndexInduction}
Any homomorphism $f:\Gamma_1 \lra \Gamma_2$ between discrete groups induces a commutative diagram
\begin{equation*}
	\xymatrix{
	\Omega^{\Che}_n(B\Gamma_1) \ar[r]^-{\sigma_{\Gamma_1}} \ar[d]^-{Bf_*} 
		& K_n(C^*\Gamma_1)[\tfrac{1}{2}] \ar[d]^-{f_*} \\
	\Omega^{\Che}_n(B\Gamma_2) \ar[r]^-{\sigma_{\Gamma_2}} 
		& K_n(C^*\Gamma_2)[\tfrac{1}{2}] }
\end{equation*}
\end{proposition}

\begin{proof}
Given $r:\sing X \lra B\Gamma_1$ we want to show
\begin{equation*}
	\Ind(D^{\sG((Bf)\circ r)})
	=f_*(\Ind(D^{\sG(r)}) )\in K_n(C^*\Gamma_2)[\frac{1}{2}].
\end{equation*}
For each $\Gamma_i,$ let
\begin{equation*}
	C^*\Gamma_i - \cU_{\Gamma_i} \lra B\Gamma_i
\end{equation*}
be the universal Mishchenko bundle and note that
\begin{equation*}
	(Bf)^*\cU_{\Gamma_2} 
	= \cU_{\Gamma_1} \otimes_{C^*\Gamma_1}C^*\Gamma_2,
\end{equation*}
where the tensor product makes use of $f_*,$ and correspondingly
\begin{equation*}
	\sG((Bf)\circ r) = \sG(r) \otimes_{C^*\Gamma_1}C^*\Gamma_2.
\end{equation*}
Thus we have
\begin{equation*}
	\Ind(D^{\sG((Bf)\circ r)})
	= \Ind(D^{\cG(r)})\otimes_{C^*\Gamma_1}C^*\Gamma_2.
\end{equation*}
On the other hand, the map 
\begin{equation*}
	K_n(C^*_r\Gamma)[\tfrac{1}{2}] = KK_n(\bbC, C^*\Gamma_1)[\tfrac{1}{2}]
	\xlra{f_*}
	KK_n(\bbC, C^*\Gamma_2)[\tfrac{1}{2}]= 
	K_n(C^*_r\Gamma_2)[\tfrac{1}{2}] 
\end{equation*}
is precisely obtained by tensoring KK-cycles $\cdot \mapsto \cdot \otimes_{C^*\Gamma_1}C^*\Gamma_2,$ and so the proof is complete.
\end{proof}

\begin{theorem}[Atiyah's $L^2$-signature theorem for Cheeger spaces]\label{thm:Atiyah}
If $\sing X$ is an even-dimensional Cheeger space and $r:\sing X \lra BG$ is the classifying space of a regular cover, then
\begin{equation*}
	\tau_\Gamma(\Ind(D^{\sG(r)}))
	= \tau(\Ind(D^{\sG(r)})).
\end{equation*}
\end{theorem}

\begin{proof} (We follow \cite[Appendix]{Chang-Weinberger}.)

\noindent
Given any injective group homomorphism $f: \Gamma_1 \lra \Gamma_2$ we have a commutative diagram
\begin{equation}\label{eq:CommTraces}
	\xymatrix{
	& K(C^*\Gamma_1) \ar[dd]^{f_*} \ar[ld]_-{\tau} \ar[rd]^{\tau_\Gamma} &\\
	\bbZ & & \bbR \\
	& K(C^*\Gamma_2) \ar[lu]^-{\tau} \ar[ru]_{\tau_\Gamma} &}
\end{equation}
(see \cite[Lemma 2.22]{PS2})
which combined with Proposition \ref{prop:IndexInduction} yields a commutative diagram
\begin{equation}\label{eq:CommBordism}
	\xymatrix{
	& \Omega^{\Che}(B\Gamma_1) 
	\ar[dd]^-{Bf_*} \ar[ld]_-{\sigma} \ar[rd]^-{\sigma_{(2)}} &\\
	\bbZ & & \bbR \\
	& \Omega^{\Che}(B\Gamma_2) \ar[lu]^-{\sigma} \ar[ru]_-{\sigma_{(2)}} &}
\end{equation}
Now let 
\begin{equation*}
	\{e\} \xlra{i} G \xlra h A
\end{equation*}
be the inclusion of the identity into $G$ and an injection of $G$ into an acyclic group $A.$ (Recall that a group $A$ is called acyclic if $BA$ is an acyclic space and that any group has an injective homomorphism into an acyclic group, discrete if $G$ is discrete.)

As $BA$ is acyclic its suspension is contractible and since, by the Eilenberg-Steenrod axioms, generalized homology theories are stably invariant, they must vanish on $BA.$ From \cite{ALMP:Novikov} we know that $\Omega^{\Che}_*$ is a generalized homology theory and so
\begin{equation*}
	(B(hi))_*:\Omega^{\Che}_*(\pt) \lra 
	\Omega^{\Che}_*(BA)
\end{equation*}
is an isomorphism.

Thus from the commutative diagram
\begin{equation*}
	\xymatrix @R=50pt @C=50pt{
	& \Omega^{\Che}(\pt) 
	\ar[d]^-{Bi_*} \ar[ld]_-{\sigma} \ar[rd]^-{\sigma_{(2)}} & \\
	\bbZ & \Omega^{\Che}(BG) 
		\ar[d]^-{Bh_*}\ar[l]_-{\sigma} \ar[r]^-{\sigma_{(2)}}
	& \bbR \\
	& \Omega^{\Che}(BA) \ar[lu]^-{\sigma} \ar[ru]_-{\sigma_{(2)}} &}
\end{equation*}
the equality $\sigma = \sigma_{(2)}$ on $\Omega^{\Che}(BG)$ reduces to the equality on $\Omega^{\Che}(\pt)$ where it is immediate.

\end{proof}

%\red{Question: can't we prove Atiyah's theorem by proving directly that \eqref{eq:CommBordism} commutes ?
%I think we had thought about this a long time ago.\\ In any case, the proof
%is  a direct consequence of the finite propagation property we stablished in the Novikov paper.
%Indeed, from Theorem 6.5 in that paper (which is now published !) we learn that we can find
%$\epsilon$-local parametrices, but then Atiyah's original proof applies verbatim.}

%\begin{remark}
%For Witt spaces of depth one, Atiyah's theorem also follows from \cite{Piazza-Vertman}.
%\end{remark}

%%%%%%%%%
\subsection{Torsion elements and the cardinality of the BQ-structure set.}$\,$
%%%%%%%%%
In this subsection we adapt an argument of \cite{Chang-Weinberger} to the setting of Cheeger spaces.
For any Cheeger space $\sing X$ of odd dimension with $\pi_1\sing X = \Gamma$ we have a group homomorphism 
\begin{equation}\label{eq:DefAlpha}
	\xymatrix{
	\tL_{\BQ}(\sing X \times I) 
		\ar[rd]_{\Ind_{{\rm APS}}} \ar[rr]^-{\alpha} & &\bbC \\
	& K_0(C^*\Gamma) \ar[ru]_{\tau_\Gamma-\tau} &}
\end{equation}
where $\tau$ and $\tau_\Gamma$ are as in \S\ref{sec:Atiyah}, which is actually valued in $\bbR,$ because the index class is ``self-adjoint".
Combining Proposition \ref{aps-versus-closed} with Atiyah's $L^2$-signature theorem (Theorem \ref{thm:Atiyah}) we have the following result.

\begin{proposition}\label{prop:alpha}
The homomorphism $\alpha$ vanishes on the image of the map 
\begin{equation*}
	\theta: \tN_{\BQ}(\sing X \times I, \sing X \times \pa I) 
		\lra\tL_{\BQ}(\sing X \times I)
\end{equation*}
from the surgery exact sequence of $\sing X.$
\end{proposition}

By exactness of the surgery sequence this proposition says that $\alpha$ vanishes on those elements of $\tL_{\BQ}(\sing X \times I)$ that act trivially on $\tS_{\BQ}(\sing X).$ Conversely, if $x \in \tL_{\BQ}(\sing X \times I)$ is such that $\alpha(x) \neq 0,$ then we can show %\footnote{\textcolor{b}{At what point do we need to have $\dim \sing X = 4\ell-1$?}\red{I think when we use that $L_{4\ell}(\bbZ C_k)$ is non-trivial and then map it to  $L_{4\ell}(\res X)\equiv L_{(4\ell-1)+1}(\res X)=L_{(\dim \sing X)+1}(\res X)$}}
that $x$ acts non-trivially on $\tS_{\BQ}(\sing X).$

Indeed, let 
$\rho: \tS_{\BQ}(\sing X) \lra K_{\dim \sing X + 1}(D^* (\cov{\sing X})^\Gamma)$
be the $\rho$-map from \S\ref{sec:RhoMap} and let $\rho_{\Gamma}$ be the composition with the natural map induced by the classifying map
\begin{equation*}
	\rho_{\Gamma}:
	\tS_{\BQ}(\sing X) \lra 
	K_{0}(D^*_{\Gamma}).
\end{equation*}
In \cite{piazza-zenobi-add} it is shown that $\rho_{\Gamma}(\iota)=0,$ where $\iota$ denotes the class in $\tS_{\BQ}(\sing X)$ represented by the identity map (the context in  \cite{piazza-zenobi-add} is that of differentiable or
topological manifold, but it is easy to see that the same arguments establish the more general statement given here).
On the other hand, Benameur-Roy in \cite{BR},  have defined a homomorphism
\begin{equation*}
	\beta_{{\rm CG}}: K_0(D^*_{\Gamma}) \lra \bbR
\end{equation*}
for which our main result, together with Corollary 3.20 in  \cite{BR}, implies the following.

\begin{proposition} The map $\rho_{{\rm CG}}= \beta_{{\rm CG}}\circ \rho_{\Gamma}: \tS_{\BQ}(\sing X)\to \bbR$ satisfies
\begin{equation*}
	\rho_{{\rm CG}} (\iota)=0 \quad\text{and}\quad \rho_{{\rm CG}} (\pa(x)(\iota))=\alpha (x)
\end{equation*}
where $\pa(x)(\iota)$ denotes the action of 
$\tL_{\BQ}(\sing X \times I)$ on $\tS_{\BQ}(\sing X)$ defined in Corollary \ref{cor:LAct}.

In particular, if $\alpha(x)\neq 0,$ then $x$ acts non-trivially on $\iota.$
\end{proposition}

It is pointed out in \cite{Chang-Weinberger} that, since $\alpha$ is a homomorphism into $\bbR,$ the existence of a non-zero element in its range implies that its range has infinite cardinality. They also point out that, if $C_k$ is the cyclic group of order $k$ then, for any $\ell \in \bbN,$ the homomorphism $\alpha_k:L_{4\ell}(\bbZ C_k) \lra \bbR$ defined as in \eqref{eq:DefAlpha} has range of infinite cardinality. Thus the idea is to use these elements to find elements in the range of $\alpha.$

%Following {\em loc. cit.} we look for non-zero elements in the range of $\alpha$ from torsion elements in a fundamental group.

%Let us assume that
%\begin{quote}
%{\em the fundamental group of $\res X$ contains an element of order $k<\infty$}
%\end{quote}
%and\footnote{if we haven't already made this assumption} that the dimension of $\sing X$ is $4\ell-1,$ for some $\ell>1.$

Assume that the dimension of $\sing X$ is $4\ell-1,$ for some $\ell>1.$
Let $i: \res X \lra \sing X$ be the inclusion of the regular part, which we recall is a BQ-transverse map, and let $d^{\wt X}(4\ell)$ be the dimension function for $\sing X$ that is equal to $4\ell$ on the inverse image of the regular part of $\sing X,$ so that as in \S\ref{sec:LESLBQ} we have a homomorphism
\begin{equation*}
	 i_* : L_{4\ell} (\res X)=L_{4\ell} (\bbZ \pi_1 (\res X))
	 \lra \tL_{\BQ, d^{\res X}(4\ell)}(\sing X) 
	 =\tL_{\BQ, d^{\res X}(4\ell)}(\sing X\times I).
\end{equation*}
Note that we also have a homomorphism $\wt\alpha:\tL_{4\ell}(\bbZ\pi_1(\res X)) \lra \bbR$ defined as above.

\begin{proposition}\label{prop:comm-key}
If the map $i_*:\pi_1(\res X) \lra \pi_1(\sing X)$ is injective and
there is a monomorphism $p: C_k \lra \pi_1(\res X)$
 then the following diagram commutes
\begin{equation*}
	\xymatrix{
	\tL_{4\ell}(\bbZ C_k) \ar[rd]_-{\alpha_k} \ar[r]^-{p_*} &
	\tL_{4\ell}(\bbZ \pi_1(\res X)) \ar[r]^-{i_*} \ar[d]^{\wt\alpha} & 
	\tL_{\BQ, d^{\res X}(4\ell)}(\sing X \times I) \ar[ld]^-{\alpha} \\
	& \bbR & }
\end{equation*}

\end{proposition}

\begin{proof}
The commutativity of the left triangle is a classical result, 
already used by Chang-Weinberger. To see that the second triangle commutes, i.e., that $\wt \alpha = \alpha \circ i_*,$ it suffices in view of \eqref{eq:CommTraces} to show that whenever $F: \sing X_1 \lra \sing X_2$ is a BQ-transverse map we have a commutative diagram
\begin{equation}\label{comm-L-K}
	\xymatrix{
	\tL_{\BQ,d}(\sing X_1) \ar[r]^-{F_*} \ar[d]_-{\Ind_{{\rm APS}}} 
		& \tL_{\BQ,d}(\sing X_2) \ar[d]^-{\Ind_{{\rm APS}}} \\
	K_*(C^*\Gamma_1)[\frac{1}{2}] \ar[r]^-{f_*} & K_*(C^*\Gamma_2)[\frac{1}{2}] }
\end{equation}
where $\Gamma_j:=\pi_1 (\sing X_j)$ and $f: \Gamma_1 \lra \Gamma_2$ is the map induced by $F.$ (For our application $F$ is given by the inclusion $\res X \xlra i \sing X.$)\\
%\footnote{\red{But who is $f$ when  we have $\sing X_1=\res X$ and $\sing X_2=\sing X$ ? Are we considering
%$\res X$ as a subset of $\sing X$ or what ? Anyway, let us go on in general for
%the time being...} }

If $\gamma \in \tL_{\BQ,d}(\sing X_1)$ is represented by 
\begin{equation*}
	(\sing M; \sing X_1, \sing X_1^\prime) \xrightarrow{(\phi;\id, \psi)}
	(\sing X_1 \times [0,1]; \sing X_1 \times \{0\}, \sing X_1 \times \{1\}) \xlra{\id}
	\sing X_1 \times I 
\end{equation*}
then $\Ind_{{\rm APS}}([\gamma])$ is equal to $\Ind(\eth_{\sign}^{\cG(r)})$ where
$\eth_{\sign}$ is the signature operator on $\sing Z = \sing M \sqcup (\sing X_1 \times I)$ and, if $R:\sing X_1 \times [0,1] \lra B\Gamma_1$ is the classifying map of the universal cover of $\sing X_1 \times [0,1],$ then 
\begin{equation*}
	r: \sing Z \lra B\Gamma_1 \text{ is given by }
	R\circ \phi \sqcup R
\end{equation*}
Notice that it is implicit the choice of a mezzoperversity and the statement that the index class is in fact independent
of this choice.
It follows that
\begin{equation*}
	\Ind_{{\rm APS}}(F_*[\gamma]) = \Ind(D^{\cG((Bf)\circ r)})
\end{equation*}
and by Proposition \ref{prop:IndexInduction} this is equal to
$f_*(\Ind(D^{\sG(r)}) )\in K(C^*_r\Gamma_2)[\frac{1}{2}],$ as required.\end{proof}

%\noindent
%\red{{\bf Remark:}} we would like to further map \eqref{comm-L-K} to $\bbR$ by using $\tau_\Gamma-\tau$
%on $K_*(C^*\Gamma_1)$ and on $ K_*(C^*\Gamma_2)$. As I said, my feeling is that commutativity does not hold
%if we don't assume that $f$ (or maybe $f_*$) is injective. To see this point, assume that $\sing X_1$ and
%$\sing X_2$ are Witt and of depth 1. As explained in Remark \ref{BROY} below, the composition of $\Ind_{APS,\Gamma_1}$
%with $\tau_\Gamma-\tau$, applied to an element $\gamma\in \tL_{\BQ,d}(\sing X_1)$, is equal to a  Cheeger-Gromov rho invariant on a $\Gamma_1$-covering of a certain closed manifold.
%By the above argument the image of the element $F_*(\gamma)\in \tL_{\BQ,d}(\sing X_2)$
%under  $(\tau_\Gamma-\tau)\circ  \Ind_{APS,\Gamma_2}$
%is equal to the Cheeger-Gromov rho invariant of a $\Gamma_2$ coverings of the same manifold; I see no reasons
%why these two rho-invariants should be the same, unless the argument in my paper with Thomas, which requires that $f:\Gamma_1\to \Gamma_2$ be an injection, can be applied.\\
%\red{{\bf End of Remark.}} 

An immediate consequence of the discussion above is the infinite cardinality of the structure set.

\begin{corollary}
Let $\sing X$ be a Cheeger space of dimension $4\ell-1,$ $\ell>1,$ such that $\pi_1 (\res X)$ has an element of finite order and 
$i_*: \pi_1(\res X)\lra \pi_1(\sing X)$ is injective.
There exist elements $x_j\in  \tL_{\BQ, d}(\sing X\times I)$, $j\in\bbZ$, such that $\alpha(x_i)\not= \alpha(x_j)$
for $i\not= j$. Consequently, the elements $\partial(x_j)(\iota)$ are all distinct in $\tS_{\BQ}(\sing X)$. In particular
$$ | \tS_{\BQ}(\sing X) | = \infty$$
\end{corollary}

\begin{remark}\label{BROY}
Write  $\partial(x_i)(\iota)=[\sing M _j \xrightarrow{f_i} \sing X] \in \tS_{\BQ}(\sing X)$ with $f_i$
a transverse stratified homotopy equivalence. 
If $\sing X$ is a Witt space of depth one, then we claim that $\sing M _i$ is not stratified diffeomorphic to $\sing M_k$ for $i\not= k$.
Indeed, let 
$$x_i=[ (\sing W_i ;\sing X,\sing M_i)\xrightarrow{(\phi;\id,f_i)} (\sing X\times [0,1];
\sing X\times \{0\}, \sing X\times \{1\})\xrightarrow{\id}\sing X\times [0,1]]
$$
and consider $\alpha (x_i)$. This is the difference of two numbers: one is the Von Neumann-index 
on the total space of a Galois $\Gamma$-covering with boundary and the other is the usual index on  the
base of such Galois covering.
We point out that the operators we are considering are invertible on the boundary because they have been 
perturbed by the Hilsum-Skandalis perturbation.
We now write the APS-index formula, upstairs and downstairs, following \cite{Piazza-Vertman}, and take the difference; we find ourselves with the Cheeger-Gromov rho invariant of the signature operator of $\sing M_i\sqcup (-\sing X)$
pertubed by the Hilsum-Skandalis perturbation associated to  $f_i$. Now we proceed as in \cite[Section 10]{PS1}, taking
an $\epsilon$-concentrated Hilsum-Skandalis perturbation and letting $\epsilon\downarrow 0$.
We then obtain, finally, that $\alpha( x_i)$ is equal to the difference of the Cheeger-Gromov rho-invariants
of $\sing M_i$ and $\sing X$. Since $\alpha( x_i)\not= \alpha(x_k)$ for $i\not= k$ we can conclude 
that the Cheeger-Gromov rho-invariants of $\sing M _i $
and  $\sing M_k$ are indeed different,  and the statement follows from the stratified diffeomorphism invariance
of the Cheeger-Gromov rho invariant on Witt spaces of depth 1, established in \cite{Piazza-Vertman}.  
\end{remark}

\begin{remark}
If $\sing X$ has depth one we can use van Kampen's theorem to see that the map $\pi_1(\res X) \lra \pi_1(\sing X)$ induced by inclusion is an isomorphism if and only if the link of the singular stratum of $\sing X$ is simply connected (i.e., $\sing X$ is supernormal). 
%\red{But for supernormal spaces of depth 1, say Witt, would our result, viz.  $ | \tS_{\BQ}(\sing X) | = \infty$,
%follow from the paper of Cappell-Weinberger in CRAS or its restatement in Weinberger's book ?}
%\blue{Where is this in Weinberger's book? It's hard to say since they don't say much in [Cappell-Weinberger]. I think some work would be needed, e.g., to justify that $\alpha$ vanishes on the image of the normal invariants.
%\red{First of all, I think Cappell-Weinberger need, in addition to supernormality, strata of even codimension. What I meant is  that maybe this result
%of ours is implicit in the Corollary page 213 of W. book or  Th\'eoreme 2 in \cite{CW-CRAS}.} 
%}
%
For $\sing X$ of arbitrary depth, using Remark \ref{rmk:SupNormal} in the setting of simply connected links lets us argue as above to see that torsion in the fundamental group of any `closed stratum' $\sing Y$ forces $|\tS_{\BQ}(\sing X) | = \infty.$ Indeed note that the condition on fundamental groups is superseded by the injective map 
\begin{equation*}
	\tL(\bbZ\pi_1(\sing Y))[\tfrac12] \lra \tL_{\BQ}(\sing X)[\tfrac12].
\end{equation*}

%
%
%The injectivity assumption on 
%%
%\begin{equation*}
%	 i_* : L_{4\ell} (\res X)=L_{4\ell} (\bbZ \pi_1 (\res X))
%	 \lra \tL_{\BQ, d^{\res X}(4\ell)}(\sing X) 
%	 =\tL_{\BQ, d^{\res X}(4\ell)}(\sing X\times I),
%\end{equation*}
%%
%has a simple replacement in the setting of simply connected links. Indeed, using Remark \ref{rmk:SupNormal}, there is an injective map
%%
%\begin{equation*}
%	\tL(\bbZ\pi_1(X^{\dagger})[\tfrac12] \lra \tL_{\BQ}(\sing X)[\tfrac12],
%\end{equation*}
%%
%where $X^{\dagger}$ is a stratum of maximal depth,
%so one could run the argument above with the assumption that $\pi_1X^{\dagger}$ has an element of finite order and conclude that the structure group has infinitely many elements.
\end{remark}

%\bibliographystyle{alpha}
%\bibliography{rho}

\begin{thebibliography}{ALMP17}

\bibitem[Alb17]{Albin:Hodge}
Pierre Albin.
\newblock On the {H}odge theory of stratified spaces.
\newblock In {\em Hodge theory and {$L^2$}-analysis}, volume~39 of {\em
  Advanced lectures in mathematics}, pages 1--78. Higher {E}ducation {P}ress
  and {I}nternational {P}ress, {B}eijing-{B}oston, 2017.

\bibitem[ABL{\etalchar{+}}15]{ABLMP}
Pierre Albin, Markus Banagl, Eric Leichtnam, Rafe Mazzeo, and Paolo Piazza.
\newblock Refined intersection homology on non-{W}itt spaces.
\newblock {\em J. Topol. Anal.}, 7(1):105--133, 2015.

\bibitem[AGR]{a-gr-dirac}
Pierre Albin and Jesse Gell-Redman.
\newblock The index formula for families of {D}irac type operators on
  pseudomanifolds.
\newblock forthcoming.

\bibitem[AGR16]{a-gr-spin}
Pierre Albin and Jesse Gell-Redman.
\newblock The index of {D}irac operators on incomplete edge spaces.
\newblock {\em SIGMA Symmetry Integrability Geom. Methods Appl.}, 12:Paper No.
  089, 45, 2016.


\bibitem[ALMP12]{ALMP:Witt}
Pierre Albin, \'Eric Leichtnam, Rafe Mazzeo, and Paolo Piazza.
\newblock The signature package on {W}itt spaces.
\newblock {\em Ann. Sci. \'Ec. Norm. Sup\'er. (4)}, 45(2):241--310, 2012.

\bibitem[ALMP16]{ALMP:Hodge}
Pierre Albin, \'Eric Leichtnam, Rafe Mazzeo, and Paolo Piazza.
\newblock Hodge theory on {C}heeger spaces.
\newblock published on-line in Journal f\"ur die reine und angewandte
  Mathematik (Crelle Journal). ISSN (Online) 1435-5345, ISSN (Print) 0075-4102,
  DOI: 10.1515/crelle-2015-0095, 2016.

\bibitem[ALMP17]{ALMP:Novikov}
Pierre Albin, Eric Leichtnam, Rafe Mazzeo, and Paolo Piazza.
\newblock The {N}ovikov conjecture on {C}heeger spaces.
\newblock {\em J. Noncommut. Geom.}, 11(2):451--506, 2017.

\bibitem[And77]{Anderson}
G.~A. Anderson.
\newblock {\em Surgery with coefficients}.
\newblock Lecture Notes in Mathematics, Vol. 591. Springer-Verlag, Berlin-New
  York, 1977.

\bibitem[Ban06]{banagl-annals}
Markus Banagl.
\newblock The {$L$}-class of non-{W}itt spaces.
\newblock {\em Ann. of Math. (2)}, 163(3):743--766, 2006.

\bibitem[BQ75]{BQ}
William Browder and Frank Quinn.
\newblock A surgery theory for {$G$}-manifolds and stratified sets.
\newblock In {\em Manifolds---{T}okyo 1973 ({P}roc. {I}nternat. {C}onf.,
  {T}okyo, 1973)}, pages 27--36. Univ. Tokyo Press, Tokyo, 1975.

\bibitem[BR15]{BR}
Moulay-Tahar Benameur and Indrava Roy.
\newblock The {H}igson-{R}oe exact sequence and {$\ell^2$} eta invariants.
\newblock {\em J. Funct. Anal.}, 268(4):974--1031, 2015.

\bibitem[Bro68]{Browder:TransG}
William Browder.
\newblock Surgery and the theory of differentiable transformation groups.
\newblock In {\em Proc. {C}onf. on {T}ransformation {G}roups ({N}ew {O}rleans,
  {L}a., 1967)}, pages 1--46. Springer, New York, 1968.

\bibitem[Bun95]{Bunke}
Ulrich Bunke.
\newblock A {$K$}-theoretic relative index theorem and {C}allias-type {D}irac
  operators.
\newblock {\em Math. Ann.}, 303(2):241--279, 1995.


\bibitem[CW91]{CW-CRAS}
Sylvain Cappell and Shmuel Weinberger.
\newblock Classification de certains espaces stratifi\'es.
\newblock {\em C. R. Acad. Sci. Paris S\'er. I Math.}, 313(6):399--401, 1991.

\bibitem[CW03]{Chang-Weinberger}
Stanley Chang and Shmuel Weinberger.
\newblock On invariants of {H}irzebruch and {C}heeger-{G}romov.
\newblock {\em Geom. Topol.}, 7:311--319, 2003.

\bibitem[CLM]{Crowley-Luck-Macko}
Diarmuid Crowley, Wolfgang L{\"u}ck, and Tibor Macko.
\newblock Surgery {T}heory: {F}oundations.
\newblock available online at \url{http://www.mat.savba.sk/~macko/}.

\bibitem[DP]{Davis-Petrosyan}
James~F. Davis and Nansen Petrosyan.
\newblock Manifolds and {P}oincar{\'e} complexes.
\newblock available online at
  \url{http://www.personal.soton.ac.uk/np3u12/dp_shillong.pdf}.

\bibitem[DR88]{Dovermann-Rothenberg}
Karl~Heinz Dovermann and Melvin Rothenberg.
\newblock Equivariant surgery and classifications of finite group actions on
  manifolds.
\newblock {\em Mem. Amer. Math. Soc.}, 71(379):viii+117, 1988.

\bibitem[DS90]{Dovermann-Schultz}
Karl~Heinz Dovermann and Reinhard Schultz.
\newblock {\em Equivariant surgery theories and their periodicity properties},
  volume 1443 of {\em Lecture Notes in Mathematics}.
\newblock Springer-Verlag, Berlin, 1990.

\bibitem[FM81]{Fulton-Mac}
William Fulton and Robert MacPherson.
\newblock Categorical framework for the study of singular spaces.
\newblock {\em Mem. Amer. Math. Soc.}, 31(243):vi+165, 1981.

\bibitem[GM83]{IH2}
Mark Goresky and Robert MacPherson.
\newblock Intersection homology. {II}.
\newblock {\em Invent. Math.}, 72(1):77--129, 1983.

\bibitem[HMM97]{HMM2}
Andrew Hassell, Rafe Mazzeo, and Richard~B. Melrose.
\newblock A signature formula for manifolds with corners of codimension two.
\newblock {\em Topology}, 36(5):1055--1075, 1997.

\bibitem[HR00]{higson-roe-book}
Nigel Higson and John Roe.
\newblock {\em Analytic {$K$}-homology}.
\newblock Oxford Mathematical Monographs. Oxford University Press, Oxford,
  2000.
\newblock Oxford Science Publications.

\bibitem[HR05a]{higson-roeI}
Nigel Higson and John Roe.
\newblock Mapping surgery to analysis. {I}. {A}nalytic signatures.
\newblock {\em $K$-Theory}, 33(4):277--299, 2005.

\bibitem[HR05b]{higson-roeII}
Nigel Higson and John Roe.
\newblock Mapping surgery to analysis. {II}. {G}eometric signatures.
\newblock {\em $K$-Theory}, 33(4):301--324, 2005.

\bibitem[HR05c]{higson-roeIII}
Nigel Higson and John Roe.
\newblock Mapping surgery to analysis. {III}. {E}xact sequences.
\newblock {\em $K$-Theory}, 33(4):325--346, 2005.

\bibitem[HR10]{higson-roe4}
Nigel Higson and John Roe.
\newblock {$K$}-homology, assembly and rigidity theorems for relative eta
  invariants.
\newblock {\em Pure Appl. Math. Q.}, 6(2, Special Issue: In honor of Michael
  Atiyah and Isadore Singer):555--601, 2010.

\bibitem[HS92]{HiSka}
Michel Hilsum and Georges Skandalis.
\newblock Invariance par homotopie de la signature \`a coefficients dans un
  fibr\'e presque plat.
\newblock {\em J. Reine Angew. Math.}, 423:73--99, 1992.

\bibitem[HW01]{Hughes-Weinberger}
Bruce Hughes and Shmuel Weinberger.
\newblock Surgery and stratified spaces.
\newblock In {\em Surveys on surgery theory, {V}ol. 2}, volume 149 of {\em Ann.
  of Math. Stud.}, pages 319--352. Princeton Univ. Press, Princeton, NJ, 2001.

\bibitem[L{\"u}c02]{lueck-survey}
Wolfgang L{\"u}ck.
\newblock A basic introduction to surgery theory.
\newblock In {\em Topology of high-dimensional manifolds, {N}o. 1, 2
  ({T}rieste, 2001)}, volume~9 of {\em ICTP Lect. Notes}, pages 1--224. Abdus
  Salam Int. Cent. Theoret. Phys., Trieste, 2002.

\bibitem[Mat12]{Mather}
John Mather.
\newblock Notes on topological stability.
\newblock {\em Bull. Amer. Math. Soc. (N.S.)}, 49(4):475--506, 2012.

\bibitem[Mel12]{MelroseSlides}
Richard Melrose.
\newblock Eta invariant on articulated manifolds.
\newblock Talk delivered at a conference at {CRM}, ``{S}pectral {I}nvariants on
  {N}on-compact and {S}ingular {S}paces". Slides available at
  \url{http://www-math.mit.edu/~rbm/CRM2012.pdf}, 2012.

\bibitem[Mil13]{Miller}
David~A. Miller.
\newblock Strongly stratified homotopy theory.
\newblock {\em Trans. Amer. Math. Soc.}, 365(9):4933--4962, 2013.

\bibitem[Mil61]{Milnor:Kill}
John Milnor.
\newblock A procedure for killing homotopy groups of differentiable manifolds.
\newblock In {\em Proc. {S}ympos. {P}ure {M}ath., {V}ol. {III}}, pages 39--55.
  American Mathematical Society, Providence, R.I, 1961.


\bibitem[Pfl01]{Pflaum}
Markus~J. Pflaum.
\newblock {\em Analytic and geometric study of stratified spaces}, volume 1768
  of {\em Lecture Notes in Mathematics}.
\newblock Springer-Verlag, Berlin, 2001.

\bibitem[PS07a]{PS1}
Paolo Piazza and Thomas Schick.
\newblock Bordism, rho-invariants and the {B}aum-{C}onnes conjecture.
\newblock {\em J. Noncommut. Geom.}, 1(1):27--111, 2007.

\bibitem[PS07b]{PS2}
Paolo Piazza and Thomas Schick.
\newblock Groups with torsion, bordism and rho invariants.
\newblock {\em Pacific J. Math.}, 232(2):355--378, 2007.

\bibitem[PS14]{PS-Stolz}
Paolo Piazza and Thomas Schick.
\newblock Rho-classes, index theory and {S}tolz' positive scalar curvature
  sequence.
\newblock {\em J. Topol.}, 7(4):965--1004, 2014.

\bibitem[PS16]{PS-akt}
Paolo Piazza and Thomas Schick.
\newblock The surgery exact sequence, {K}-theory and the signature operator.
\newblock {\em Ann. K-Theory}, 1(2):109--154, 2016.

\bibitem[PV]{Piazza-Vertman}
Paolo Piazza and Boris Vertman.
\newblock Eta and rho invariants on manifolds with edges.
\newblock available online at arXiv:1604.07420.

\bibitem[PZ16]{piazza-zenobi-add}
Paolo Piazza and Vito~Felice Zenobi.
\newblock Additivity of the rho map on the topological structure group.
\newblock available online at arXiv:1607.07075, 2016.

\bibitem[Qui70a]{Quinn}
Frank Quinn.
\newblock A geometric formulation of surgery.
\newblock In {\em Topology of {M}anifolds ({P}roc. {I}nst., {U}niv. of
  {G}eorgia, {A}thens, {G}a., 1969)}, pages 500--511. Markham, Chicago, Ill.,
  1970.

\bibitem[Qui70b]{Quinn:Thesis}
Frank~Stringfellow Quinn, III.
\newblock {\em A {GEOMETRIC} {FORMULATION} {OF} {SURGERY}}.
\newblock ProQuest LLC, Ann Arbor, MI, 1970.
\newblock Thesis (Ph.D.)--Princeton University.

\bibitem[Ran02a]{ranicki-ags}
Andrew Ranicki.
\newblock {\em Algebraic and geometric surgery}.
\newblock Oxford Mathematical Monographs. The Clarendon Press Oxford University
  Press, Oxford, 2002.
\newblock Oxford Science Publications.

\bibitem[Ran02b]{Ranicki}
Andrew Ranicki.
\newblock {\em Algebraic and geometric surgery}.
\newblock Oxford Mathematical Monographs. The Clarendon Press, Oxford
  University Press, Oxford, 2002.
\newblock Oxford Science Publications.

\bibitem[Roe87]{roe-foliation}
John Roe.
\newblock Finite propagation speed and {C}onnes' foliation algebra.
\newblock {\em Math. Proc. Cambridge Philos. Soc.}, 102(3):459--466, 1987.

\bibitem[Sie83]{Siegel:Witt}
P.~H. Siegel.
\newblock Witt spaces: a geometric cycle theory for {$K{\rm O}$}-homology at
  odd primes.
\newblock {\em Amer. J. Math.}, 105(5):1067--1105, 1983.

\bibitem[Tho69]{Thom:Ensembles}
R.~Thom.
\newblock Ensembles et morphismes stratifi\'es.
\newblock {\em Bull. Amer. Math. Soc.}, 75:240--284, 1969.

\bibitem[Ver84]{Verona}
Andrei Verona.
\newblock {\em Stratified mappings---structure and triangulability}, volume
  1102 of {\em Lecture Notes in Mathematics}.
\newblock Springer-Verlag, Berlin, 1984.

\bibitem[Wah13]{Wahl_higher_rho}
Charlotte Wahl.
\newblock Higher {$\rho$}-invariants and the surgery structure set.
\newblock {\em J. Topol.}, 6(1):154--192, 2013.

\bibitem[Wal99]{Wall}
C.~T.~C. Wall.
\newblock {\em Surgery on compact manifolds}, volume~69 of {\em Mathematical
  Surveys and Monographs}.
\newblock American Mathematical Society, Providence, RI, second edition, 1999.
\newblock Edited and with a foreword by A. A. Ranicki.

\bibitem[Wal16]{Wall:DiffTop}
C.~T.~C. Wall.
\newblock {\em Differential topology}, volume 156 of {\em Cambridge Studies in
  Advanced Mathematics}.
\newblock Cambridge University Press, Cambridge, 2016.

\bibitem[Wei94]{W-book}
Shmuel Weinberger.
\newblock {\em The topological classification of stratified spaces}.
\newblock Chicago Lectures in Mathematics. University of Chicago Press,
  Chicago, IL, 1994.


\bibitem[WXY16]{Weinberger-Xie-Yu}
Shmuel Weinberger, Zhizhang Xie, and Guoliang Yu.
\newblock Additivity of higher rho invariants and nonrigidity of topological
  manifolds.
\newblock available online at arXiv:1608.03661, 2016.

\bibitem[WY15]{WY-finite}
Shmuel Weinberger and Guoliang Yu.
\newblock Finite part of operator {$K$}-theory for groups finitely embeddable
  into {H}ilbert space and the degree of nonrigidity of manifolds.
\newblock {\em Geom. Topol.}, 19(5):2767--2799, 2015.

\bibitem[Whi65]{Whitney:Tangents}
Hassler Whitney.
\newblock Tangents to an analytic variety.
\newblock {\em Ann. of Math. (2)}, 81:496--549, 1965.

\bibitem[Zen17]{Zenobi:Mapping}
Vito~Felice Zenobi.
\newblock Mapping the surgery exact sequence for topological manifolds to
  analysis.
\newblock {\em J. Topol. Anal.}, 9(2):329--361, 2017.

\end{thebibliography}

\newcommand{\etalchar}[1]{$^{#1}$}

\end{document}